\documentclass[10pt, a4paper, reqno]{amsart}
\usepackage{mathrsfs}
\usepackage{amsfonts}
\usepackage{amsmath}
\usepackage{graphicx}
\usepackage{amssymb}
\usepackage{amsthm}
\usepackage{color}
\usepackage{vmargin}

\long\def\symbolfootnote[#1]#2{\begingroup%
\def\thefootnote{\fnsymbol{footnote}}\footnote[#1]{#2}\endgroup}

\setmarginsrb{20mm}{20mm}{20mm}{20mm}{10mm}{10mm}{10mm}{10mm}

\newtheoremstyle{remark}
  {}{}{}{}{\bfseries}{.}{.5em}{{\thmname{#1 }}{\thmnumber{#2}}{\thmnote{ (#3)}}}

\DeclareMathOperator\dist{dist}

\RequirePackage{amsthm}
\newtheorem*{theorem*}{Theorem}
\newtheorem{theorem}{Theorem}[section]
\newtheorem{lem}[theorem]{Lemma}
\newtheorem{cor}[theorem]{Corollary}
\newtheorem{prop}[theorem]{Proposition}
\theoremstyle{definition}
\newtheorem{defn}[theorem]{Definition}
\theoremstyle{remark}
\newtheorem{rem}[theorem]{Remark}

\def\uz{{u_{z}}}
\def\ub{{u_{\overline{z}}}}
\def\epuz{{u^{\epsilon}_{z}}}
\def\epub{{u^{\epsilon}_{\overline{z}}}}
\renewcommand{\bar}[1]{\overline{#1}} 
\def\baruz{\overline{{u_{z}}}}
\def\ub{{u_{\overline{z}}}}
\def\barub{\overline{{u_{\overline{z}}}}}
\def\zbar{{\overline{z}}}

\def\utet{{u_{\theta}}}
\def\utetb{{u_{\overline{\theta}}}}
\def\ubtet{{{\overline u}_{\theta}}}
\def\ubtetb{{{\overline u}_{\overline{\theta}}}}
\def\tbar{{\overline{\theta}}}

\def\tilu{\tilde{u}}
\def\tilU{\tilde{U}}
\def\utt{{u_{\theta\theta}}}
\def\uttb{{u_{\theta \tbar}}}
\def\utbt{{u_{\tbar \theta}}}
\def\utbtb{{u_{\tbar \tbar}}}

\def\ups{{u^{\epsilon}}}
\def\dups{{Du^{\epsilon}}}
\def\vps{{v^{\epsilon}}}

\def\zinv{{z^{{\rm inv}}}}

\def\vint{\mathop{\mathchoice%
          {\setbox0\hbox{$\displaystyle\intop$}\kern 0.22\wd0%
           \vcenter{\hrule width 0.6\wd0}\kern -0.82\wd0}%
          {\setbox0\hbox{$\textstyle\intop$}\kern 0.2\wd0%
           \vcenter{\hrule width 0.6\wd0}\kern -0.8\wd0}%
          {\setbox0\hbox{$\scriptstyle\intop$}\kern 0.2\wd0%
           \vcenter{\hrule width 0.6\wd0}\kern -0.8\wd0}%
          {\setbox0\hbox{$\scriptscriptstyle\intop$}\kern 0.2\wd0%
           \vcenter{\hrule width 0.6\wd0}\kern -0.8\wd0}}%
          \mathopen{}\int}
\def\bdy {{\partial}}

\newcommand{\Sob}{{W}^{1,p}}

\newcommand{\Sobzero}{{W}^{1,p}_{0}}
\newcommand{\Om}{\Omega}
\newcommand{\C}{\mathbb{C}}
\newcommand{\R}{\mathbb{R}}

\newcommand{\ep}{\epsilon}

\newcommand{\diam}{{\rm diam}}
\newcommand{\bz}{\bar{z}}

\newcommand{\covers}{C}

\newcommand{\rr}{\mathbb{R}}
\newcommand{\cc}{\mathbb{C}}

\newcommand{\bt}{\bar{\theta}}
\newcommand{\ut}{u_\theta}
\def\utb{{u_{\overline{\theta}}}}
\newcommand{\asc}{\mathcal{A}}
\newcommand{\bsc}{\mathcal{B}}
\newcommand{\lsc}{\mathcal{L}}

\renewcommand{\Re}{\operatorname{Re}}
\renewcommand{\Im}{\operatorname{Im}}

\definecolor{blau}{rgb}{0.1,0.0,0.9}
\definecolor{violet}{rgb}{0.54, 0.17, 0.89}
\newcommand{\blue}{\color{blau}}

\newcommand{\kom}[1]{}
\renewcommand{\kom}[1]{{\bf \blue /#1/}}

\newcounter{komcounter}
\numberwithin{komcounter}{section}

\begin{document}

\frenchspacing
\allowdisplaybreaks

\pagestyle{headings}
\title[The Rad\'o--Kneser--Choquet theorem for  $p$-harmonic mappings
between Riemannian surfaces]{The Rad\'o--Kneser--Choquet theorem for \\ $p$-harmonic mappings
between Riemannian surfaces}

\author[Tomasz Adamowicz]{Tomasz Adamowicz{\small$^1$}}
\address{Institute of Mathematics, Polish Academy of Sciences,
\'Sniadeckich 8, Warsaw, 00-656, Poland\/}
\email{tadamowi@impan.pl}

\author[Jarmo J\"a\"askel\"ainen]{Jarmo  J\"a\"askel\"ainen{\small$^2$}}
\address{Department of Mathematics and Statistics, University of Jyv\"askyl\"a, P. O. Box 35, FI-40014 University of Jyv\"askyl\"a; Department of Mathematics and Statistics, University of Helsinki}
\email{jarmo.t.jaaskelainen@jyu.fi}

\author[Aleksis Koski]{Aleksis Koski\small{$^3$}}
\address{Department of Mathematics and Statistics, University of Jyv\"askyl\"a, P. O. Box 35, FI-40014 University of Jyv\"askyl\"a}
\email{aleksis.t.koski@jyu.fi}

\date{\today}
\maketitle

\footnotetext[1]{T. Adamowicz was supported by a grant of National Science Center, Poland (NCN),
UMO-2013/09/D/ST1/03681.}
\footnotetext[2]{J. J\"a\"askel\"ainen was supported by the Academy of Finland (318636 and 276233).}
\footnotetext[3]{A. Koski  was supported by the V\"ais\"al\"a Foundation and the ERC Starting Grant number 307023.}

\noindent{\small
{\bf Abstract}. In the planar setting the Rad\'o--Kneser--Choquet theorem states that a harmonic map from the unit disk onto a Jordan domain bounded by a convex curve is a diffeomorphism provided that the boundary mapping is a homeomorphism. We prove the injectivity criterion of Rad\'o--Kneser--Choquet for $p$-harmonic mappings between Riemannian surfaces.

In our proof of the injectivity criterion we approximate the $p$-harmonic map with auxiliary mappings that solve uniformly elliptic systems. We prove that each auxiliary mapping has a positive Jacobian by a homotopy argument. We keep the maps injective all the way through the homotopy with the help of the minimum principle for a certain subharmonic expression that is related to the Jacobian.
}



\tableofcontents

\section{Introduction}\label{sec-Intro}

\noindent Finding injective solutions to a given system of PDEs is one of the most profound and fundamental questions in the geometric function theory.  In the nonlinear elasticity this and related questions about the Jacobian determinants
have been investigated, for instance, by Antman \cite{antman}, Ball~\cite{ball2, ball1, ballopen}, Ciarlet \cite{ciarlet}. The main focus of our work is devoted to the injectivity criterion known in the Euclidean setting of $\R^2$ as the Rad\'o--Kneser--Choquet theorem (the RKC theorem, for short), see e.g. the presentation in Duren~\cite[Chapter 3]{dur}:  \emph{A harmonic mapping from the unit disk onto a Jordan domain bounded by a convex curve is a diffeomorphism provided that the boundary mapping is a homeomorphism}.

This remarkable result fails to hold in the higher Euclidean dimensions, see Laugesen's counterexample for spheres~\cite{lau}. However, counterparts of the RKC theorem are known to hold in various other settings: for harmonic mappings between two compact Riemannian surfaces where the target surface has non-negative curvature (Schoen--Yau~\cite{sy}), without the curvature assumptions but under the assumption on the smallness of the image set (Jost~\cite{jost81}), see also Kalaj~\cite{kal} for some further generalizations.

Another direction of extending the RKC theorem is to consider the setting of planar domains but for more general mappings, see Alessandrini--Nesi~\cite{alne1, alne2}, Alessandrini--Sigalotti~\cite{al-si}, Bauman--Phillips~\cite{baup} and Iwaniec--Onninen~\cite{iwon2}. In particular, the case of the (isotropic) $p$-harmonic mappings in the plane was investigated in~\cite{iko}.

New important applications of the $p$-harmonic RKC theorem, for both $p = 2$ and $p \in (1, \infty)$, have been seen in the diffeomorphic approximation of $W^{1,p}$-Sobolev homeomorphisms, where the local $p$-harmonic replacements are the necessary building blocks, see Iwaniec--Kovalev--Onninen~\cite{IKO11, IKO12a, IO13}, and in the approximation of  $W^{1,p}$-monotone mappings by monotone homeomorphisms, see Iwaniec--Onninen~\cite{iwon2}.

\smallskip

The main goal of this work is to prove the RKC theorem for the $p$-harmonic mappings between  Riemannian surfaces. We will now describe the details of our setting and the main assumptions on surfaces and mappings.

When studying deformation of some surface $M$ under a map $u: M \to N'$ (where $N'$ denotes a target manifold), one often assumes that $M$ responds to the energy-minimal deformation likewise in all directions. In other words, the energy functional $\mathbb{E}[u]$ assumes the form
\begin{equation}\label{energy}
\mathbb{E}[u] = \int_{M} E(|Du|^2)dV_M.
\end{equation}
In this article we focus on the study of minimizers for the $p$-harmonic-type energy, i.e.  $E(t) = (\ep^2+t)^{\frac{p}{2}}$ for $2 \leq p < \infty$ and $\ep\geq 0$ with the degenerate elliptic case $\ep=0$ being our central target of investigation. We reserve the notion of a \emph{$p$-harmonic mapping} to exclusively mean a minimizer of the $p$-harmonic energy among Sobolev mappings with given boundary data.
Naturally, these minimizers also solve the associated Euler-Lagrange system of equations. However, unlike in the case of the $p$-harmonic mappings in the plane, the appearance of the curvature makes these equations nonhomogeneous. The nonhomogeneity depends on the second fundamental form of the target surface. This discrepancy leads to several difficulties which one needs to handle.

 Let $M$ be a Riemannian surface with $C^2$-boundary and $N'$ be a compact Riemannian surface without boundary. We equip $M$ and $N'$ with conformal metrics $\sigma$ and $\rho$, respectively. Moreover, we assume that $\sigma$ and $\rho$ are bounded from below and above and we make the following additional assumptions.
\begin{enumerate}
\item[(A1)] The Gaussian curvature of $M$ is nonpositive.
\item[(A2)] Suppose $N \subset N'$ is a geodesically convex submanifold with boundary, and $\partial N$ is assumed to be $C^{1,\alpha}$-regular. Assume also that it satisfies the following smallness condition: $N$ is contained in a small geodesic ball:
\begin{equation}\label{cond-small}\tag{S}
N \subset B(P_0, r_{N'\!, p}), \ \ P_0 \in N', r_{N'\!, p} > 0,
\end{equation}
where by a geodesic ball $B(p,r)$ we mean a ball of radius $r$ centered at $p$ in the metric induced by $\rho$ on the Riemannian manifold $N'$. This condition may be written equivalently in terms of the diameter of $N$ as
\begin{equation*}
\diam(N) < \epsilon_{N',p}.
\end{equation*}
This condition also implies that $N$ is covered by a single coordinate chart. 
\item[(A3)] The Gaussian curvature of $N'$ is nonnegative on $N$.
\end{enumerate}

\begin{theorem}[The RKC theorem for $p$-harmonic mappings between Riemannian surfaces]\label{MainThm}
Suppose that assumptions (A1-3) hold. Denote by $\phi_0 : \bdy M \to \bdy N$ a $C^{1,\alpha}$-homeomorphism between boundaries with non-vanishing tangential derivative.

Let $u : M \to N'$ denote a minimizer of the $p$-harmonic energy with boundary data $u = \phi_0$ on $\bdy M$, and assume that $u$ belongs to $C^{1,\alpha}(M, N')$. Then $u$ is in fact a $C^{1, \alpha}$-diffeomorphism from $M$ to $N$ with nonvanishing Jacobian in $M$. In particular, $u\in C^\infty(M)$.
\end{theorem}

A key aspect in our investigations is the problem of Jacobian estimates: In the geometric mapping theory and the nonlinear elasticity (e.g. Ball \cite{ball1, ballopen}) one often needs to control the Jacobian determinant $J(z, u) = \det(Du(z))$ from below, for example by means of a suitable minimum principle. For the uncoupled system of $2$-Laplace equations the Jacobian determinant of a solution satisfies
the minimum principle in any region where it is positive, which is a straightforward consequence of the superharmonicity of $z \mapsto \log J(z, u)$. When we deal with more general energy integrals it is not obvious which differential expressions work for Jacobian estimates. This has raised some recent interest for superharmonic expressions, see~\cite{iko},
 \cite{koski} and Kalaj~\cite{kalaj}. The proof of our analogue of the Rad\'o-Kneser-Choquet  theorem is based on the following subharmonicity result.

\begin{theorem}\label{Ttheorem}
Assume (A1) and (A3). Let $\ups: M \to N'$ be a stationary solution of the Euler--Lagrange system of equations associated with energy functional \eqref{energy} with $E(|D\ups|^2) = \left(\epsilon^2 + |D\ups|^2\right)^{p/2}$ for $2 \leq p < \infty$ and $\epsilon \geq 0$. Moreover, suppose that the Jacobian of $\ups$ satisfies $J_{\ups} = \tfrac{\rho(\ups(z))}{\sigma(z)}(|u^{\epsilon}_z|^2 - |u^{\epsilon}_{\zbar}|^2)> 0$ on $M$.  Define the differential expression $T$ on $M$ by the formula
\[
T  = E'\left(\tfrac{\rho(\ups(z))}{\sigma(z)}(|u^{\epsilon}_z|^2 - |u^{\epsilon}_{\zbar}|^2)\right)\; \frac{\rho(\ups(z))}{\sigma(z)}(|u^{\epsilon}_z|^2 - |u^{\epsilon}_{\zbar}|^2).
\]
Then, there exists a sufficiently large exponent $N_E > 0$ such that the function $-T^{-N_E}$ is a supersolution of the Beltrami--Laplace equation on $M$. In the particular case of the $p$-harmonic energy it holds that $N_E=N(p)$.
\end{theorem}

 In our paper we apply this observation only to $p$-harmonic type mappings (including the case of $\ep>0$), but the technique of the proof of Theorem~\ref{Ttheorem} is robust enough to encompass wider class of isotropic energies.

Since supersolutions of the Beltrami--Laplace equation satisfy the minimum principle, we obtain the following observation.

\begin{cor} \label{MinPrin}
Under the assumptions of Theorem \ref{Ttheorem}, the expression $T$ satisfies the minimum principle in $M$, meaning that for any compact subset $K \subset M$ we have
\begin{equation*}
\inf_{z\in K} T(z) \geq \inf_{z \in \partial K} T(z).
\end{equation*}
\end{cor}

Our aim is to make use of the above result to prove, in the end, that the Jacobian of $u$ does not vanish in $M$. Basic topology then implies that $u : M \to N$ is a smooth homeomorphism. In fact, $u$ is $C^{1, \alpha}$-diffeomorphism from $M$ to $N$, by the inverse function theorem. However, since Theorem~\ref{Ttheorem} is only applicable to mappings already having a positive Jacobian, we will have to find a nondirect way to apply it.

\smallskip

Let us now briefly discuss our assumptions and describe the structure of the paper. Section~\ref{sec-basic} is devoted to recalling some elementary facts in differential geometry needed largely in Section~\ref{sec-subharm} for the proof of the subharmonicity result, Theorem~\ref{Ttheorem}. In Section~\ref{sect-map} we set up the stage for further discussion: we recall the notion of Sobolev spaces and $p$-harmonic mappings on Riemannian surfaces. Furthermore, we derive the key Euler--Lagrange system of equations (Definition~\ref{phm-weak}).

We note that $p$-harmonic mappings automatically belong to $C^{1,\alpha}_{loc}(\mathrm{int}\,M,N')$ with $\alpha$ depending on $p$ and the geometry of $M$ and $N'$, \cite{HL}. Similar regularity results holding up to the boundary of the domain are much less prominent in the literature, and this is the reason for our a priori regularity assumption in our RKC theorem (Theorem~\ref{MainThm}) on the $p$-harmonic map, i.e., $C^{1,\alpha}(M, N')$, though the restriction to this regularity class might be superfluous. This assumption already guarantees the $C^{1, \alpha}$-smoothness of the boundary $\bdy M$. For further discussion of the $C^{1,\alpha}$-regularity of $p$-harmonic mappings, including the minima of the uniformly elliptic $p$-harmonic systems, we refer to Section~\ref{sect-systems}.

The fact that we assume in (A2) that the target domain $N$ is convex is absolutely necessary, which is already seen from the flat case due to the failure of the classical RKC theorem for non-convex targets, see Chapter 3.1 in Duren~\cite{dur}.

The smallness assumption (S) on the target $N$ in (A2) is also necessary to guarantee uniqueness for the $p$-harmonic and related Dirichlet problems. For a simple example on the necessity of this assumption merely notice that if $M = \{(x,y,z) \in S^2, z \geq 0\}$ and $N' = S^2$, then both the identity and reflection map $(x,y,z) \mapsto (x,y,-z)$ from $M$ to $S^2$ are minimizers of the Dirichlet energy with the same boundary values.  Let us mention, that in some important cases of manifolds, one can compute the bounds for the uniqueness radius in \eqref{cond-small}. For example, if $\kappa$ denotes an upper bound of the sectional curvature of $N'$, then Fuchs~\cite{fuchs} showed that $r_{N',p}<\frac{\pi}{2\sqrt{\kappa}}$.

We would like to emphasize that unlike some of the literature that study the so-called small solutions, see \cite{hildetal, jagermeister, jost81, hil, fuchs, fare}, we do not make any extra a priori assumptions on the size of the image set for mappings in the class in which we minimize the $p$-harmonic energy. Rather, we are able to conclude that any energy minimizer with given boundary data $\phi_0$ must satisfy such a smallness property via the maximum principles obtained in Section~\ref{sect-geometry}.

The smallness assumption (S) is needed, in addition, to establish the uniqueness (Appendix A), also for the Caccioppoli-type estimates in Section \ref{sect9} and for the maximum principle, see Appendix B.

In Section~\ref{sect-geometry} we also show the positivity of the Jacobian along the boundary. Here we use $C^2$-smoothness of $\bdy M$ to guarantee the interior sphere condition. Furthermore, the regularity of $\bdy M$ guarantees the existence of a $p$-harmonic mapping with given boundary values (Section 7). The positivity of the Jacobian on the boundary also lets us reduce the problem of proving the RKC theorem to the case where $M$ and $N$ have smooth boundaries and the boundary data is a $C^\infty$-smooth diffeomorphism (see also Section~\ref{subs-4.3}).

After the reduction to the smooth case, we study auxiliary mappings $\ups$ that approach our mapping $u$ in the limit (the existence and the uniqueness of these maps is discussed in Section 7, Appendix A). Mappings $\ups$ solve the uniformly elliptic systems and each such map  is shown to have a positive Jacobian. In the flat case, one can obtain this result as a direct consequence of Bauman--Marini--Nesi~\cite{bmn}, but we have to go through a different route and our proof is based on a homotopy argument along the lines of Iwaniec--Onninen~\cite{iwon} and Jost~\cite{jost81}, see Theorem~\ref{thm-jacob}.

Our homotopy argument requires uniform H\"older and $C^{1,\alpha}$-estimates up to the boundary, which we show in Section~\ref{sect-homotopy}.  The positivity of Jacobians follows by the minimum principle (Corollary \ref{MinPrin}). For the minimum principle we need the Gauss curvature assumptions (A1) and (A3).

The remaining major step to be made is then completed in Section \ref{sect9}. There, we show the  $L^p$-convergence for differentials $D\ups$ to $Du$ and the convergence of Jacobians of $\ups$ to the Jacobian of $u$. The minimum principle  keeps the maps $\ups$ injective all the way through the homotopy, which enables us to conclude the RKC theorem in Section~\ref{Sec10}.

\smallskip

In order to put our results in a wider perspective, let us notice that our variant of Rad\'o--Kneser--Choquet theorem can be seen as a contribution to the studies of the mapping problem for $p$-harmonic type mappings, known for example in the setting of planar conformal mappings (the Riemann mapping theorem) or planar quasiconformal mappings (the measurable  Riemann mapping theorem). We finish the introduction with a remark that our approach to the Rad\'o--Kneser--Choquet theorem on surfaces can be a subject of further generalization for isotropic energies, other than the $p$-harmonic ones.

\smallskip

\noindent {\bf Acknowledgements.}  Parts of the manuscript were written when T.A. was visiting the University of Helsinki and J.J. and A.K. were visiting the Institute of Mathematics of the Polish Academy of Sciences in Warsaw. Authors would like to thank the hosting institutions for their hospitality and support.

The authors would like to thank the anonymous referees whose remarks helped to improve the exposition in this paper.

\section{Differential geometry on Riemannian surfaces}\label{sec-basic}

\noindent
In this section we recall some basics of the differential forms and the related geometry of Riemannian surfaces, needed in the further presentation, especially in Section~\ref{sec-subharm}.

Let $M$ be a Riemannian surface with boundary and $N'$ be a compact Riemannian surface without boundary. We denote the global coordinates on $M$ by $z=x+iy = x^1 + i x^2$ and the local coordinate chart on $N'$ by $\tilu=\tilu^1+i\tilu^2$. We also denote by $z$ the coordinate map $z:M\to z(M)\subset \C$ and similarly for $\tilu$. Namely, $\tilu: N' \to \tilu(N')\subset \C$. Since $\partial M$ is assumed to be non-empty, we also require that for any $x_0\in \partial M$ and for all sets $U\subset M$ open in the topology of $M$ with $x_0\in \overline{U}$ it holds that $z(U\cap M)$ is, upon postcomposition with a conformal map in $\C$ (if necessary),  mapped into a half disc $B(0,1)\cap \{\Im z\geq 0\}$ with $z(x_0)\in B(0,1)\cap \{\Im z=0\}$.

 Moreover, we equip both surfaces with conformal metrics $\sigma(z)|dz|^2$ on $M$ and $\rho(\tilu)|d\tilu|^2$ on $N'$. We sometimes abbreviate these metrics as $\sigma$ and $\rho$. The existence of the conformal (isothermal) coordinates on the Riemannian surfaces can be proven, e.g. by employing the Beltrami equation, see Imayoshi--Taniguchi~\cite[Chapter 1.5.1]{imtan}. Moreover, in what follows we assume that $\sigma$ and $\rho$ are both bounded from below and above.

Recall that, if $w = a(\tilde{u})\, d\tilde{u} + b(\tilde{u})\, d\overline{\tilu}$ is a $1$-form on $N'$, its pullback by a mapping $u : M \to N'$ is a $1$-form $u^{*}(w)$ on $M$. It is defined formally as
\begin{equation}\label{defn-pullb}
u^{*}(w) = a(u(z))\big(d(\tilde{u}^1 \circ u) + id(\tilde{u}^2 \circ u)\big) + b(u(z))\big(d(\tilde{u}^1 \circ u) -i d(\tilde{u}^2 \circ u)\big).
\end{equation}

For the further discussion of the differential geometry on Riemannian surfaces, we will use similar notation as in Schoen--Yau \cite[Section 1]{sy}. Let us define the following $1$-forms on $M$ and $N'$, respectively,
\begin{equation}\label{sig-om}
 \theta=\sqrt{\sigma(z)}dz, \qquad \omega=\sqrt{\rho(\tilu)}d\tilu.
\end{equation}
For a mapping $u:M\to N'$ we define first derivatives of $u$ with respect to $\theta$ via the pullback of $1$-form $\omega$
\begin{equation}\label{om-pullb}
\aligned
 u^{*}(\omega) &=\utet \theta + \utetb\, \tbar \\
 u^{*}(\overline{\omega}) &= \ubtet \theta + \ubtetb\, \tbar.
 \endaligned
\end{equation}
Using the definition of the pull-back and \eqref{om-pullb} we may find $\utet, \utetb, \ubtet, \ubtetb$. Indeed, by \eqref{defn-pullb} and \eqref{sig-om} we have
\begin{equation}\label{pullback}
 u^{*}(\omega)= u^{*}(\sqrt{\rho(\tilu)}d\tilu)=\sqrt{\rho(u(z))}(du^1+i du^2)=\sqrt{\rho(u(z))}(u_z^1dz+u_{\zbar}^1d\zbar+iu_z^2dz+iu_{\zbar}^2d\zbar),
\end{equation}
where we use the standard notation for the complex derivatives of $u$
$$
2\uz\,=\,u_x-\,i\,u_y,\quad 2\ub\,=\,u_x+\,i\,u_y.
$$
Using the definition of $1$-form $\theta$ we find that
\begin{equation*}
\utet \theta + \utetb\, \tbar = \utet \sqrt{\sigma(z)}dz + \utetb\,\sqrt{\sigma(z)}d\zbar.
\end{equation*}
Comparing the corresponding coefficients in the last equation and \eqref{pullback} we arrive at the following formulas:
\begin{equation*}
\aligned
 \utet=\frac{\sqrt{\rho(u(z))}}{\sqrt{\sigma(z)}} (u_z^1+i u_z^2)=\frac{\sqrt{\rho(u(z))}}{\sqrt{\sigma(z)}}\uz, 
 \qquad
 \utetb=\frac{\sqrt{\rho(u(z))}}{\sqrt{\sigma(z)}} (u_\zbar^1+i u_\zbar^2)=\frac{\sqrt{\rho(u(z))}}{\sqrt{\sigma(z)}}\ub.
 \endaligned
\end{equation*}
Similar direct computations allow us to find that $\ubtet=\overline{\utetb}$ and $\ubtetb=\overline{\utet}$.

Following the standard approach we define $\utt$, $\uttb$, $\utbt$, $\utbtb$, the second covariant derivatives of map $u$ (in fact, the Hessian of $u$):
\begin{align*}
 d\utet +\utet\theta_c-\utet u^{*}\omega_c&=\utt \theta+ \uttb \tbar \\
 d\utetb +\utetb\overline{\theta}_c-\utetb u^{*}\omega_c&=\utbt \theta+ \utbtb \tbar.
\end{align*}
Here $\theta_c$ and $\omega_c$ are Riemannian connection $1$-forms on $M$ and $N'$, i.e.,
$$
\aligned
\theta_c &= \frac{\partial \log(\sigma(z))}{\partial \bz}d\bz - \frac{\partial \log(\sigma(z))}{\partial z}dz, \\
\omega_c &= \frac{\partial \log(\rho(\tilu))}{\partial \overline{\tilu}}d\overline{\tilu} - \frac{\partial \log(\rho(\tilu))}{\partial \tilu}d\tilu.
\endaligned
$$

\section{Sobolev spaces of mappings between surfaces and the Euler--Lag\-range system of equations}\label{sect-map}

\noindent
Recall that by $M$ and $N'$ we denote Riemannian surfaces equipped with conformal metrics $\sigma$ and $\rho$, respectively. Moreover, we assume that $\sigma$ and $\rho$ are both bounded. As in our standing assumptions for Theorem \ref{MainThm}, $M$ will be a surface diffeomorphic to the unit disc and $N'$ will be a compact surface without boundary. Recall further, that $N \subset N'$ denotes a subset which is sufficiently small in diameter. In particular, $N$ is covered by a single coordinate chart and is also diffeomorphic to the unit disc.

To define our minimization problem properly, we first need to recall how the notion of the Sobolev spaces is defined for mappings with values in a manifold. For a mapping $v:M \to N'$ which may take values in different charts of manifold $N'$, one typically defines the \emph{extrinsic} Sobolev space $W^{1,p}_{ex}(M,N')$ by means of the isometric Nash embedding $i : N' \to \R^k$. This embedding lets us identify the $2$-dimensional manifold $N'$ with a subset of $\R^k$ for some $k$, allowing us to define
\begin{align*} W^{1,p}_{ex}(M,N') &= \{ v : M \to N' : i \circ v \in W^{1,p}(M,\R^k)\}
\\&= \{ w \in W^{1,p}(M,\R^k) : w(x) \in i(N') \text{ for a.e. } x\in M\}.
\end{align*}
See also Pigola--Veronelli~\cite[Section 2]{pg3} and Fardoun--Regbaoui~\cite[Section 1]{fare} for the extrinsic Sobolev space.

For our purposes it will be often enough to study mappings that (at least locally) take values in a single coordinate chart, say, $(\Omega, \tilde{u})$ on $N'$. This coordinate chart allows us to define the \emph{intrinsic} Sobolev space, denoted simply by $W^{1,p}(M,\Omega)$, as
\[
W^{1,p}(M,\Omega) = \{v: M \to \Om : \tilde{u} \circ v \in W^{1,p}(M,\cc)\},
\]
where the right hand side is understood as the closure of $C^{\infty}(M, \C)$ in the Sobolev norm
\[
 \|\tilde{v}\|_{\Sob(M, \C)}:=\|\tilde{v}\|_{L^p(M, \C)}+\|D\tilde{v}\|_{L^p(M, \C)}.
\]
Further, we denote by  $W^{1,p}_{0}(M,\Omega)$ the Sobolev norm closure of smooth maps with compact support in $\mathrm{int}\, M$, $C^{\infty}_0(\mathrm{int}\, M, \C)$.

By Theorem 2 in~\cite{pg3} the intrinsic and extrinsic definitions of Sobolev spaces are the same for mappings taking values in a single coordinate chart. In particular, from this one can infer that if $v \in W^{1,p}(M,N)$, then $v \in W^{1,p}_{ex}(M,N')$.

In our energy functional \eqref{energy} we use the following inner product: Denote by $Du(z)= (u_\theta,u_{\bar{\theta}})$ and $Dv(z)=(v_\theta,v_{\bar{\theta}})$ the derivatives as in \eqref{om-pullb} of $u$ and $v$, respectively. Via the local charts of $N'$ these derivatives map into $\C^2$. We denote by
\[D_z u = (u_z, u_{\bz})\]
the regular differential matrix of a map $u : M \to N'$. The same notation is used for the differential matrix for maps from $M$ to $\cc$. The scalar product
\begin{equation}\label{our-scalar-prod}
 \langle D_z u(z), D_z v(z)\rangle := \Re(\uz \bar{v_{z}})+\Re(\ub \bar{v_{\overline{z}}} ).
\end{equation}
denotes the usual scalar product in $\rr^{2\times 2} = \cc^2$. We also define the norm $|Du|$ by the formula
\begin{equation}\label{eq-Du-norm}
  |Du(z)|^2 = |\ut|^2+|\utetb|^2= \frac{\rho(u(z))}{\sigma(z)} \left(|\uz|^2+|\ub|^2\right).
\end{equation}
Hence, the norm also depends on $\sigma$ and $\rho$. However, since we assume boundedness of both metrics, the Sobolev space arising from \eqref{eq-Du-norm} coincides with the intrinsic Sobolev space defined above and the corresponding Sobolev norms are equivalent. In our Sobolev estimates below we choose to use the following norm
$$
\| Du \|_{L^p} = \left(\int_M |Du|^pdV_M\right)^{\frac1p}.
$$

Let $u : M\rightarrow N'$ be a mapping a priori assumed to belong to $W^{1,p}_{ex}(M, N')$. For such a map let us consider the following energy integral:
\begin{equation}\label{en1}
 \mathbb{E}(u)=\int_{M} E(|Du|^2) dV_M.
\end{equation}
In this paper we are mainly concerned with energies having the stored energy function $E(|Du|^2) = \left(\epsilon^2 + |Du|^2\right)^{p/2}$ with $2 \leq p < \infty$ and $\epsilon \geq 0$, though the Euler-Lagrange systems of equations we derive in this section are also valid for more general energies.

\smallskip

From now on we will assume that $p \geq 2$ and $u \in W^{1,p}_{ex}(M, N')$ is a minimizer of the energy \eqref{en1} with $E(t) = (\ep^2 + |t|^2)^{p/2}$. Thus $u$ is, in particular, continuous and hence locally takes values in a single coordinate chart. The continuity follows for $p > 2$ by the Sobolev embedding theorem of Riemannian manifolds and for $p = 2$ from the classical theory of harmonic mappings, see e.g.~Chapter 3.8 in Jost~\cite{jost}.

Being a minimizer, the map $u$ is a stationary point of the energy \eqref{en1}. In other words, the first variation of $\mathbb{E}(u)$ vanishes. Let $\psi \in C_0^{\infty}(M,\C)$ be a test function and denote by $u^s$ a variation of $u$ such that $\left. \frac{d}{ds}  \right\vert_{s = 0} u^s = \psi$. Explicitly, $u^s$ is defined by $u^s(z) = \exp_{u(z)}(s \psi(z))$. We also abbreviate some notation by defining $\lambda(z)= E'(|Du(z)|^2)$. Since the map $u$ locally takes values in a single coordinate chart, for a test function $\psi$ with a support in a small enough neighborhood in $M$ the following computation is justified in local coordinates:
 $$
  \aligned
0 = \left. \frac{d}{ds}  \right\vert_{s = 0} \mathbb{E}(u^s) &=
 2\Re \int_{M} \lambda(z) \frac{\rho_{u}(u(z))}{\sigma(z)} (|\uz|^2+|\ub|^2)\psi +  \lambda(z) \frac{\rho(u(z))}{\sigma(z)}\left(\baruz\psi_z + \barub \psi_{\overline{z}}\right) dV_M\\
 &=-  2\Re \int_M \frac{\rho(u)}{\sigma}\left(2\lambda u_{z\bz} + \lambda_{\bz} u_z + \lambda_zu_{\bz} + 2\lambda \frac{\rho_u(u)}{\rho(u)}u_zu_{\bz} \right)\overline{\psi} dV_M,
 \endaligned
 $$
where the last equality is valid for $u \in C^2(M,N')$. Here we have slightly abused the notation, as by $\rho_u(u)$ we actually mean $(\rho \circ \tilu^{-1})_{\tilu} \circ \tilu \circ u$ where the derivative with respect to $\tilu$ is a regular complex derivative in the plane. This notational convention will be used throughout the paper.

In particular, we now find that the minimizer $u$ is a solution of the Euler--Lagrange system of equations
$$
2\lambda u_{z\bz} + \lambda_{\bz} u_z + \lambda_zu_{\bz} + 2\lambda \frac{\rho_u(u)}{\rho(u)}u_zu_{\bz} = 0.
$$
Namely, it holds that
\begin{equation}\label{u-system2}
 [\lambda(z) \uz]_{\zbar}+[\lambda(z) \ub]_{z}+2 \lambda(z)\left(\frac{\partial}{\partial u} \log \rho(u(z))\right)
 \uz\ub=0,
\end{equation}
which we interpret in the weak form, unless $u \in C^2(M,N')$. We state the above Euler--Lagrange system of equations also in our $\theta$-derivative notation:
\begin{equation}\label{u-system}
2\lambda \uttb + \lambda_{\theta} \utetb+\lambda_{\overline{\theta}} \utet=0,
\end{equation}
cf.  also the equation (9) on pg. 267 in \cite{sy}. Furthermore, we note that this equation is not equivalent to the divergent-type formulation
$[\lambda \utet]_{\tbar}+[\lambda \utetb]_{\theta}=0$.
In the literature the Euler--Lagrange systems of equations \eqref{u-system2}--\eqref{u-system} are sometimes expressed as
\begin{equation}\label{system}
 \delta \left(E'(|Du|^2)\,du\right)=0,
\end{equation}
where $\delta$ denotes the formal adjoint of the exterior differential $d$ with respect to the standard
$L^2$ inner product on vector-valued differential $1$-forms on $M$, see Hamburger~\cite{ham}.

Since we do not a priori assume that our mappings $u : M \to N'$ are smooth, we need instead to work with the weak formulation of the Euler--Lagrange systems of equations  \eqref{u-system2}--\eqref{u-system}.  Moreover, we also express the above system by using  notation that clearly emphasizes the relations between the mappings and the geometry of $M$ and $N'$. Namely, below we will use the notation $A(u)(Du, Du)$ to denote the value of the quadratic form $A$ at a point $u=u(z)\in N'$ for $z\in M$:
 \begin{equation}\label{u-system-A-fund}
   A(u)(Du, Du)=\left(\frac{\partial}{\partial u} \log \rho(u(z))\right)|Du|^2=\frac{\rho_{u}(u(z))}{\rho(u(z))}\,|Du|^2=
   \left(\Re\frac{\rho_{u}(u(z))}{\rho(u(z))}\,+\,i\Im\frac{\rho_{u}(u(z))}{\rho(u(z))}\right)\,|Du|^2,
 \end{equation}
 where the norm $|Du|$ is as in \eqref{eq-Du-norm}. The quadratic form $A:\C\times\C\to \C$ is related to the second fundamental form of $N'$ embedded to $\R^k$, see Section \ref{sect31}.

From the computations for our first variation we obtain the following weak formulation for the Euler-Lagrange system of equations.

\begin{defn}\label{phm-weak}
We say that $u:M \to N'$ is a weak solution of the Euler-Lagrange equations \eqref{u-system2}--\eqref{u-system}, if $u \in W^{1,p}_{ex}(M, N')$ satisfies
\begin{equation*}\label{u-system-weak-pre}
 \Re  \int_{M} \lambda(z) \frac{\rho_{u}(u(z))}{\sigma(z)} (|\uz|^2+|\ub|^2)\psi\, dV_M =- \Re \int_{M} \lambda(z) \frac{\rho(u(z))}{\sigma(z)}\left(\baruz\psi_z + \barub \psi_{\overline{z}}\right)\, dV_M.
\end{equation*}
or, equivalently,
\begin{equation}\label{u-system-weak}
 - \int_{M} \lambda(z) \frac{\rho(u(z))}{\sigma(z)}\langle D_z u, D_z \psi \rangle \,dV_M =\int_{M} \lambda(z) A(u)(Du, Du)\cdot \psi \,dV_M,
\end{equation}
for all $\psi\in \Sobzero (M, \C)$ such that $u$ maps the support of $\psi$ onto a single coordinate chart of $N'$.

By $\langle\cdot,\cdot\rangle$ we denote the scalar product as defined in \eqref{our-scalar-prod}, while the notation $X\cdot Y$ for $X,Y\in \C$ stands for the standard Euclidean scalar product in $\C$.
\end{defn}

\subsection{The Extrinsic Euler-Lagrange equations via the Nash embedding}\label{sect31}

To adopt certain proof techniques from the literature, we will also need to derive an Euler-Lagrange system for the map $v := i \circ u$, where $u: M \to N'$ denotes an energy-minimizer and $i:N' \to \rr^k$ denotes the isometric Nash embedding of $N'$ into Euclidean space. If $i(N')$ is a compact surface in $\rr^k$, then for sufficiently small $\epsilon$ the shortest distance projection to $i(N')$ is well-defined on the set $i(N') + B(0,\epsilon)$. Denoting such a projection by $\Pi_{N'}$, we may then apply the first variation $\Pi_{N'}(v + \epsilon \psi)$ to the map $v$, for a suitably chosen test function $\psi \in C_0^\infty(M, \R^k)$, as $v$ minimizes the energy integral
\[
\int_M |\nabla v|^p dV_M  = \int_{M} |D u|^p dV_M
\]
among maps from $M$ to $i(N')$ with given boundary data $\phi_0$.
Doing so we arrive at the Euler-Lagrange equation for $v$, namely
\begin{equation}\label{v-system-weak}
 - \int_{M} \lambda \nabla v \cdot \nabla \psi \,dV_M =\int_{M} \lambda A'(v)(\nabla v, \nabla v)\cdot \psi \,dV_M,
\end{equation}
where, for the convenience, we denote $\lambda = \lambda(|\nabla v|)$  and $A'$ stands for the second fundamental form of $N'$ in $\R^k$.

\smallskip

More precisely, $v = (v^1, \ldots, v^k) : M \to \R^k$ and the gradient is taken with respect to the metric $\sigma(z)|dz|^2$, that is,
$$
\nabla v^j = [\sigma]^{-1}dv^j = \sum_{\alpha = 1}^{2} \sum_{\beta = 1}^{2} \sigma^{\alpha \beta} \frac{ \partial v^j }{\partial x^\beta} \frac{ \partial \, }{\partial x^\alpha},
$$
where we denote by $[\sigma]^{-1}$ the inverse matrix of the metric $\sigma(z)|dz|^2$;
$$
\sigma^{\alpha\beta} =\begin{cases}
\sigma(z)^{-1}, & \alpha = \beta \\
0, & \alpha \neq \beta
\end{cases}.
$$
As the gradient is a linear map defined on the tangent space we define its norm as the Hilbert-Schmidt norm (the trace with respect to the metric), i.e.,
$$
|\nabla v|^2 =  \sum_{j = 1}^k |\nabla v^j|^2_{\sigma} = \sum_{j = 1}^k \left( \sum_{\alpha = 1}^{2} \sum_{\beta = 1}^{2} \sigma^{\alpha \beta}\frac{ \partial v^j }{\partial x^\alpha} \frac{ \partial v^j }{\partial x^\beta} \right).
$$

In particular, we have $|\nabla v| = |Du|$. Indeed, $|Du|^2$ is the trace (with respect to the metric of $M$) of the pullback under $u$ of the metric tensor of $N'$ and  for the pullback of the metric tensor $\rho$ we have for $X, Y \in TM$
$$
u^*(\rho)(X, Y) = \rho(du(X), du(Y)) = di(du(X)) \cdot di(du(Y)) = u^*(di)(X) \cdot u^*(di)(Y) = dv(X) \cdot dv(Y),
$$
where by $\rho(du(X), du(Y))$ we mean the value of the conformal metric $\rho(\tilu)|d\tilu|^2$ on $(du(X), du(Y))$. Above the second equality holds as $i$ is the isometric Nash embedding (the dot product is the standard Euclidean scalar
product of $\R^k$) and the last equality follows from $dv = d(i\circ u) =  d(u^{*}(i))= u^{*}(di)$. Now, the  trace (with respect to the metric of $M$) of the right hand side is $|\nabla v|^2$ and thus $|\nabla v|^2 = |D u|^2$.

The dot product on the right hand side of \eqref{v-system-weak} means the standard Euclidean scalar
product of $\R^k$ and
$$
 \nabla v \cdot \nabla \psi  = \sum_{\alpha = 1}^{2} \sum_{\beta = 1}^{2} \sigma^{\alpha \beta} \frac{ \partial v }{\partial x^\alpha} \frac{ \partial \psi }{\partial x^\beta}.
$$
Recall that the second fundamental form is a quadratic form on the tangent space that takes values in the normal space, that is,
$A'(v) = \nabla^{\R^k} - \nabla^{N'} : T_{u(z)}N' \times T_{u(z)}N' \to (T_{u(z)} N')^{\bot}$, where $N'$ is a submanifold of $\R^k$. Thus, in coordinates,
${A'}_{mn}^{j}(v) = \frac{\partial^2 i^j}{\partial u^m \partial u^n} - \sum_{l = 1}^{2}\frac{\partial i^j}{\partial u^l} \Gamma^l_{mn}$,
where $\Gamma_{mn}^l$ stand for the Christoffel symbols of metric $\rho$, and $A'(v)$ acts bilinearly on  $\frac{\partial u}{\partial x^\alpha} = du\left(\frac{\partial \,}{\partial x^\alpha}\right) \in TN'$:
$$
{A'}_{mn}^{j}(v)(\nabla v, \nabla v) = \sum_{\alpha = 1}^{2} \sum_{\beta = 1}^{2}\sigma^{\alpha\beta}{A'}_{mn}^{j}\left(\frac{\partial u}{\partial x^\alpha}, \frac{\partial u}{\partial x^\beta} \right).
$$
Moreover, we see straight from the coordinate definition of the second fundamental form that
$|A'(v)(\nabla v, \nabla v)| \leq C_{N'} |\nabla v|^2$,
where $C_{N'}$ depends only on geometry of $N'$ as $N'$ is compact and coefficients $A^{' j}_{mn}$ are defined via $\rho$ and the Nash embedding.

\section{The $p$-harmonic and uniformly elliptic $p$-harmonic systems}\label{sect-systems}

\noindent
As mentioned in the previous section, we will restrict ourselves to the study of two special cases of the energy functional \eqref{energy} in order to prove our main result, Theorem \ref{MainThm}. The first of these is unsurprisingly the $p$-harmonic energy
\begin{equation}\label{penergy}
\mathbb{E}(u) = \int_{M} |Du|^p dV_M,
\end{equation}
where it is always assumed that the exponent $p \geq 2$. In certain parts of the proof of Theorem \ref{MainThm} we will also need to work with an auxiliary energy functional, called the \emph{uniformly elliptic $p$-harmonic energy}, which we define as follows:
\begin{equation}\label{ep-energy}
\mathbb{E}_{\ep}(u) = \mathbb{E}_{\ep,p}(u) =\int_{M} (\ep^2+|Du|^2)^{\frac{p}{2}} dV_M.
\end{equation}
Here $\ep\in (0,1)$ is a fixed number and again $p\geq 2$. If we set $\ep=0$, then we retrieve the $p$-harmonic case \eqref{penergy}. The benefit in considering this energy lies in the fact that for $\epsilon > 0$, the associated Euler--Lagrange system of equations becomes uniformly elliptic. This makes it easier to deduce properties such as the boundary regularity for minimizers of \eqref{ep-energy}, leading us to prove the RKC-theorem first in the case $\epsilon > 0$ and then concluding the proof of Theorem \ref{MainThm} with a suitable convergence result as $\epsilon \to 0$.

Under the notation and assumptions of Theorem \ref{MainThm} and Section \ref{sect-map}, we will always denote by $u$ a minimizer of the $p$-harmonic energy \eqref{penergy} among mappings in the class $W^{1,p}_{ex}(M,N')$ with boundary values equal to the homeomorphic map $\phi_0 : \partial M \to \partial N$. Furthermore, $u^{\epsilon}$ will denote a minimizer of the energy \eqref{ep-energy} among the same class of mappings, in particular having the same boundary values $\phi_0$. Some results which we state for $u^{\ep}$ hold in the case $\ep = 0$ as well, for this consult the exact assumptions of the result in question.

Recall that as energy minimizers, mappings $u$ and $u^{\ep}$ are weak solutions to the associated Euler--Lagrange systems of equations for the energies \eqref{penergy} and \eqref{ep-energy} as in \eqref{u-system-weak} with $\lambda:=\lambda^{\ep}$, where
\begin{equation}\label{def-lam}
\lambda^{\ep}=\left(\ep^2+|\dups|^2\right)^{\frac{p-2}{2}}.
\end{equation}
In the following discussion, we often call the variational system corresponding to the case $\ep=0$ the \emph{$p$-harmonic system}, whereas the analogous system for $\ep>0$ is called the \emph{$\ep$-perturbed ($p$-harmonic) system} or \emph{perturbed ($p$-harmonic) system}.

\subsection{The $C^{1,\alpha}$-regularity of the stationary solutions of the $\ep$-perturbed systems}\label{sec-C1alp}

For the forthcoming computations we will need some a priori knowledge on the regularity of maps $u$ and $u^{\ep}$. For the mapping $u$, we explicitly assume $C^{1,\alpha}$-regularity up to the boundary of $M$ as part of Theorem \ref{MainThm}. For the $\ep$-perturbed system we find such regularity as a consequence of the uniform ellipticity. First we recall an interior regularity result:

\begin{prop}\label{jatkuvuus}
 Let $\ups : M \to N'$ denote a weak solution of the Euler--Lagrange system \eqref{u-system-weak} with $\lambda^{\ep}$ as in \eqref{def-lam}. Then $\ups$ belongs to $C^{1,\alpha}_{loc}(\mathrm{int}\,M,N')$ with $\alpha$ depending on $p$ and the geometry of $M$ and $N'$.
\end{prop}
\begin{proof}
 As we are interested in the local regularity property of map $\ups$, we are able to exploit a similar argument as in the first part of the proof for Proposition 2.3 in Sacks--Uhlenbeck~\cite{sau}, which also employs Theorem 1.11.1 in Morrey~\cite{mor}, cf. conditions (1.10.7'') on pg. 33 in \cite{mor}.
Together, these imply that $\ups$ is $C^{1,\alpha}(\Om,N')$ for all open $\Om\subset \mathrm{int}\,M$ as desired.
\end{proof}

Similar regularity results holding up to the boundary of the domain are much less prominent in the literature, and this is the reason for our a priori regularity assumption on the $p$-harmonic map $u$ (i.e. the case $\ep = 0$). However, for the uniformly elliptic systems a result which holds up to the boundary can be found in the paper of Beck~\cite{beck}, allowing us to state the following regularity observation.

\begin{theorem}[Theorem 1.4 in \cite{beck}]\label{Thm14-Beck}
 Let $p\geq 2$ and $\Om\subset \C$ be a bounded domain of class $C^{1,2\beta}$ for some $\beta<1/2$. Moreover, let map $g\in C^{1,2\beta}(\overline{\Om}, \C)$. If $u \in \Sob(\Om, \C)$ is a local minimizer of the
energy functional $\mathcal{F}(u,\Om)$, under the growth assumptions (2a--b), (4), (7), (9a--b) in \cite{beck} and boundary values $u=g$ on $\bdy \Om$, then $u \in C^{1,\beta}(\overline{\Om}, \C)$.
\end{theorem}
In fact, Beck's result applies to a wider class of energy functionals $\mathcal{F}(u,\Om)$ than our  $\mathbb{E}_{\ep}(u)$, see \cite[formula (1)]{beck}, Theorem 1.4 and Corollary 1.5 in \cite{beck}. Since this result is only stated in the flat case, we are unable to apply it at this moment due to the fact that the target space $N'$ might not be diffeomorphic to $\C$. However, in Section \ref{sect-geometry} we will show that a minimizer indeed takes values in a single coordinate chart. Hence at this point let us assume that the map $\ups$ takes values in a single coordinate chart -- this allows us to reduce to the flat case via coordinate maps and obtain the following corollary:

\begin{cor}\label{BeckCorollary} Let $\ep > 0$ and suppose that the map $u^{\ep}:M \to N'$ minimizes the energy \eqref{ep-energy} among maps in $W^{1,p}_{ex}(M,N')$ with fixed $C^{1, \alpha}$-smooth boundary values $\phi_0$ on $\partial M$. Suppose additionally that the image set $u^{\ep}(M)$ is contained in a single coordinate chart. Then $u^{\ep}$ is $C^{1,\beta}$-regular up to the boundary for some $\beta > 0$.
\end{cor}

For the sake of completeness of our discussion, we now provide an argument that energies \eqref{ep-energy} satisfy the assumptions of Theorem~\ref{Thm14-Beck}. Another motivation for somewhat detailed verification of assumptions from \cite{beck} is that in Section~\ref{sect-homotopy} we study similar energies, see \eqref{Ft-system}.

Recall, that under our assumptions $M$ is simply-connected Riemannian surface and, thus, covered by one coordinate map. Moreover, in Corollary \ref{BeckCorollary} we assume that the image $u^{\ep}(M)$ is contained in a single coordinate chart on $N'$ with metric $\tilu$. Furthermore, we denote by $\Om:=z(\mathrm{int}\,M)$, the image of $\mathrm{int}\,M$ under the coordinate map $z : M \to \bar{\Om}$. With this notation we define the energy $W_\ep:\Om\times \C\times \C^2\to \R$, associated with $\mathbb{E}_{\ep}$, as follows:
\begin{equation}\label{f-ep-energy}
W_\ep(z^{-1},\tilu,\zeta):=\left(\ep^2+\frac{\rho(\tilu)}{\sigma(\zinv(z^{-1}))}|\zeta|^2\right)^{\frac{p}{2}}|J_{\zinv}(z^{-1})|,
\end{equation}
where $J_{\zinv}(z^{-1})$ stands for the Jacobian of a map $\zinv$ (the inverse map for $z$) at the point $z^{-1}\in \Om$. (Here, the inverse map of $z$ is denoted by $\zinv$, instead of $z^{-1}$ in order to avoid the confusion with notation
for a point $z^{-1}:=\zinv(z)\in \Om$).

One verifies that a map $\zeta\to W_\ep(\cdot,\cdot, \zeta)$ is of class $C^2$, i.e. the second derivative operator of $W_\ep(z^{-1},\tilu,\zeta)$ with respect to $\zeta$ variable is continuous in $\Om\times \C\times \C^2$. Indeed, this follows by the direct twice differentiation of $W_\ep(\cdot,\cdot, \zeta)$ with respect to variables $\zeta_1,\ldots, \zeta_4$ with $\zeta=(\zeta_1, \zeta_2, \zeta_3, \zeta_4)\in \R^4\approx \C^2$, also by the fact that $\ep>0$, and so the expression $\left(\ep^2+\frac{\rho(\cdot)}{\sigma(\zinv(\cdot))}|\zeta|^2\right)^{\frac{p}{2}-2}$ appearing in $\frac{\partial^2}{\partial \zeta_i \partial \zeta_j} W_\ep(\cdot,\cdot, \zeta)$ for $i,j=1,\ldots, 4$ is defined for all $\zeta$. Moreover, it holds that
$$
\nu |\zeta|^p \leq W_\ep(z^{-1},\tilu,\zeta) \leq L(1+|\zeta|)^p
$$
with constants
\begin{equation}\label{nuL-constants}
 \nu:=\min_{\tilu\in \C} \min_{\Om}\frac{\rho(\tilu)}{\sigma(\zinv(z^{-1}))}\min_{\Om}|J_{\zinv}(z^{-1})|>0\quad \hbox{and} \quad
 L:=\max\left \{1,\max_{\tilu\in \C} \max_{\Om}\frac{\rho(\tilu)}{\sigma(\zinv(z^{-1}))}\right\}\max_{\Om}|J_{\zinv}(z^{-1})|.
\end{equation}

Therefore, we verified conditions (2a--b) and (9a) in \cite{beck}. Similarly, by employing the mean value theorem for functions, we check \cite[assumption (4)]{beck}:
\begin{align}
 &|W_\ep(z_1^{-1},\tilu_1,\zeta)-W_\ep(z_2^{-1},\tilu_2,\zeta)| \nonumber \\
 &=\left(\ep^2+\frac{\rho(\tilu_1)}{\sigma(\zinv(z_1^{-1}))}|\zeta|^2\right)^{\frac{p}{2}}|J_{\zinv}(z_1^{-1})|-\left(\ep^2+\frac{\rho(\tilu_2)}{\sigma(\zinv(z_2^{-1}))}|\zeta|^2\right)^{\frac{p}{2}}|J_{\zinv}(z_2^{-1})|\nonumber \\
 &=\left[\left(\ep^2+\frac{\rho(\tilu_1)}{\sigma(\zinv(z_1^{-1}))}|\zeta|^2\right)^{\frac{p}{2}}-\left(\ep^2+\frac{\rho(\tilu_2)}{\sigma(\zinv(z_2^{-1}))}|\zeta|^2\right)^{\frac{p}{2}}\right]|J_{\zinv}(z_1^{-1})|\nonumber \\
&\phantom{AA} +\left(\ep^2+\frac{\rho(\tilu_2)}{\sigma(\zinv(z_2^{-1}))}|\zeta|^2\right)^{\frac{p}{2}}\left(|J_{\zinv}(z_1^{-1})|-|J_{\zinv}(z_2^{-1})|\right)\nonumber \\
 &\leq c(p, \epsilon)L(1+|\zeta|)^p\left(\left\|\frac{\partial G}{\partial {\tilu}}\right\|_{L^{\infty}(\Om)}|\tilu_1-\tilu_2|+\left\|\frac{\partial G}{\partial {z^{-1}}}\right\|_{L^{\infty}(\Om)}|z_1^{-1}-z_2^{-1}|\right)\big\|J_{\zinv}\big\|_{L^{\infty}(\Om)} \nonumber \\
 &\phantom{AA}+\frac{1}{|J_{\zinv}(z_2^{-1})|}L(1+|\zeta|)^p\big\|\nabla |J_{\zinv}|\big\|_{L^{\infty}(\Om)}|z_1^{-1}-z_2^{-1}|,  \label{f-ep-ass4} \\
 & \leq\! c(p, \epsilon)L(1+|\zeta|)^p\!\left(\big\|\frac{1}{J_{\zinv}}\big\|_{L^{\infty}(\Om)}+\big\|J_{\zinv}\big\|_{L^{\infty}(\Om)}\right)\!\Bigg[\left\|\frac{\partial G}{\partial {\tilu}}\right\|_{_{L^{\infty}(\Om)}}\!|\tilu_1-\tilu_2| \nonumber \\
 &\phantom{AAAAAAAAAAAAAAAAAAAAAAAAAAAAAAA}+\left(\left\|\frac{\partial G}{\partial {z^{-1}}}\right\|_{_{L^{\infty}(\Om)}}\!+\big\|\nabla J_{\zinv}\big\|_{_{L^{\infty}(\Om)}}\!\right)\!|z_1^{-1}-z_2^{-1}|\Bigg],  \label{f-ep-ass44}
\end{align}
where $G(\tilu,z^{-1}):=\frac{\rho(\tilu)}{\sigma(\zinv(z^{-1}))}$ and the boundedness of $L^{\infty}$ norm of $\nabla G$ follows from the corresponding properties of the geometries $\rho$ and $\sigma$. Moreover, since by assumptions $M$ is covered by one map, we may as well assume that it is sense-preserving and let $|J_{\zinv}|=J_{\zinv}$. Then, in \eqref{f-ep-ass4}, \eqref{f-ep-ass44} we may instead consider $\|\nabla J_{\zinv}\|_{L^{\infty}(\Om)}$ and estimate it in terms of (second order) partial derivatives of $\zinv$. From inequality \eqref{f-ep-ass44} we infer that \cite[assumption (4)]{beck} holds and that the modulus of continuity with respect to variables $(\tilu, z^{-1})$, denoted in \cite{beck} by $\omega$, depends on the geometry of $M$ and the coordinate chart on $N'$, $p$ and $\epsilon$ only. This estimate a priori depends also on the choice of map $z$ and constants arising from estimates of the first and the second partials of $\zinv$. However, when changing the coordinate system on $M$, we change values of these constants by a factor related to a supremum norm of the change of variables, its Jacobian and Hessian only.

Observe that by \eqref{f-ep-ass44} condition (7) in \cite{beck} is trivially satisfied as, upon constant multiplication, the modulus of continuity, denoted by $\omega_1$ in \cite[formula (7)]{beck}, is in our case $\omega_1(t):=t$. Thus, for $t\in (0,1)$ one can choose any $\alpha_1\in (0,1)$, whereas for $t>1$ we utilize the boundedness of $z_1^{-1}, z_2^{-1}$ by $\diam M$ and the boundedness of $\tilu_1, \tilu_2$ by diameter of the coordinate chart on $N'$, in order to bound $\omega_1$ by a constant.

Finally, we verify by the direct computations, standard for the $p$-harmonic type energies, that $W_\ep$ satisfies assumption (9b) in \cite{beck}, i.e.
\begin{equation}\label{f-ep-ass9}
\nu(1+|\zeta|)^{p-2}|\xi|^2\leq D_{\zeta \zeta}W_\ep(z^{-1},\tilu,\zeta)\xi\cdot\xi \leq L (1+|\zeta|)^{p-2}|\xi|^2
\end{equation}
for any $\xi\in \C^{2}$. Here $\nu$ and $L$ are similar to constants in \eqref{nuL-constants} and, additionally,  depend on $p$ in a bounded manner.

\smallskip

Recall that $\phi_0$ denotes a homeomorphism from $\partial M$ into $\partial N$, see assumptions of Theorem~\ref{MainThm}. The above discussion allows us to apply Theorem~\ref{Thm14-Beck}  to the Dirichlet problem for the energy
\begin{equation}
 \int_{\Om} W_\ep(z^{-1}, u, Du) \label{f-ep-energy2}
\end{equation}
with the boundary data $z(\phi_0)$.   The minimizer $\ups$ of \eqref{ep-energy} (for $\ep > 0$)  corresponds to a minimizer $\tilu\circ\ups\circ \zinv:\Om \to \C$ of \eqref{f-ep-energy2}. Thus we obtain Corollary~\ref{BeckCorollary} provided we show existence and uniqueness of minimizers to~\eqref{ep-energy}. This discussion will be postponed till Section~\ref{sect-ex}. It turns out that the existence of minimizers for \eqref{ep-energy} and \eqref{f-ep-energy2} follows in the usual way from the direct methods in the calculus of variations, since the energy functionals in subject are continuous, quasi-convex with polynominal growths, bounded from below on the corresponding Sobolev spaces and weakly lower semicontinuous. However, for the uniqueness part we need to understand better the geometry of minimizers. We will study this issue in the next sections.

\subsection{Positivity of the Jacobian implies smoothness}\label{subs-4.3}

In the next section we prove Theorem~\ref{Ttheorem}, in which we assume that the Jacobian determinant of our  minimizer of~\eqref{energy} is positive. In order to justify the computations in Theorem~ \ref{Ttheorem}, we need to assume that  maps are at least $C^3$-smooth. However, any stationary solution which has positive Jacobian is already smooth, as proven in the following proposition.

\begin{prop}\label{PosJacSmoothness} Let $\ups$ denote a stationary solution of the energy \eqref{ep-energy}, including the case $\ep = 0$. Suppose that $J_{\ups}>0$ in $M$. Then $\ups$ is $C^{\infty}$-smooth in $M$.
\end{prop}
\begin{proof} It is a type of classical result that the solutions to many nonlinear divergence-type equations are smooth in the set where their differential does not vanish. While we did not find an exact reference stating this to hold for the $\epsilon$-system arising from the energy \eqref{ep-energy}, we may adopt the lines of the proof for example from Lemma 2.1 in Duzaar~\cite{duzaar}. There, the claim of the proposition is shown for $p$-harmonic mappings between manifolds, which takes care of the case $\epsilon = 0$. The underlying assumption in \cite{duzaar} is that the manifold $M$ is at least three-dimensional, but the proof of the referenced Lemma 2.1 never uses this fact and, in fact, the same proof applies in the two dimensional case, as well.

As for the case $\epsilon > 0$, we briefly recall the arguments for the case $\epsilon = 0$ in \cite{duzaar} to verify that the same proof goes through for $\ep>0$ as well. The proof is essentially divided into two steps. First one shows that the map $\vps = i \circ \ups$ in fact lies in $W^{2,2}_{loc}(M, \R^k)$, where we have Nash embedded $N'\,$ to $\R^k$ by $i$. After showing this, one is able to differentiate the Euler-Lagrange equations in a weak sense to arrive at a linear second-order equation for the first derivatives of $\vps$. This linear equation becomes uniformly elliptic due to the fact that the differential of $\vps$ does not vanish. The theory of elliptic linear equations then yields a self-improving estimate on the regularity of the derivatives of $\vps$, which can be iterated further to obtain the required smoothness.

\smallskip

We start with the argument to show that $\vps\in W^{2,2}_{loc}(M, \R^k)$, for which we adopt the arguments of Lemma 2.2 in Duzaar--Fuchs~\cite{duzaarfuchs}. Our goal here is to justify that the proof of the lemma also goes through in the case $\ep > 0$, yielding that $(\ep^2 + |\nabla\vps|^2)^{(p-2)/4}\nabla\vps$ is in the regularity class $W^{1,2}_{loc}(M, \R^k)$. Note that all the calculations are done via the coordinate map $z$ and thus $\vps\,$ is a map from planar domain to $\R^k$. Let us denote by $\lambda^{\ep}(t) = (\ep^2 + t)^{(p-2)/2}$. Upon computations similar to those in the proof of~\cite[Lemma 2.2]{duzaarfuchs} one verifies that the estimates in the proof leading up to formula (2.6) in \cite{duzaarfuchs} hold when $|Du|^{p-2}$ is replaced with $\lambda^{\ep}(|\nabla\vps|^2)$ and $|Du|^{p-4}$ with $d\lambda^{\ep}/dt(|\nabla\vps|^2)$, modulo a constant appearing from the differentiation with respect to $t$. Indeed, let us present the key steps of the reasoning. Instead of the key equation (1.2) in~\cite{duzaarfuchs}  we consider \eqref{v-system-weak} for $\lambda:=\lambda^{\ep}(|\nabla\vps|^2)$ and test functions $\psi\in C_{0}^1(M, \R^k)$:
\begin{equation}\label{prop44-v-system-weak}
 \int_{M} \lambda^{\ep}(|\nabla\vps|^2) \left( \nabla \vps \cdot \nabla \psi +A'(\vps)(\nabla \vps, \nabla \vps)\cdot \psi\right) \,dV_M=0,
\end{equation}
where  $A'$ stands for the second fundamental form of $N'$ in $\R^k$ (cf. the discussion in Section~\ref{sect31}). Upon applying the coordinate map $z: M \to \overline{\Om}$, where $\Om=z({\rm int}M)$, we obtain the corresponding weak formulation for a map from a planar domain to $\R^k$: 
\begin{equation}\label{prop44-v-system-weak2}
 \int_{z(M)} \lambda^{\ep}(|\nabla\vps(t)|^2) \left( \nabla \vps(t) \cdot \nabla \psi(t) +A'(\vps(t))(\nabla \vps(t), \nabla \vps(t))\cdot \psi(t)\right) |J_{\zinv}(t)|\,d\mathcal{L}^2=0,
\end{equation}
where $t=z^{-1}\in \Om$ is the point corresponding to the given $z\in M$ under $\zinv$ the inverse map of $z$. In order to simplify further presentation, in what follows we denote $\vps(t)$ by $\vps$. Next, as in \cite{duzaarfuchs}, for a ball $B_r\subset z(M)$ we consider the nonnegative test function $\phi\in C_0^{\infty}(B_{\frac32 r}, \R)$, such that $\phi\equiv 1$ on $B_r$ and $|\nabla \phi|\leq \frac{4}{r}$. As in the proof of~\cite[Lemma 2.2]{duzaarfuchs} we consider $\Delta_{h,i}\vps$, the difference quotient in the $i$-direction, $i = 1, 2$, i.e., $\Delta_{h,i}\vps = (1/h)(\vps(z + he_i) - \vps(z))$. Notice that $\Delta_{h,i} \vps \in \R^k. $Furthermore, for the test function $\phi^2\Delta_{h,i}\vps$ we get a counterpart of estimate (2.3) in \cite{duzaarfuchs}:
\begin{align}
0&=\int_{B_{2r}} \Delta_{h,i}\left(\lambda^{\ep}(|\nabla\vps|^2) \nabla \vps \right)\cdot \nabla(\phi^2\Delta_{h,i}\vps)+ \Delta_{h,i}\left(A'(\vps)(\nabla \vps, \nabla \vps) J_{\zinv}\right)\cdot\nabla(\phi^2\Delta_{h,i}\vps) \nonumber \\
&\geq \int_{B_{2r}} \phi^2 \Delta_{h,i}\left(\lambda^{\ep}(|\nabla\vps|^2) \nabla \vps \right)\cdot\nabla(\Delta_{h,i}\vps)-2\|\nabla \phi\|_{L^{\infty}(B_{2r})}\int_{B_{2r}} \phi |\Delta_{h,i} \vps|\,|\Delta_{h,i}\left(\lambda^{\ep}(|\nabla\vps|^2\right)| \nonumber \\
&-\left|\int_{B_{2r}} \Delta_{h,i}\left(\lambda^{\ep}(|\nabla\vps|^2\, A'(\vps)(\nabla \vps, \nabla \vps) J_{\zinv}\right)\cdot\nabla(\phi^2\Delta_{h,i}\vps)\, \right|. \label{prop44-jack}
\end{align}
Hence, the essential difference between our computations and the one presented in~\cite{duzaarfuchs} is the presence of the Jacobian of $\zinv$ in the last integral. 

Following the same lines, we obtain a counterpart of (2.4) in \cite{duzaarfuchs} for $\vps$ with expressions
$\lambda^{\ep}\left(|\nabla(\vps+s h \Delta_{h, i}\vps)|^2\right)$ for $s\in[0,1]$ substituting $|Du_{s}|^{p-2}$ used in \cite{duzaarfuchs}.

 Similar observation applies to estimates (2.6) and (2.7) which also admit analogous versions with $\lambda^{\ep}(|\nabla\vps|^2)$ in place of $|Du|^{p-2}$. In particular, the Jacobian $J_{\zinv}$ appearing in \eqref{prop44-jack} influences  only the constant denoted in \cite[Formula (2.6)]{duzaarfuchs} by $C_3$ and this constant now depends also on $\|J_{\zinv}\|_{L^{\infty}}$. Therefore, the counterpart of (2.6) a priori depends also on the choice of map $z$ and constants arising from estimates of the first partials of $\zinv$. However, when changing the coordinate system on $M$, we change values of these constants by a factor related to a supremum norm of the change of variables and its Jacobian only.
 
 Moreover, instead of the estimate appearing right after formula (2.7), in the case $\ep > 0$ we obtain the following
\[
|\Delta_{h,i} \nabla \vps|^2 \leq \frac{1}{\ep^{p-2}} |\nabla \Delta_{h,i}\vps|^2 \int_0^1 \lambda^{\ep}(|s\nabla\vps(x+h) + (1-s) \nabla\vps(x)|^2) ds.
\]
Combining this with the rest of the obtained estimates results in an upper bound for the integral of $|\Delta_{h,i} \nabla \vps|^2$ over a ball $B_r$ independent of $h$. Going to the limit with $h$ gives the required $W^{2,2}_{loc}$-regularity.

Coming back to the proof of the smoothness of $\vps$, we now repeat the arguments in Lemma 2.1 from \cite{duzaar}, obtaining a second-order linear equation for the first derivatives of $\vps$. The coefficients of this equation are only slightly different for $\ep > 0$, but the equation remains uniformly elliptic, as in~\cite{duzaar}. In a consequence, the classical regularity theory gives that the solution enjoys a higher degree of the H\"older regularity based on the regularity of the coefficients. Indeed, such coefficients are $C^\alpha$-smooth, as they involve the derivatives of the map $\vps$ and we have Theorem \ref{Thm14-Beck} at our disposal. This argument also improves the regularity of the coefficients by a degree of one and, therefore, a self-improving estimate is obtained.  Hence, $\vps$ is smooth.
\end{proof}

\section{Subharmonicity, the proof of Theorem \ref{Ttheorem}}\label{sec-subharm}

\begin{proof} Throughout this section, map $u = \ups $ denotes a stationary solution for the energy \eqref{ep-energy} with $\ep \geq 0$. In other words, $u$ is a solution of the system \eqref{u-system} for the aforementioned energy functional. By the assumptions of Theorem \ref{Ttheorem} and Proposition \ref{PosJacSmoothness} we know that the Jacobian of $u$,  $J = J_u$, is positive and that $u$ is smooth.

The proof is now based on a number of elementary but nontrivial calculations, and we follow the proof in the flat case \cite{iko}.

One may verify that the $\theta$-derivatives of mappings and functions obey the Leibniz rules:
\[
(ab)_\theta = a_\theta b + a b_\theta \ \ \ \text{ and }\ \ \ (ab)_{\theta\bt} = a_{\theta \bt} b + a_{\theta}b_{\bt} + a_{\bt}b_\theta + a b_{\theta\bt},
\]
where $a, b$ are maps or functions defined on $M$.

 We will prove that $-T^{-N_E}$ is superharmonic in $M$, meaning that $\left(-T^{-N_E}\right)_{\theta\bt} \leq 0$, i.e.,
 $$
 N_E T^{-N_E - 1}\left(T T_{\theta\bt} - (N_E+1)|T_{\theta}|^2 \right) \leq 0.
 $$
Hence is enough to show the inequality
\begin{equation}\label{ienq}
T T_{\theta\bt} - (N_E+1)|T_{\theta}|^2 \leq 0 \ \ \ \ \text{ in } M.
\end{equation}
To do this we first compute a formula for $T_{\theta\bt}$. Using Leibniz rules we see that
\begin{equation}\label{tbt}
T_{\theta\bt} = (\lambda J)_{\theta\bt} = \lambda_{\theta\bt} J + 2\Re\left(\lambda_{\theta}J_{\bt}\right) + \lambda J_{\theta\bt}.
\end{equation}

Let us now simplify terms $\lambda_{\theta\bt} J$ and $\lambda J_{\theta\bt}$. From \eqref{u-system}, we infer the following identity:
\begin{equation}\label{eq1}
2\lambda(\bar{\utb}\uttb - \ut\bar{\uttb}) = J\lambda_\theta.
\end{equation}
We apply the $\bt$-derivative on both sides of this equation, obtaining via a computation that
\[J \lambda_{\theta\bar{\theta}} = - J_{\bt} \lambda_\theta + 2\lambda_{\bar{\theta}}(\bar{\utb}\uttb - \ut\bar{\uttb}) + 2\lambda(\bar{\utb}u_{\theta\bt\bt} - \ut\bar{u_{\theta\bt\theta}}).\]

Let us make a computation for $J_{\theta\bt}$. For this we first find $(|u_{\theta}|^2)_{\theta\bt}$, beginning with the computation:
\begin{align*} d(|\ut|^2) = \bar{\ut}d\ut + \ut \bar{d\ut}
&= \bar{\ut}(\utt \theta + \uttb \bt - \ut\theta_c + \ut u^* \omega_c) + \ut\bar{(\utt \theta + \uttb \bt - \ut\theta_c + \ut u^* \omega_c)}
\\&= (\bar{\ut}\utt + \ut\bar{\uttb})\theta + (\bar{\ut}\uttb + \ut\bar{\utt})\bt.
\end{align*}
It follows that $\left(| \ut |^2\right)_{\theta} = \bar{\ut}\utt + \ut\bar{\uttb}$. Further cancellation occurs computing $(|u_{\theta}|^2)_{\theta\bt}$ as follows
\begin{align*}
& d\left(| \ut |^2\right)_{\theta} + \left(| \ut |^2\right)_{\theta} \theta_c \\
&\ = \bar{\ut}d\utt + \ut\bar{d\uttb} + \utt\bar{d\ut} + \bar{\uttb}d\ut + (\bar{\ut}\utt + \ut\bar{\uttb})\theta_c
\\&\ = \bar{\ut}\left(u_{\theta\theta\theta}\theta + u_{\theta\theta\bt}\bt + \utt u^* \omega_c - 2\utt \theta_c\right) + \ut\left(\bar{u_{\theta\bt\theta}}\bt + \bar{u_{\theta\bt\bt}}\theta - \bar{\uttb} u^* \omega_c\right) + \utt\left(\bar{\utt} \bt + \bar{\uttb}\theta + \bar{\ut}\theta_c - \bar{\ut} u^* \omega_c\right)
\\&\ \ \ \ + \bar{\uttb}\left(\utt \theta + \uttb \bt - \ut\theta_c + \ut u^* \omega_c\right) + (\bar{\ut}\utt + \ut\bar{\uttb})\theta_c
\\&\ = \left(2\Re(\utt\bar{\uttb}) + \bar{\ut}u_{\theta\theta\theta} + \ut\bar{u_{\theta\bt\bt}}\right)\theta +  \left(|\utt|^2 + |\uttb|^2 + \bar{\ut}u_{\theta\theta\bt} + \ut\bar{u_{\theta\bt\theta}}\right)\bt.
\end{align*}
This gives the formula
\[
\left(| \ut |^2\right)_{\theta\bt} = |\utt|^2 + |\uttb|^2 + \bar{\ut}u_{\theta\theta\bt} + \ut\bar{u_{\theta\bt\theta}}.
\]
Similarly
\[
\left(| \utb |^2\right)_{\theta\bt} = |\utbtb|^2 + |\uttb|^2 + \bar{\utb}u_{\theta\bt\bt} + \utb\bar{u_{\bt\bt\theta}}.
\]
Combining these, we obtain that
\[
\lambda J_{\theta\bt} = \lambda\left(|\utt|^2 - |\utbtb|^2\right) + \lambda \left(\bar{\ut}u_{\theta\theta\bt} + \ut \bar{u_{\theta\bt\theta}} - \bar{\utb} u_{\theta\bt\bt} - \utb \bar{u_{\bt\bt\theta}}\right).
\]
Upon substituting the above formulas into equation (\ref{tbt}) we obtain
\begin{align*}
T_{\theta\bt} &= \lambda\left(|\utt|^2 - |\utbtb|^2\right) + \Re\left(\lambda_\theta J_{\bt}\right) + 2\Re\left(\lambda_{\bar{\theta}}(\bar{\utb}\uttb - \ut\bar{\uttb})\right) \\&\ \ \ \ + \lambda\Re\left(\bar{\ut}u_{\theta\theta\bt} + \bar{\utb} u_{\theta\bt\bt} - \ut \bar{u_{\theta\bt\theta}} - \utb \bar{u_{\bt\bt\theta}}\right).
\end{align*}
Let us now express the terms containing the third derivatives, which we may write as
\[
\lambda \Re\left(\bar{\ut}(u_{\theta\theta\bt} - u_{\theta\bt\theta})\right) + \lambda\Re\left(\bar{\utb} (u_{\theta\bt\bt}  - u_{\bt\bt\theta})\right),
\]
in terms of lower-order derivatives of the solution. For this we apply the formulas
\[
u_{\theta\theta\bt} - u_{\theta\bt\theta} = -u_\theta \frac{K'}{2}J + u_\theta \frac{K}{2}
\quad \hbox{ and }\quad
u_{\theta\bt\bt}  - u_{\bt\bt\theta} = -u_{\bt} \frac{K'}{2} J - u_{\bt}\frac{K}{2}
\]
found in \cite[(14)]{sy} (here $K$ and $K'$ are Gauss curvatures of $M$ and $N'$, respectively). Note that the authors of \cite{sy} state that they use their formula (10), which is the 2-Harmonic equation, to obtain formula (14), but we believe this is a misprint and they actually refer to using formula (12) since one may check that their computation is actually valid for an arbitrary $C^3$-map $u$ on $M$. Hence, by recalling that $|Du|^2 = |\ut|^2+|\utetb|^2$ (see \eqref{eq-Du-norm}), we get
\begin{align*}
\lambda \Re\left(\bar{\ut}(u_{\theta\theta\bt} - u_{\theta\bt\theta})\right) + \lambda\Re\left(\bar{\utb} (u_{\theta\bt\bt}  - u_{\bt\bt\theta})\right) &= - \lambda |u_\theta|^2 \frac{K'}{2}J + \lambda |u_\theta|^2 \frac{K}{2} - \lambda|u_{\bt}|^2 \frac{K'}{2} J - \lambda|u_{\bt}|^2 \frac{K}{2}
\\ &= - \lambda |Du|^2 J \frac{K'}{2} + \lambda J \frac{K}{2}.
\end{align*}
Thus, we find that
\begin{equation}\label{eq2}
T_{\theta\bt} = \lambda\left(|\utt|^2 - |\utbtb|^2\right) + \Re\left(\lambda_\theta J_{\bt}\right) + 2\Re\left(\lambda_{\bar{\theta}}(\bar{\utb}\uttb - \ut\bar{\uttb})\right) - \lambda |Du|^2 J \frac{K'}{2} + \lambda J \frac{K}{2}.
\end{equation}
Let us now use the shorthand notation
\begin{equation*}
\left\{
\begin{array}{l}
\asc = \lambda\left(\ut\bar{\utbtb} - \bar{\utb}\utt\right)\\
\bsc = \lambda\left(\bar{\ut}\utt - \utb\bar{\utbtb}\right)\\
\lsc = \lambda\left(\ut\bar{\uttb} - \bar{\utb}\uttb\right)
\end{array} \right. 	
\end{equation*}
\emph{Claim:} \emph{The following identity holds:}
\[
TT_{\theta\bt} - |T_\theta|^2 = |\lsc|^2 - |\asc|^2 + T^2\left(\frac{K}{2}-\frac{K'}{2}|Du|^2\right).
\]
\emph{Proof of Claim:} Using the above notation, equation (\ref{eq1}) becomes $\lambda_\theta J = -2\lsc$. We may now express
\begin{align*}
T_\theta &= \lambda_\theta J + \lambda J_\theta = -2\lsc + \lambda (\bar{\ut}\utt + \ut \bar{\uttb} - \bar{\utb}\uttb - \utb\bar{\utbtb})= -2\lsc + (\bsc + \lsc) = \bsc - \lsc.
\end{align*}
In this notation equation (\ref{eq2}) yields
\[
TT_{\theta\bt} = \lambda^2 J (|\utt|^2 - |\utbtb|^2 ) + 2|\lsc|^2 - 2\Re(\lsc\bar{\bsc}) + T^2\left(\frac{K}{2}-\frac{K'}{2}|Du|^2\right).
\]
Elementary computations give that
\begin{align*}
\lambda^2 J &(|\utt|^2 - |\utbtb|^2 ) + 2|\lsc|^2 - 2\Re(\lsc\bar{\bsc}) - |\bsc - \lsc|^2 \\&= \lambda^2 J (|\utt|^2 - |\utbtb|^2 )  + |\lsc|^2 - |\bsc|^2
\\&= \lambda^2(|\ut|^2 - |\utb|^2)(|\utt|^2 - |\utbtb|^2 ) - \lambda^2 |\bar{\ut}\utt - \utb\bar{\utbtb}|^2 + |\lsc|^2
\\&= |\lsc|^2 - |\asc|^2.
\end{align*}
This proves the claim.

Let us now make the estimates to prove inequality (\ref{ienq}). Denote by $\alpha$ the expression given by
\[\alpha = \frac{|Du|^2E''(|Du|^2)}{E'(|Du|^2)} = \frac{|Du|^2E''(|Du|^2)}{\lambda}.\] In the case when $E(t) = (\ep^2 + t^2)^{p/2}$ one may compute that the expression $\alpha$ is bounded by
\begin{equation}\label{alphabounds}-1/2 < \alpha_0 < \alpha < \alpha_1 < \infty,\end{equation} where $\alpha_0, \alpha_1$ are constants.
We now first compute that
\[
T_\theta = \lambda J_\theta + \lambda_\theta J = \bsc + \lsc + J E''(|Du|^2) |Du|^2_\theta = \bsc + \lsc +\frac{\alpha}{|Du|^2}\lambda J |Du|^2_\theta.
\]
Then, we use the definition of $\lsc$ to get the identity: $\lambda J \uttb = \utb \lsc + \ut \bar{\lsc}$ and use this to obtain the formula
$\lambda J |Du|^2_\theta = |Du|^2\bsc + 2\bar{\ut}\utb \asc + |Du|^2 \lsc + 2 \ut\bar{\utb}\bar{\lsc}$. Thus, we find that
\[
T_\theta = (1+\alpha)(\bsc + \lsc) + \frac{2\bar{\ut}\utb}{|Du|^2} \alpha\asc + \frac{2 \ut\bar{\utb}}{|Du|^2}\alpha\bar{\lsc}.
\]
It follows from the identity $T_\theta = \bsc - \lsc$, that $\bsc + \lsc = 2\lsc + T_\theta$. Thus equation may also be written as
\[
-\alpha T_\theta = 2(1+\alpha)\lsc + \frac{2\bar{\ut}\utb}{|Du|^2} \alpha\asc + \frac{2 \ut\bar{\utb}}{|Du|^2}\alpha\bar{\lsc}.
\]
Due to an elementary inequality $|2\ut\utb| \leq |Du|^2$, we obtain
$|\alpha| |T_\theta| \geq (2 + 2\alpha - |\alpha|)|\lsc| - |\alpha| |\asc|$.
Define $C = \frac{|\alpha|}{2 + 2\alpha - |\alpha|}$ so that $|\lsc| \leq C|T_\theta| + C|\asc|$.
Finally, we derive the following estimate
\begin{align*}
TT_{\theta\bt} - (N_E+1)|T_\theta|^2 &= |\lsc|^2 - |\asc|^2 + T^2\left(\frac{K}{2}-\frac{K'}{2}|Du|^2\right) - N_E|T_\theta|^2
\\&\leq |\lsc|^2 - |\asc|^2 - N_E|T_\theta|^2
\\&\leq C^2 \left(|T_{\theta}| + |\asc|\right)^2 - |\asc|^2 - N_E|T_\theta|^2
\\&\leq (C^2 - N_E)|T_\theta|^2 + 2C^2|T_\theta||\asc| + (C^2 - 1)|\asc|^2.
\end{align*}
In the first inequality we use our assumption that $K \leq 0$ and $K' \geq 0$.
In order to show that this quadratic form is nonpositive everywhere, it is enough to verify two conditions: $C^2 - N_E < 0$ and $C^2 - 1 \leq \frac{C^4}{C^2 - N_E}$. If $\alpha > -1/2$, then $C \in [0,1)$. Then, by the direct computations we verify that for any such $C$ one can choose $N_E > 0$ large enough that the above conditions hold, for example $N_E>\frac{C^2}{1-C^2}$ would do. 
\end{proof}

\begin{rem}
The above superharmonicity may also be generalized to any energy functional which satisfies the bounds \eqref{alphabounds}, which includes the $p$-harmonic energy for $1 < p < 2$ for example. However, in this case one needs to either assume or prove that $u \in C^3$ for solutions $u$ with positive Jacobian, as we have only proven this regularity for the $p$-harmonic and $\ep$-energies with $p \geq 2$.
\end{rem}

\section{Geometry of $p$-harmonic maps and one-sided neighborhoods}\label{sect-geometry}

\noindent
Equipped with the regularity results of Section \ref{sect-systems}, in this section we will focus on the mapping properties of our $p$-harmonic and uniformly $p$-harmonic maps close to the boundary of $M$. The main results of this section are: (i) the maximum principle for our energy minimizers and (ii) the positivity of the Jacobian along the boundary (Proposition \ref{Prop-2}). The positivity of the Jacobian, together with the minimum principle (Corollary \ref{MinPrin}), will play an important role in the proof of Theorem \ref{MainThm}. In this section we will also reduce the problem to the case where $M$ and $N$ have smooth boundaries.

In order to state slightly more general results, we will work with the following assumptions in this section.
\begin{enumerate}
\item{The Riemannian surface $M$ (with boundary) is diffeomorphic to the closed unit disc with bounded conformal metric $\sigma$.}
\item{The Riemannian surface $N'$ is compact with no boundary with bounded conformal metric $\rho$.}
\item{The subset $N \subset N'$ is geodesically convex and is compactly contained in a small geodesic ball $B_r$.}
\item{A homeomorphism $\phi_0:\partial M \to \partial N$ with nonvanishing tangential derivative is given.}
\item{The mapping $u^{\ep}:M\to N'$ minimizes the energy  \eqref{ep-energy} (or the energy \eqref{penergy} in the case $\epsilon = 0$) among all mappings in $W^{1,p}_{ex}(M,N')$ with boundary values $\phi_0$.}
\item{When $\ep = 0$, the map $u = u^0$ is $C^{1,\alpha}$-regular up to the boundary.}
\end{enumerate}
Note that we do not assume anything about the curvatures of $M$ and $N$ unlike in our main result, Theorem \ref{MainThm}.

\smallskip

Our first aim is to establish the maximum principle for minimizers of the energies \eqref{penergy} and \eqref{ep-energy}. By a maximum principle for mappings between the manifolds $M$ and $N'$ we mean that we wish to show that no point in $M$ is mapped outside the subset $N \subset N'$. This is made possible by the fact that we fix our boundary map $\phi_0:\partial M \to \partial N$.

A key ingredient in the forthcoming proof of the maximum principle is the signed distance function to the boundary of set $N$, whose convexity we will be able to exploit. Hence, as a reminder, we briefly discuss convex functions on surfaces.

A geodesic curve $\gamma$ on a Riemannian surface $B$ with a conformal metric $\rho(\tilde{u})(d\tilde{u}^1 \otimes d\tilde{u}^1, d\tilde{u}^2 \otimes d\tilde{u}^2)$ satisfies the so-called geodesic system of equations, if
$$
\ddot \gamma^k + \sum_{i, j = 1}^2 \Gamma_{ij}^k \dot \gamma ^i \dot \gamma^j = 0, \quad k=1,2,
$$
where $\Gamma_{ij}^k$ stand for the Christoffel symbols of metric $\rho$ and $\gamma = \gamma^1 + i \gamma^2$ with $\gamma^j = (\tilde{u} \circ \gamma)^j$. One can write this system of equations in our local coordinates as
$$
\ddot \gamma = - \frac{\partial \log(\rho(\tilde{u}))}{\partial \tilde{u}} (\dot \gamma)^2 = -A(\tilde{u})(\dot \gamma)^2,
$$
where $A(\tilu)$ is related to the second fundamental form of $N'$ embedded to $\R^k$, cf. \eqref{u-system-A-fund} and Section \ref{sect31}.

Recall that $g : B \to \rr$ is a convex function, if $(g \circ \gamma)'' \geq 0$ for every geodesic curve $\gamma$ on $B$. In the local coordinates the convexity of $g$ reads
$$
\aligned
0 \leq (g \circ \gamma)'' = (g_{\gamma}\, \dot \gamma + g_{\bar{\gamma}}\, \bar{\dot \gamma})' &= g_{\gamma}\, \ddot \gamma +g_{\bar{\gamma}}\, \bar{\ddot \gamma} + g_{\gamma\gamma}\, (\dot \gamma)^2 + g_{\bar{\gamma}\bar{\gamma}}\, (\bar{\dot \gamma})^2 + 2g_{\gamma\bar{\gamma}} \left(|\dot \gamma|^2\right) \\
&=  \left(g_{\gamma\gamma} - A(\tilde{u})g_{\gamma}\right) (\dot \gamma)^2 + \left(g_{\bar{\gamma}\bar{\gamma}} - A(\tilde{u})g_{\bar{\gamma}}\right)(\bar{\dot \gamma})^2 + 2g_{\gamma\bar{\gamma}} \left(|\dot \gamma|^2\right).
\endaligned
$$
In particular, the following pointwise inequality holds for a given point $q \in B$ and fixed $\xi_1, \xi_2 \in \C$
\begin{align}\label{pisteittain}
\!\!\!\!\!\left(g_{qq} - A(\tilde{u})g_{q}\right)\!(\xi_1 + i\xi_2)(\xi_1 - i\xi_2) + \left(g_{\bar{q}\bar{q}} - A(\tilde{u})g_{\bar{q}}\right)\!\bar{(\xi_1 + i\xi_2)(\xi_1 - i\xi_2)} + g_{q\bar{q}} \left(|(\xi_1 + i\xi_2)|^2 + |\xi_1 - i\xi_2|^2\right)\!\geq\! 0.
\end{align}
Indeed, for $\gamma$ we choose the initial point in $B$, denoted $q$, and add up the convexity inequality with two different initial velocities of $\dot \gamma$, namely $\xi_1$ and $\xi_2$.

In the next observation we prove that a convex function composed with a solution to an $\ep$-perturbed $p$-harmonic system on $M$ is a subsolution of a certain elliptic operator.

\begin{lem}\label{konveksi}
Suppose $u^{\ep} \in W^{1, p}_{ex}(M, N')$ is a weak solution to the Euler--Lagrange system of equations
$$
 [\lambda^{\ep} \uz]_{\zbar}+[\lambda^{\ep} \ub]_{z}+2 \lambda^{\ep}\left(\frac{\partial}{\partial u} \log \rho(u(z))\right)
 \uz\ub=0,
$$
where $\lambda^{\ep}=\left(\ep^2+|Du^{\ep}|^2\right)^\frac{p-2}{2}$and $\ep \in [0, 1)$.
Suppose additionally that the image set $u^{\ep}(M)$ is contained in a single coordinate chart and $g : N' \to \R$ is a convex function on this chart. Then, the function $g \circ u^{\ep}$ satisfies distributionally the subelliptic inequality
\[L
 (g \circ u^{\ep}) \geq 0, \ \text{ where }\  L = 2\lambda^{\ep}\, \partial_z\partial_{\bz} + \lambda^{\ep}_z \,\partial_{\bz} + \lambda^{\ep}_{\bz}\, \partial_z = \partial_x\left(\frac12 \lambda^{\epsilon} \partial_x (\cdot) \right) + \partial_y\left(\frac12 \lambda^{\epsilon} \partial_y (\cdot)\right)
 \]
that is, $\Re \int_M L (g \circ u^{\ep}) \psi\, dV_M \geq 0$ for all $\psi\in \Sobzero (M, \C)\cap L^{\infty}(M, \C)$.
\end{lem}
\begin{proof}
By the weak formulation, see Definition~\ref{phm-weak}, it holds that
$$
 \Re  \int_{M} \lambda(z) \frac{\rho_{u^{\ep}}(u^{\ep}(z))}{\sigma(z)} (|u^{\ep}_z|^2+|u^{\ep}_{\bz}|^2)\psi+  \lambda(z) \frac{\rho(u^{\ep}(z))}{\sigma(z)}\left(\bar{u^{\ep}_z}\psi_z + \bar{u^{\ep}_{\bz}} \psi_{\overline{z}}\right)\, dV_M = 0
 $$
for all $\psi\in \Sobzero (M, \C)\cap L^{\infty}(M, \C)$. This is equivalent with
\begin{equation}\label{weakform}
 -\Re  \int_{M} \left(2\lambda^{\epsilon}\, u^{\ep}_{z\bz} + \lambda^{\epsilon}_{\bz}\, u^{\ep}_z + \lambda^{\epsilon}_{z}\, u^{\ep}_{\bz}  + 2\lambda^{\ep}A(u^{\ep})\,u^{\ep}_zu^{\ep}_{\bz}\right)\overline{\psi}\,dV_M = 0,
 \end{equation}
where  $A(u^{\ep}) = \frac{\partial}{\partial u^{\ep}} \log \rho(u^{\ep})$. We compute first
\[
(g \circ u^{\ep})_z = g_{u^{\ep}}\, u^{\ep}_z + g_{\bar{u^{\ep}}}\, \bar{u^{\ep}_{\bz}},
\]
which holds almost everywhere as $g$ is convex, and thus locally Lipschitz, and $u^{\ep}$ is locally $C^{1, \alpha}$ by Proposition~\ref{jatkuvuus}. Hence weak derivatives of $(g \circ u^{\ep})_z$ exist and we will do the rest of calculations only formally (that is, the differentiation below should be understood in the sense of distributions).
It holds,
\[
(g \circ u^{\ep})_{z\bz} = g_{u^{\ep}}\, u^{\ep}_{z\bz} +g_{\bar{u^{\ep}}}\, \bar{u^{\ep}_{z\bz}} + g_{u^{\ep}u^{\ep}}\, u^{\ep}_z u^{\ep}_{\bz} + g_{\bar{u^{\ep}}\bar{u^{\ep}}}\, \bar{u^{\ep}_z u^{\ep}_{\bz}} + g_{u^{\ep}\bar{u^{\ep}}} \left(|u^{\ep}_z|^2 + |u^{\ep}_{\bz}|^2\right).
\]
At this point of the presentation we use \eqref{weakform} for $u^{\ep}_{z\bz}$ and $\overline{u^{\ep}_{z\bz}}$ to get the following formula
\begin{align*}
2\lambda^{\ep}\,( g_{u^{\ep}}\, u^{\ep}_{z\bz} +g_{\bar{u^{\ep}}}\, \bar{u^{\ep}_{z\bz}}) &= -g_{u^{\ep}}\big(\lambda^{\ep}_{\bz}\, u^{\ep}_z + \lambda^{\ep}_z\, u^{\ep}_{\bz}+ 2\lambda^{\ep}A(u^{\ep})\,u^{\ep}_zu^{\ep}_{\bz}\big) - g_{\bar{u^{\ep}}} \big(\lambda^{\ep}_{z}\, \bar{u^{\ep}_z}  + \lambda^{\ep}_{\bz}\, \bar{u^{\ep}_{\bz}} + 2\lambda^{\ep}A(u^{\ep})\,\bar{u^{\ep}_zu^{\ep}_{\bz}}\big)
\\&= -\lambda^{\ep}_z (g \circ u^{\ep})_{\bz} - \lambda^{\ep}_{\bz} (g \circ u^{\ep})_z - g_{u^{\ep}}\, 2\lambda^{\ep}A(u^{\ep})\,u^{\ep}_zu^{\ep}_{\bz} - g_{\bar{u^{\ep}}} \,2\lambda^{\ep}A(u^{\ep})\,\bar{u^{\ep}_zu^{\ep}_{\bz}}.
\end{align*}
Thus we find
\[
L(g\circ u^{\ep}) = 2\lambda^{\ep}\left[(g_{u^{\ep}u^{\ep}} - g_{u^{\ep}}\, A(u^{\ep})) u^{\ep}_z u^{\ep}_{\bz} + (g_{\bar{u^{\ep}}\bar{u^{\ep}}} - g_{\bar{u^{\ep}}} \,A(u^{\ep})) \bar{u^{\ep}_z u^{\ep}_{\bz}} + g_{u^{\ep}\bar{u^{\ep}}} \left(|u^{\ep}_z|^2 + |u^{\ep}_{\bz}|^2\right)\right] \geq 0,
\]
where the last inequality is true for every fixed point  $z\in M$ by the pointwise inequality \eqref{pisteittain}. Indeed, in the inequality we let $q := u^\ep(z) \in u^{\ep}(M)$, $\xi_1 := u^\ep_x$ and $\xi_2 = u^\ep_y$.
\end{proof}

Since the operator $L$ above is uniformly elliptic for $\epsilon > 0$, by Lemma~\ref{konveksi} we obtain the following consequence of \cite[Theorem 8.1]{gt}.
\begin{cor}\label{convexCorollary} Let $\epsilon \in [0,1)$ and $g$ be any convex function defined in a neighborhood of the set $\ups(M)$. Then the function $g\circ \ups : M \to \rr$ satisfies the weak maximum principle in $M$, i.e.,
$$
\sup_{M} g\circ \ups \leq \sup_{\partial M} g\circ \ups.
$$
Moreover, if $\ep > 0$ then $g \circ \ups$ satisfies the strong maximum principle in $M$, i.e. if it attains its supremum in the interior of $M$ then it is constant.
\end{cor}
\begin{proof} We consider first the case $\ep > 0$. In the proof of Lemma~\ref{konveksi}, we have shown the differential inequality $L(g\circ u^{\ep})\geq 0$. Recall that the expression $\lambda^{\ep}$ is bounded from below and above by a positive constant (as $\ups(M)$ is assumed to be in a single coordinate chart, $Du^{\ep}$ is continuous up to the boundary (Corollary \ref{BeckCorollary}) and $\rho$ and $\sigma$ are bounded from below and above by our long standing assumptions). Thus, the above linear differential inequality is in fact uniformly elliptic. Then the assertion of the corollary is obtained from the strong maximum principle for strictly elliptic divergence-type inequalities \cite[Theorem 8.19]{gt}.

For $\ep = 0$ we obtain the weak maximum principle for smooth $g$ from Pigola--Veronelli \cite[Theorem 4.1]{pg2}. Extending this to general convex functions $g$ is achieved simply by approximating $g$ uniformly with smooth convex functions.
\end{proof}
\noindent We are now ready to prove the maximum principle, starting with the following weak version of it.

\begin{prop}[A very weak maximum principle]\label{weakmaxprin}
 Let $\ups$ denote a minimizer of the energy \eqref{ep-energy} for $\ep \in [0,1)$. Under the above assumptions (1)--(6) (in particular assuming that $r$ is small enough), we have that $\ups(M) \subset B_r$.
\end{prop}

\begin{proof} Note that by Proposition \ref{jatkuvuus}, we know that the map $\ups$ is $C^{1,\beta}$-regular locally in $\mathrm{int}\,M$. This will justify the possible regularity requirements for the forthcoming arguments.

As the target domain $\ups(M)$ could a priori be the whole manifold $N'$, we must split the discussion into cases depending on the topological type of $N'$.\\\\
\emph{Case 1.} $N'$ is a Riemannian surface of genus $1$ or larger.\\\\ In this case, we utilize the fact that the universal covering space of $N'$, denoted by $\covers$, is topologically either the plane or the unit disc and, hence, is covered by a single chart. Let $\tau: \covers \to N'$ denote the projection map. We endow $\covers$ with a conformal structure by pulling back the Riemannian metric from $N'$ via $\tau$. This makes the projection map $\tau$ a local isometry. Let us now lift the mapping $u^{\ep}$ into a map $u^{\ep}_{\text{lift}} : M \to \covers$, which is possible as $M$ is simply connected, see e.g. Lee~\cite[Chapter 11]{lee}. Since $N'$ and $\covers$ are locally isometric, the map $u^{\ep}_{\text{lift}}$ is a solution of the Euler--Lagrange system for the energy \eqref{ep-energy}  (respectively \eqref{penergy}, if $\ep=0$). Moreover, since ball $B_r$ is, by assumption (3), small enough, the boundary value map $\phi_0$ lifts to a map $\phi_{0,\text{lift}}$ taking values in a small ball $B_{r,\text{lift}}$ of radius $r$ in $C$.

Let $g:C\to \R$ now denote the distance function to the boundary of $B_{r,\text{lift}}$, negative inside and positive outside. Corollary \ref{convexCorollary} now implies that the function $g \circ u^{\ep}_{\text{lift}}$ satisfies the maximum principle on $M$. From this we directly obtain that $u^{\ep}_{\text{lift}}(M) \subset B_{r,\text{lift}}$. Upon taking projections we obtain that $\ups(M) \subset B_r$. \\\\
\emph{Case 2.} $N'$ is diffeomorphic to the unit sphere.\\\\
For this case, we use the fact that in Appendix B we have constructed a map $\Psi:N' \to N'$ which is the identity map on $B_r$ and a local contraction from $N' \setminus B_r$ to $B_r$. Suppose that $\ups(M)$ is not contained in $B_r$, then the map $\Psi \circ \ups$ has smaller energy than $\ups$, as $\Psi$ is a local contraction outside $B_r$. This gives a contradiction, which in turn implies that $\ups(M) \subset B_r$.
\end{proof}
Combining Proposition \ref{weakmaxprin} and Corollary \ref{BeckCorollary} now gives
\begin{cor}\label{reunalleasti} For $\ep > 0$, the minimizer $\ups$ is $C^{1,\beta}$-regular up to the boundary of $M$.
\end{cor}

Applying the above results now gives us the strong maximum principle for $p$-harmonic and uniformly $p$-harmonic mappings on surfaces.

\begin{prop}[Strong maximum principle]\label{strongmax}
Let $u^{\ep} \in W^{1, p}_{ex}(M, N')$ be a minimizer for the  $\ep$-perturbed $p$-harmonic energy $\int_{M} (\ep^2+|Du^{\ep}|^2)^{\frac{p}{2}} dV_M$ for a fixed $\ep\in [0,1)$ and $p\geq 2$ with the boundary values $u^{\ep} = \phi_0$.
The following hold under assumptions (1)-(6):
\begin{itemize}
\item[(i)] if $\ep \in (0, 1)$, then $u^{\ep}(\mathrm{int}\,M) \subset \mathrm{int}\,N$.
\item[(ii)] if $\ep = 0$, then there exists an open neighborhood $V \subset M$ of the boundary $\partial M$ such that $u(V \setminus \partial M) \subset \mathrm{int}\,N$.
\end{itemize}
\end{prop}

\begin{proof}
Proposition \ref{weakmaxprin} lets us conclude that  $u^{\ep}(M) \subset B_r$.
Let now $g : B_r \to \rr$ denote the distance function to the boundary $\partial N$, taken to have a negative sign inside of $N$ and a positive sign outside. As $\partial N$ is a $C^{1,\alpha}$-regular convex boundary by assumption, $g$ is a convex function on $B_r$.

By our regularity assumption (6) on $u$, and Corollary~\ref{reunalleasti}, we have that $u^{\ep} \in C^{1, \gamma}(M, N')$. Set $G := g \circ u^{\ep} \circ \zinv : z(\mathrm{int}\,M) \to \rr$ a $C^1$ function up to the boundary of a planar domain $z(\mathrm{int}\,M)$. Moreover, as an energy minimizer, $u^{\ep}$ is a weak solution to Euler--Lagrange equation \eqref{u-system2} and thus,  by Lemma \ref{konveksi}, $G$ is a $C^1$ distributional solution  to an elliptic differential inequality
\begin{equation}\label{lineaarineney}
 \partial_x\left(\frac12 \lambda^{\epsilon} \partial_x G\right) + \partial_y\left(\frac12 \lambda^{\epsilon} \partial_y G \right) = L(G) \geq 0,
\end{equation}
where $\lambda^{\epsilon} = \left(\epsilon^2 + |Du^{\ep}|^2\right)^{\frac{p - 2}{2}}$.

Since $u^{\ep}$ is given, we can consider it as a fixed parameter.

\smallskip

\noindent Case (i): By Corollary \ref{convexCorollary}, the function $G$ satisfies a strong maximum principle. Hence, if there is $z_0 \in z(\mathrm{int}\,M)$ such that $G(z_0) = 0 = \sup_{z(\mathrm{int} M)} G$, then $G \equiv 0$ in $z(\mathrm{int}\,M)$. As the boundary map is a homeomorphism we conclude that this can not be the case. Thus no points in the interior are mapped into the boundary.

\smallskip

\noindent Case (ii): $\ep = 0$. In this case we may not use the strong maximum principle directly since the differential inequality \eqref{lineaarineney} fails to be uniformly elliptic at the points where $|Du| = 0$, as then also $\lambda(=\lambda^{0})$ vanishes. Since the boundary data $\phi_0$ has non-vanishing tangential derivative, we know however that on $\partial M$ we have the inequality
$$
0 < |\phi_{\mathbf{t}}|^2 \leq |u_{\mathbf{n}}|^2 + |u_{\mathbf{t}}|^2 = |u_{z}|^2 + |u_{\bz}|^2 = |Du|^2.
$$ The subscript ${\mathbf{t}}$ refers to the tangential differentiation along $\partial (z(M))$ (in positive direction) and ${\mathbf{n}}$ to the outward drawn normal derivative. Thus there exists an open neighborhood $V \subset M$ of the boundary $\partial M$ such that $|Du|^2 > 0$ and hence our differential inequality \eqref{lineaarineney} is uniformly elliptic at least on the set $\zinv(V)$.

Furthermore, by Corollary \ref{convexCorollary} the function $G$ satisfies the weak maximum principle in the whole set $z(\mathrm{int}\,M)$. Thus $G$ is non-positive in this set. We may now employ the strong maximum principle in the planar domain $z(V\setminus \partial M)$, see \cite[Theorem 8.19]{gt}. Hence, if there exists $z_0 \in z(V\setminus \partial M)$ such that $G(z_0) = 0 = \sup_{z(V\setminus \partial M)} G$, then $G \equiv 0$ in $z(V\setminus \partial M)$. Since the boundary map is a homeomorphism, we conclude that no points in the interior map into the boundary.
\end{proof}

\begin{rem}
The strong maximum principle for the $p$-harmonic map $u$ and the whole surface $M$ (i.e., $u(\mathrm{int}\, M) \subset \mathrm{int}\, N$) follows from our main result Theorem \ref{MainThm}.
\end{rem}

\begin{prop}\label{Prop-2}
Let $u^{\ep} \in\! W^{1, p}_{ex}(M, N')$ be a minimizer for a  $\ep$-perturbed $p$-harmonic energy $\int_{M} (\ep^2+|Du^{\ep}|^2)^{\frac{p}{2}} dV_M$ for a fixed $\ep\in [0,1)$ and $p\geq 2$ with the boundary values $u^{\ep} = \phi_0$. Assume further that  $z(M)$ satisfies an interior sphere condition.

Then the Jacobian $J_{u^\ep}$ does not vanish on the boundary $\partial M$.
\end{prop}

Note that the $C^2$-regularity of $\partial M$ implies the interior sphere condition, while for $\alpha<1$ the $C^{1, \alpha}$ is not enough, see also an example in \cite[Section 3.2]{gt}. On the other hand, the $C^{1,1}$-regularity of the boundary is equivalent to the interior and exterior sphere conditions holding together for the given domain, see e.g. Lemma 2.2 in Aikawa--Kilpel\"ainen--Shan\-mugalin\-gam--Zhong~\cite{aksz}.

\begin{proof}
We assume without loss of generality that the boundaries $\partial M$ and $\partial N$ are positively oriented (counterclockwise). We may as well assume that the boundary homeomorphism $\phi_0 : \partial M \to \partial N$ is orientation preserving.

By Proposition \ref{weakmaxprin}, we know that $u^{\ep}(M) \subset B_r$.
Suppose on the contrary that the Jacobian of $\ups$ vanishes at some point on the boundary. Recall that both $M$ and $N'$ are equipped with the conformal coordinates, these are $z = x + iy$ and $\tilde{u} = \tilde{u}^1 + i\tilde{u}^2$ on $M$ and $B_r$, respectively. We have $J_{u^\ep} = u^1_x\, u^2_y - u^1_y\, u^2_x = u^1_{\mathbf{n}}\, u^2_{\mathbf{t}} -  u^1_{\mathbf{t}}\, u^2_{\mathbf{n}}$. Here $u^j = (\tilde{u} \circ u^\ep)^j$ and the subscript ${\mathbf{t}}$ refers to the tangential differentiation along $\partial (z(M))$ (in the positive direction) and ${\mathbf{n}}$ to the outward drawn normal derivative. Now, by our counter assumption there is $z_0 \in \partial (z(M))$ such that $J_{u^\ep}(z_0) = 0$.

We can assume, after rotation and translation if necessary, that the interior of the small open neighborhood $E \subset z(M)$ of $z_0$ lies in the lower half plane $\{z = x+iy : y < 0\}$ and $z_0 = 0$ is the highest point in $\partial (z(M) \cap E)$, whereas the image lies in the lower half plane $\{\tilde{u} = \tilde{u}^1+i\tilde{u}^2 : u^2 < 0\}$ and the image of $z_0$  is the highest point $\tilde{u}_0 = 0$  in the boundary of the image. Since
$u^2$ assumes its local maximum along $\partial (z(M))$ at the point $z_0$, we have $u^2_{\mathbf{t}}(z_0) = 0$. However, $\phi_{\mathbf{t}} = u^1_{\mathbf{t}} + iu^2_{\mathbf{t}} \neq 0$ everywhere in $\partial (z(M))$, by assumption. In particular, $u^1_{\mathbf{t}}(z_0) < 0$ (because of the orientation $u^1$ changes its sign from positive to negative). Therefore, $J_{\ups} = |\phi_{\mathbf{t}}|\, u^2_{\mathbf{n}}$ at $z_0$.

It remains to show that $ u^2_{\mathbf{n}} > 0$.
We use the same notation as in the proof of the strong maximum principle (Proposition \ref{strongmax}) and similarly extend $G$ to the boundary as a $C^1$ map $G := g \circ u^{\ep} \circ \zinv : z(M) \to \rr$ that is a $C^1$ distributional solution  to an elliptic differential inequality
$$
0 \leq L(G) = 2 \lambda^{\epsilon} G_{z\bz} + \lambda^{\epsilon}_z G_{\bz} + \lambda^{\epsilon}_{\bz} G_{z} = \partial_x\left(\frac12 \lambda^{\epsilon} \partial_x G\right) + \partial_y\left(\frac12 \lambda^{\epsilon} \partial_y G \right).
$$
We apply the boundary point lemma \cite[Theorem 1.1]{sabina} (an analog of the boundary point lemma by Finn--Gilbarg, see~\cite[Lemma 7]{finn}, for the differential inequalities) to $G : E \to [0, \infty)$ at the boundary point $z_0$.

Note that we can choose $E$ small enough (i.e. $\mathrm{int}\,E \subset V\setminus \partial M$) so that the strong maximum principle holds, cf. Proposition \ref{strongmax}. Moreover, from the proof of the strong maximum principle we infer that $|Du^{\ep}| > 0$ on $\mathrm{int}\,E$, as the boundary data has non-vanishing tangential derivative. Moreover,
$$
\lambda^{\epsilon} = \left(\epsilon^2 + |Du^{\ep}|^2\right)^{\frac{p - 2}{2}}
$$
is bounded from below and above on $\mathrm{int}\,E$ (as in the proof of the strong maximum principle)
and, furthermore, $\lambda^{\epsilon}(z)$ is H\"older continuous on $E$, as $u^\epsilon \in C^{1, \gamma}$ up to the boundary, see assumption  6. and Corollary~\ref{reunalleasti}. Thus our differential inequality is uniformly elliptic with H\"older continuous coefficients, i.e., the boundary point lemma can be applied.

Again, by the strong maximum principle $G(z) < 0 = G(z_0)$ for all $z \in \mathrm{int}\,E$. Since $z(M)$ satisfies the interior sphere condition, there is a ball $B \subset \mathrm{int}\,E$ such that $z_0 \in \partial B$. Hence, the boundary point lemma implies that the derivative with respect to the outer normal $\mathbf{n}$ is positive:
$$
G_{\mathbf n} > 0 \qquad \text{on $\partial (z(M))$}.
$$
As our target has $C^{1, \alpha}$-smooth boundary, it holds that the distance to the boundary $\partial N$ in $u^{\ep}(\zinv(E))$  is comparable to $|u^2|$ in a small enough neighborhood $E$. Hence $u^2_{\mathbf n} > 0$ at $z_0$.
\end{proof}
\noindent Recall that $u$ is our $p$-harmonic mapping. We would like to strengthen the above result to saying that $u$ is in fact, a homeomorphism in a slightly larger set than the boundary. For this, let us call a set $V$ a one-sided neighborhood of $M$ if it is a topological annulus whose outer boundary coincides with $\partial M$. Since $M$ is simply connected, $V$ lays in $M$.

\begin{prop}\label{prop67-sec6}
There exist a one-sided neighborhood $V$ of $\partial M$ such that $u: V \to u(V)$ is a homeomorphism. Moreover, the Jacobian of $u$ does not vanish in $V$.
\end{prop}
\begin{proof}
Without loss of generality we may assume that $\phi_0\,$ is orientation preserving.   By Proposition~\ref{Prop-2}, we already know that $J_{u} > 0$ along the boundary $\bdy M$ and the map $u$ is a local homeomorphism along the boundary. The rest follows by topology, and we sketch an argument here. Let us take a decreasing sequence of one-sided neighborhoods $V_n$ converging to the boundary $\partial M$. Supposing that $u$ is noninjective on each $V_n$, we find distinct points $x_n$ and $y_n$ in $V_n$ such that $u(x_n) = u(y_n)$ for all $n$. Since $u$ is a local homeomorphism along $\partial M$,  the sequences $(x_n)$ and $(y_n)$ never get close to each other. Passing to a subsequence and taking the limit gives a contradiction to the fact that the boundary map is a homeomorphism.
\end{proof}
\noindent Now it is clear that we may replace $N$ by a slightly smaller smooth convex region, and replace $M$ by the preimage of this region under $u$, which is smooth since $J_u > 0$ on $V$. Henceforth, we may assume in the subsequent sections that the boundary $\partial M\, $ is smooth and the boundary data $\phi_0\,$ is $C^\infty$-smooth diffeomorphism.

\section{Existence and uniqueness}\label{sect-ex}

\noindent
In this section we discuss the existence and the uniqueness for the Dirichlet problem of minimizing the energy \eqref{ep-energy} for mappings from $M$ to $N'$ with fixed boundary data. This serves two purposes: Firstly the existence and the uniqueness of $p$-harmonic mappings (i.e. the case $\ep = 0$) motivates the statement of our main result, Theorem \ref{MainThm}, and secondly we will use the results of this section to construct a family of solutions to \eqref{ep-energy} for $\ep > 0$ in the next section.

In particular, we must make use of the smallness condition \eqref{cond-small} imposed on the boundary of the set $N$ in the target, as the necessity of this condition for uniqueness was already mentioned in the introduction. Our proofs of existence and uniqueness are based on adapting arguments from the work of Fardoun--Regbaoui, see \cite[Proposition 3.1]{fare} (existence) and \cite[Theorem 1.1, 1.2]{fare} (uniqueness), where the case $\ep = 0$ is covered. There is a slight difference in the notion of minimization problem they consider, as they minimize only over mappings which take values in a small neighborhood. However, our very weak maximum principle, Proposition \ref{weakmaxprin}, proves that there is no loss of generality in restricting mappings to such a small neighborhood as long as the fixed boundary values are contained in this neighborhood.

\begin{prop}\label{existence}
Suppose that we are given a $W^{1, p}_{ex}(M, N')$-map $\Phi: M \to \Phi(M) =: N$ such that (S) holds, i.e., $N \subset B(P_0, r_{N',p})$ for $P_0 \in N'$ and $r_{N',p}$ small enough. For any $\epsilon \in [0, 1)$ there exists $\ups \in W^{1, p}_{ex}(M, N')$ such that $\ups$ is continuous up to the boundary of $M$ and
\begin{equation}\label{excond}
\ups(M) \subset B(P_0, r_{N',p}) \qquad \text{and} \qquad \ups = \Phi \text{ on } \partial M
\end{equation}
and $\ups\,$ minimizes the energy $\mathbb{E}_{\ep}$ in \eqref{ep-energy} among all maps in $W^{1, p}_{ex}(M, N')$.
\end{prop}

Note that since the boundary of $M$ is already assumed to be $C^2$-smooth and since we may assume that $p > 2$, by the Sobolev embedding theorem any map in $W^{1, p}_{ex}(M, N')$ is continuous up to the boundary of $M$.

\begin{proof}
 The proof follows the lines of the proof for Proposition 3.1 in \cite{fare}. For this reason we will restrict our presentation to a sketch only, for further details referring to \cite{fare}.

 First of all, we embed the space $W^{1, p}_{ex}(M, N')$ into the linear space $W^{1, p}(M, \R^k)$. Hence any minimizing sequence has a subsequence converging weakly in $W^{1,p}$ and strongly in $L^p$ to a map $i \circ\ups \in W^{1, p}(M, \R^k)$. Such a minimizing sequence may be chosen to converge pointwise almost everywhere, which shows that $\ups \in W^{1, p}_{ex}(M, N')$. The limit map $\ups$ is also continuous up to the boundary and equal to $\Phi$ there. Since the energy functional $\mathbb{E}_{\ep}(u)$ is bounded below and lower semicontinuous, the map $\ups$ is a minimizer. By the maximum principle \ref{weakmaxprin} we also get that $\ups$ satisfies the smallness condition $\ups(M) \subset B(P_0, r_{N',p})$. As a minimizer we note that $\ups$ also solves the Euler-Lagrange system \eqref{u-system-weak} with $\lambda^{\ep}$.
 \end{proof}

\begin{prop}\label{sect-exist-uni-Prop4}
Under the assumptions of Proposition \ref{existence} it holds that the minimizer in the class of $W^{1, p}_{ex}(M, N')$ satisfying \eqref{excond} is unique.
\end{prop}
\begin{proof}
The uniqueness follows from Theorem~\ref{app-thm-uniq} in Appendix A.
\end{proof}

We remark that under different assumptions on the curvature sign of $M$ and $N'$, the existence and uniqueness for $p$-harmonic mappings has recently been studied by Pigola--Veronelli, see Problem A and Theorem B in~\cite{pg1} (see also references therein).

\section{Homotopy to a conformal map}\label{sect-homotopy}

\noindent
The main goal of this section is to prove the following theorem.

\begin{theorem}\label{thm-jacob}
Assume that $M$ and $N$ have $C^\infty$-smooth boundary. Further, suppose that $\phi_0 :\bdy M\to \bdy N$ is $C^\infty$-smooth diffeomorphism with positive orientation and nonvanishing tangential derivative. Then, for any fixed $\ep>0$ a weak solution $u^{\epsilon}$ to the boundary value problem

\begin{equation}\label{yhtalot}
\begin{cases}
& [\lambda^{\ep}(z) \epuz]_{\zbar}+[\lambda^{\ep}(z) \epub]_{z}+2\lambda^{\ep}(z)\left(\frac{\partial}{\partial \ups} \log \rho(\ups(z))\right) \epuz\epub=0\quad \hbox{ in }M\\
& u^{\ep}=\phi_0 \quad \hbox{ on }\bdy M
\end{cases}
\end{equation}
is $C^\infty$-smooth and has nonvanishing Jacobian determinant in $M$. Here $\lambda^{\epsilon} = \left(\epsilon^2 + |Du^{\ep}|^2\right)^{\frac{p - 2}{2}}$.
\end{theorem}

\noindent The proof is based on constructing a smooth homotopy between the map $u^{\epsilon}$ and a conformal map and requires a careful analysis of the regularity of this homotopy up to the boundary. As such the main arguments of the proof are postponed to the end of this section, after Corollary~\ref{prop-unifnormal}.

Let $\Psi$ be any conformal mapping between ${\rm int} (M)$ and ${\rm int} (N)$. The construction of the homotopy between $u^{\epsilon}$ and $\Psi$ will be simple: First we define a homotopy between the boundary values of these two maps. After that we define a homotopy between the following systems of two equations: the Laplace equation and the Euler-Lagrange system of equations \eqref{yhtalot}, as these two systems are satisfied by the mappings $\Psi$ and $u^{\epsilon}$ respectively. Having a homotopy between the boundary values and the respective systems of equations at hand, we finally obtain a homotopy between the two mappings by solving the Dirichlet problem uniquely.

By the Carath\'eodory--Osgood--Taylor theorem extending $\Psi$ up to the boundary $\partial M$ is always possible for simply-connected Jordan domains in $\C$ and this property is preserved in the setting of Riemannian surfaces. Indeed, by assumptions $M$ and $N$ are covered by exactly one map $z$ and $\tilu$, respectively. Therefore, the extension theorem (see e.g. \cite{car1, ot}) applied to the following conformal map between Jordan domains in the complex plane:
\[
G:=\tilu|_{{\rm int} N}\circ \Psi\circ z|_{{\rm int} M}: z({\rm int}\, M)\to \tilu({\rm int}\, N)
\]
showes that $g$ possesses a homeomorphic extension, also denoted by $G$, between $z(M)$ and $\tilu(N)$ (recall in particular that both $M$ and $N$ are compact). Hence, we may define the boundary values of the conformal map $\Psi$ as follows:
\[
 \Psi_0:= \zinv \circ G=\Psi|_{\partial M}.
\]
Comparing the map $\Psi$ to map $u^\ep$, it is clear that the boundary values $\Psi_0$ and $\phi_0$ may not necessarily agree on $\partial M$, but at least we know that $\Psi$ is smooth up to the boundary and has constant sign Jacobian on $M$. Indeed, without loss of generality we may assume that $J_{\Psi}> 0$ on $M$, as we only require that $\Psi$ is conformal.

Let $\phi^t$ denote a smooth homotopy between the boundary maps $\phi^0 := \Psi_0$ and $\phi^1:=\phi_0$ such that the tangential derivative of $\phi^t$ is always nonvanishing. More precisely, since $M$ is by assumptions a surface diffeomorphic to a unit disc and $\partial M\not=\emptyset$ we may assume without the loss of generality that $\partial M$ is a closed arc, whose image in $\C$ under map $z$ we denote by $\gamma_M$. Let us define $\widetilde{\Psi_0}:\gamma_M\to \partial N$ as follows  $\widetilde{ \Psi_0}:=\Psi_0 \circ \zinv|_{\gamma_M}$. Similarly we define a map $\widetilde{\phi_0}: \gamma_M\to \partial N$ corresponding to boundary map $\phi_0$. Next, we define the associated velocity functions: $\tau_0:=\frac{\partial \widetilde{\Psi_0}(e^{is})}{\partial s}$ and $\tau_1:=\frac{\partial \widetilde{\phi_0}(e^{is})}{\partial s}$, where $s\in I$, for some closed interval $I\subset \R$. We set a homotopy between the speed functions, as follows:
\[
 |\tau_t(s)|=(1-t)|\tau_0(s)|+t|\tau_1(s)|,\quad t\in [0,1].
\]
One checks by the direct computations that $\int_{I}  |\tau_t(s)| ds=|\partial N|$, see \cite[Section 4]{iwon} for the similar reasoning. As consequence, for every $t\in [0,1]$ function  $|\tau_t|$ represents a unique diffeomorphism $\phi^t: \partial M\to \partial N$. For each $t\in[0,1]$ we denote by $\Phi^t$ a Lipschitz extension of the boundary value map $\phi^t$ into $M$ obtained from the McShane extension theorem. From now on we fix extensions $\Phi^t$ for all $t$.

\subsection{An auxiliary family of systems of equations}

Recall that by $M$ and $N$ we denote Riemannian surfaces equipped with conformal metrics $\sigma$ and $\rho$ respectively.

Fix $\ep\in(0,1)$. To define the homotopy between $u^\epsilon$ and the conformal map $\Psi$, we note that since $\Psi$ is conformal it is also $2$-harmonic. Hence, we vary the exponent from $2$ to $p$ in our defining energy integral \eqref{ep-energy}, and use the fact that the Dirichlet problem has the unique solution with given boundary data $\phi^t$ to our advantage, Proposition \ref{sect-exist-uni-Prop4}. For our given exponent $p\geq 2$ we define a homotopy from the exponent $2$ to $p$ as follows:
\begin{equation}\label{pt-est}
p_t := 2(1-t) + pt\quad \hbox{for } t\in[0,1];\quad 2\leq p_t \leq p\quad\hbox{for all } 0\leq t \leq 1.
\end{equation}

This leads us to consider the following particular cases of energy functional \eqref{ep-energy} and system \eqref{u-system-weak} for $U^t \in W^{1, p_t}_{ex}(M, N')$ with $U^t|_{\partial M} = \phi^t$:
\begin{align}
& \mathbb{E}_{\ep,p_t}(U^t)=\int_{M} \left(\ep^2+|DU^t|^2\right)^{\frac{p_t}{2}}\,dV_M, \label{Ft-system} \\
&[\lambda^{\ep}(z) U^t_z]_{\zbar}+[\lambda^{\ep}(z) U^t_{\zbar}]_{z}+\lambda^{\ep}(z)\left(\frac{\partial}{\partial U^t} \log \rho(U^t(z))\right) U^t_z U^t_{\zbar}=0, \hbox{ where } \lambda^{\ep}:=\left(\ep^2+|DU^t(z)|^2\right)^{\frac{p_t-2}{2}}. \nonumber
\end{align}
Thus, with this  notation $U^1 = u^\ep$.

\begin{rem}
 Recall, that the above mappings $U^t : M \to N'$ satisfy both of the conclusions of Propositions \ref{strongmax} and \ref{Prop-2}. In particular, it holds that $U^t(M) \subset N$ with no points in the interior being mapped to the boundary, and the Jacobian $J_{U^t}$ is positive on the boundary $\partial M$. By Corollary \ref{BeckCorollary} it holds that $U^t$ has also continuous derivatives up to the boundary. Moreover, the existence and uniqueness for $U^t$ follows from Propositions~\ref{existence}, \ref{sect-exist-uni-Prop4}.
\end{rem}

In what follows we will need the regularity properties of mappings $U^t$ up to the boundary of $M$, and for this reason we again employ results of Beck~\cite{beck} and follow the discussion from Section~\ref{sec-C1alp}. In particular, in order to do this we express the energy \eqref{Ft-system} as an energy with respect to the domain in $\C$. Hence, by using the coordinate chart $\tilu$ we set
\[
 \tilU^t=\tilu(U^t)
\]
to be map from $M\to \C$. In consequence, we arrive at  the energy integral \eqref{f-ep-energy2} and the integrand \eqref{f-ep-energy} with $s$ and $p$ corresponding now to $\tilU^t$ and $p_t$, respectively.

\subsection{Uniform H\"older and $C^{1,\alpha}$-estimates}\label{subs82}
In this section we study the properties of mappings $\tilU^t$ in more depth. First, we prove the H\"older estimates on $M$, uniform with respect to the homotopy parameter $t$. It turns out that the result of Beck plays a crucial role in such investigations, cf. Theorem~\ref{Thm14-Beck}.

Note that verifying the assumptions of Theorem~\ref{Thm14-Beck} for the energies $\mathbb{E}_{\ep,p_t}$ (cf. \eqref{Ft-system} above), with respect to mappings $\tilU^t$ and exponents $p_t$, for $t\in [0,1]$, reduces to the discussion in Section~\ref{sec-C1alp}. Indeed, by \eqref{pt-est} we have uniform estimates for $p_t$ in Theorem~\ref{Thm14-Beck} from which the estimates involving $p_t$ reduce to the similar estimates for $p$, e.g. $c(p_t)L(1+|\zeta|)^{p_t}$ can be estimated by $c(p)L(1+|\zeta|)^{p}$.

\begin{prop}\label{prop-unifsmooth}(Uniform $C^\gamma$-estimate). There exists $\gamma > 0$ such that
\[
\|\tilU^t\|_{C^{\gamma}(M)} \leq C
\]
uniformly for $t \in [0,1]$. The constant $C$ depends on the same set of parameters as the corresponding constant in the proof of Theorem~\ref{Thm14-Beck}.
\end{prop}

\begin{proof}
As in Section~\ref{sec-C1alp} we apply a change of variables $\zinv$ to the integrand of \eqref{Ft-system} (considered with respect to $\tilU^t$ and $p_t$)  in order to obtain an integral on $z(M)\subset \C$ and the corresponding map solving the Euler-Lagrange system of equations of energy $\mathbb{E_\ep}(\tilU^t)$. Then, by the proofs of Theorems 1.3 and 1.4 in \cite{beck} (cf. Theorem~\ref{Thm14-Beck} above) we get the bound for the H\"older norm of $D\tilU^t$ on $z(M)$ (and hence on $M$), see the details of the proof in \cite[Section 7]{beck}. Indeed, Step 2b of the proof of Theorem 1.3 together with Lemma 3.3 in \cite{beck} imply that
\begin{equation}\label{est1-unifsmooth}
 \|D\tilU^t\|_{C^{0,\beta}(M)}\leq C(p_t, \nu, L, \|\tilU^t\|_{L^{\infty}(M)}, \|D\tilU^t\|_{L^{q_t}(M)}),
\end{equation}
where $q_t$ stands for the higher integrability exponent for $\tilU^t$, cf. \cite[Lemma 3.3]{beck}.
Since for all $\tilU^t$ the target domain is $\tilu(N)$, we have that $\|\tilU^t\|_{L^{\infty}(M)}\leq \diam (\tilu(N))$ uniformly in $t$. The uniform  bound in $t$ for $\|D\tilU^t\|_{L^{q_t}(M)}$ requires some more work. First, note that, since $\tilU^t=\tilu(U^t)$, then
\begin{equation}\label{estDu-sect 82}
 \int_{M} |D\tilU^t|^{p_t}\,dV_M\leq \int_{M} |DU^t|^{p_t} |D\tilu|^{p_t}\,dV_M \leq \|\tilu\|_{W^{1,\infty}(M)}^{p_t}\int_{M}|DU^t|^{p_t}\,dV_M.
\end{equation}

Recall that $\Phi^t$ denote extensions of the boundary value maps $\phi^t$ into $M$ (see the beginning of Section~\ref{sect-homotopy} for a detailed description of $\phi^t$ and their extensions $\Phi^t$). By applying the Caccioppoli-type estimate (which we will prove in the next section, see \eqref{CacU-1}) we get that
\begin{equation}\label{Prop5-intpt}
 \int_{M} |DU^t|^{p_t}\,dV_M \leq c\int_{M} \left(\ep^2+|D\Phi^t|^2\right)^{\frac{p_t-2}{2}} |D\Phi^t|^2\,dV_M.
\end{equation}
By choosing the extensions $\Phi^t$ appropriately, one can see that the right hand side of \eqref{Prop5-intpt} can be controlled uniformly in $t$, cf. the introduction to Section~\ref{sect-homotopy}. Indeed, since  the boundary homotopy $\phi^t$ is assumed to be smooth on $M$, then the Lipschitz norms of $\phi^t$ are uniformly bounded in $t$  in terms of the data $\Psi_0$ and $\phi_0$. Then, the McShane extension theorem applied to maps $\Phi^t$, the Lipschitz extensions of $\phi^t$ into $M$, implies that also the Lipschitz norms of $\Phi^t$ are bounded uniformly in $t$.

Next, we apply the Gehring-type estimate in Lemma 3.3(a) in \cite{beck} in order to obtain the estimate of $\int_{M} |D\tilU^t|^{q_t}\,dV_M$ in terms of its $p_t$-energy. This, combined with \eqref{estDu-sect 82} and \eqref{Prop5-intpt}, result in the uniform in $t$ estimate for $\|D\tilU^t\|_{L^{q_t}(M)}$. In fact, despite $q_t$ depends on number of parameters, the $p_t$ is the only one where the homotopy parameter comes in. One verifies directly, that the right-hand side of the inequality in the statement of \cite[Lemma 3.3(a)]{beck} can be estimated in terms of expressions whose powers depend on $p, \gamma$ and the dimension $n=2$ only (due to the embedding of spaces $L_{loc}^{p}\hookrightarrow L_{loc}^{p_t}$ on balls). As $U^t$ is $C^{1, \beta}$-smooth, this discussion justifies writing $\|D\tilU^t\|_{L^{p}(M)}$ in estimates below.

 Furthermore, the careful analysis of constants in the proof of Theorem~\ref{Thm14-Beck} reveals that the above constant $C$ is, in fact, independent of $p_t$, again due to the uniform bound $2\leq p_t\leq p$. In a consequence, \eqref{est1-unifsmooth} holds true for $C$ independent of $t$.

 By the definition of the (intrinsic) Sobolev spaces in Section~\ref{sec-basic}, we have that $C^{\infty}(M, N)$ are dense in the Sobolev norm in $W^{1, p_t}(M, N)$. Moreover, since all $p_t\geq 2$, the Morrey embedding theorem is available. In particular, Theorem 2.8 in Hebey~\cite[Section 2.6]{hebey} applied to every component function of $\tilU^t$ allows us to conclude
that for $\beta'=1-\frac{2}{p_t}$ it holds
\[
\|(\tilU^t)^i\|_{C^{0,\beta'}}\leq C(p_t, M) \|(\tilU^t)^i\|_{W^{1, p_t}(M)},\qquad i=1,2.
\]
Notice that $0\leq \beta' \leq 1-\frac{2}{p}$. Furthermore, the dependence of constant $C$ on $p_t$ can be reduced to the dependence on $p$, as in the Euclidean case, see e.g. the proof of Theorem 7.1 in \cite{gt} and the remark following it on pg. 158 in \cite{gt}. By combining the Morrey estimate with \eqref{est1-unifsmooth}--\eqref{Prop5-intpt} and taking into account that, as already observed in the begining of the proof, $\|\tilU^t\|_{L^{\infty}(M)}\leq \diam (\tilu(N))$ uniformly in $t$ we obtain that
\[
 \|\tilU^t\|_{C^{0,\beta}(\overline{M})}\leq 4C(p, \nu, L, \diam_{\sigma} (N), \|D\tilU^t\|_{L^{p}(M)}).
\]
Here the constant $\beta$ is the minimum of the H\"older exponent in Theorem~\ref{Thm14-Beck} and $1-\frac{2}{p}$.

Hence we proved that the H\"older continuity norm of $\tilU^t: z(M)\to \C$ is uniformly bounded by an expression independent of $t$. In order to complete the proof, we notice that when coming back to the setting of mappings between $M$ and $N$ (via map $\zinv$), the above estimates change only by the factor of $\|D\zinv\|_{L^p(z(M))}$. This, however, does not affect the validity of the assertion. Thus, the proof of Proposition~\ref{prop-unifsmooth} is completed.
\end{proof}
In the next result we show the local H\"older continuity of $D\tilU^t$ uniformly with respect to the homotopy parameter $t$ for any fixed value of $\ep>0$.

\begin{prop}[$C^{1,\alpha}$-estimate up to the boundary] \label{prop-unifderiv}
There exists an exponent $\gamma' > 0$ such that the following uniform estimate holds for the family $\tilU^t$ up to the boundary of $M$:
\begin{equation}
\|D\tilU^t\|_{C^{\gamma'}(M)} \leq C.
\end{equation}
The constants $C$ and $\gamma'$ depend on the same set of parameters as the corresponding constants in the proof of Theorem~\ref{Thm14-Beck}, but are  uniform in $t$, in particular $C$ and $\gamma'$ do not explode when $p_t\searrow 2$ or $p_t\nearrow p$.
\end{prop}

\begin{proof}
 Let us analyze the steps of the proof for Theorem 1.4 in \cite{beck} and the related results to which that proof appeals to.
 Since the proof in \cite{beck} relies on a number of auxiliary results presented in \cite{beck} and refers to many other observations well-established in the literature, below we restrict ourselves only to sketching the main ideas and focus our attention on arguments showing the independence of the corresponding constants on $p_t$.

  Our goal is to explain that the following counterpart of \cite[Formula (31)]{beck} holds for mappings $\tilU^t\circ \zinv: z(M)\to\C$. In what follows we will abuse the notation and write $\tilU^t$ to denote these maps (since maps $\tilU^t$ understood as maps from $M$ to $\C$ do not appear explicitly in the proof). Denote by
 $$
  B_{r}^{+}(x_0)=\{x=(x_1, x_2)\in \C: |x-x_0|<r, x_2\geq x^2_0\},
 $$
 a half ball in $\C$, centered at point $x_0=(x^1_0,x^2_0)\in \C$ in the boundary of the image of $z(M)$, obtained by flattening the boundary, cf. the discussion in the beginning of Section~\ref{sec-basic}. Furthermore, let $V(D\tilU^t)$ stand for the following expression (cf. Lemma~\ref{lem93} below):
 \[
 V(D\tilU^t):=\big(\ep^2+ |D\tilU^t|^2\big)^{\frac{p_t-4}{2}}D\tilU^t.
 \]
 Notice that in the definition of $V$ we use the norm as defined in \eqref{eq-Du-norm}. Then, for $r\leq r_0$ it holds by \cite[Formula (31)]{beck} that:
 \begin{equation}\label{prop-Beck-form}
  \int_{B_{r}^{+}(x_0)}\Big|V(D\tilU^t)-\left(V(D\tilU^t)\right)_{r, x_0}\Big|^2\,dx\leq c(p, \sigma, \rho)r^{2(1+\lambda)},
 \end{equation}
 where $r_0$ is independent of $x_0$ and depends on $p$, $\|\tilU^t\|_{L^{\infty}(M)}<\diam N$ (for all $t$) and $\|D\tilU^t\|_{L^{q_t}}(M)$ (with $q_t$, the higher integrability exponents as in Proposition~\ref{prop-unifsmooth}). Thus, by Proposition~\ref{prop-unifsmooth} the dependence of $r_0$ on $\|D\tilU^t\|_{L^{q_t}(M)}$ can be reduced to the dependence on $\|D\tilU^t\|_{L^{p}(M)}$. As for exponent $\lambda>0$, see \eqref{prop-Beck-form}, it depends on $p$, the Sobolev conjugate exponent $p^{*}$ (set to be equal, for instance, $p+1$, as $p\geq 2=n$) and also on $\sup_{N}|\rho|$ and  $\inf_{M}|\sigma|$.

 Let us also discuss exponents $\alpha_1$ and $\alpha_2$ appearing in Beck's proof and arising from structure conditions (4) and (6b):  $\alpha_1$ can be taken as an arbitrary number in $(0,1)$, see the discussion in the paragraph before \eqref{f-ep-ass9}. As for $\alpha_2$, it arises in \cite{beck} as the modulus of continuity for term $h=h(x,u)$ in the energy integral \cite[formula (1)]{beck}. However, such a term does not appear in our case, and so $\alpha_2$ can be neglected.

 Once we show that estimate \eqref{prop-Beck-form} is independent of $p_t$, then the argument for H\"older continuity of $D\tilU^t$ follows from the standard Campanato's characterization of H\"older continuous functions and the proposition is proven (for further details we refer to pg. 821 in \cite{beck}).

 The key tools leading to \eqref{prop-Beck-form} are  Lemmas 3.3, 4.1 and 4.4, Proposition 4.2 (all stated in \cite{beck}) and the general approach from Steps 2a and 2b of the proof for \cite[Theorem 1.3]{beck}. We will discuss that the dependence on $p_t$ in constants appearing in the aforementioned tools can be reduced to dependence on $p$. Upon completing this task, one can check directly the proof \cite[Theorem 1.4]{beck} on pg. 820-821 that, due to \eqref{pt-est}, the constants in the remaining inequalities can as well be estimated by constants where dependence on $p_t$ is reduced to the dependence on $p$. The same applies to constants denoted in \cite[pg. 820-821]{beck} as $\delta, \sigma$ and $\tau_0$.

{\bf Lemma 3.3.} The discussion of the similar Gehring-type inequalities can be found, for example, in another work of Beck~\cite[Lemma 4.2]{Beck2009} and in Duzaar--Grotowski--Kronz~\cite[Lemma 3.1]{dgk}. Since the latter result is proven for the $p$-harmonic type energies, we choose its proof to discuss the uniformity of constants with respect to $t$.

By the discussion in \cite{dgk} it is enough to localize the higher integrability estimate to set, which in our case reads, $B^+:=B(0, R)\cap \{\Im z> 0\}$. Let $x_0\in B^+$ and as in \cite{dgk} we study two cases: (1) $\Im x_0\leq \frac{3}{4}r$, and (2) $\Im x_0> \frac{3}{4}r$ for $0<r<R-|x_0|$. Since the analysis of constants follows the similar approaches in both cases we discuss the first case only (see formula (15) in \cite{dgk}). Upon choosing $0<w<s\leq r$ and the test function $\eta\in C_0^{\infty}(B(x_0, s))$, $0\leq \eta \leq 1$ with $\eta|_{_{B(x_0, w)}}\equiv 1$ and $|\nabla \eta|\leq \frac{2}{s-w}$ we have that, by $\tilU^t$ minimizing the energy \eqref{Ft-system}, it holds
\begin{align*}
c(p, \rho, \sigma, \zinv) \int_{B(x_0, w)\cap B^+} |D\tilU^t|^{p_t}&\leq \int_{B(x_0, w)\cap B^+} \big(\ep^2+|D\tilU^t|^2\big)^{\frac{p_t}{2}}\,|J_{\zinv}| \\
& \leq \int_{B(x_0, s)\cap B^+} \big(\ep^2+\big|D(\tilU^t-\eta(\tilU^t-\Phi^t))\big|^2\big)^{\frac{p_t}{2}}\,|J_{\zinv}|.
\end{align*}
Recall that $\Phi^t$ stand for the fixed Lipschitz extensions of the corresponding maps $\phi^t$ (see Section~\ref{sect-homotopy} and the discussion following \eqref{Prop5-intpt}). Next, one applies the classical Young inequality $a^{p_t}\leq 3^{p_t}(b^{p_t}+c^{p_t}+d^{p_t})$ for $a = b + c + d$, $b,c,d\geq 0$ to the right hand side, together with Lemma 2.1 in \cite{dgk} to obtain a counterpart of estimate (13) in \cite{dgk}:
\begin{align*}
 c(p, \rho, \sigma, \zinv) \vint_{B(x_0, \frac{r}{2})\cap B^+} |D\tilU^t|^{p_t}  & \leq c(p^t, \zinv)\left(\vint_{B(x_0, r)\cap B^+} \frac{|\tilU^t-\Phi^t|^{p_t}}{r^{p_t}}+ \vint_{B(x_0, r)\cap B^+} |D\Phi^t|^{p_t}+\ep^{p_t}|B^+|\right).
\end{align*}
Finally, we apply the Sobolev inequality to the first term on the right-hand side together with the H\"older inequality and arrive at the counterpart of (14) in \cite{dgk}. The arising constants are, upon the increment, independent on $t$ (by \eqref{pt-est}). In consequence, \cite[Theorem 2.4]{dgk} gives us a counterpart of the assertion of \cite[Lemma 3.1]{dgk}.

{\bf Lemma 4.1.} For the proof of this lemma,  \cite{beck} appeals to Lemma 3.2 in \cite{dgk}. However, the proof of that result is a direct consequence of the just described counterpart of Lemma 3.1 in \cite{dgk}.  In particular, assumptions (17) in \cite{dgk} are substituted by the analogous growth conditions on the integrand in the Dirichlet energy~\eqref{Ft-system}. Moreover, $h:=\Phi^t$ and $\overline{q}\in (p,\infty)$ can be chosen arbitrary. As in the case of Lemma 3.3, all constants do not depend on $t$.

 Let us now focus our attention on {\bf Lemma 4.4}. The exponent $p_t$ appears in a constant, denoted by Beck, $c(p,\nu)$, cf. pg. 806 in \cite{beck}. This constant arises from geometric inequalities for $p$-harmonic-type vector fields and can be estimated from below in terms of $p$ via \eqref{pt-est}. Then, constants $\nu$ and $L$ appearing in \cite[formula (13)]{beck} are in our case expressed in terms of $p$ and the geometry of $\rho$ and $\sigma$. The remaining part of the proof of Lemma 4.4 relies on Lemmas 3.3 and 4.1 and estimates where the dependence of exponents and constants on $p_t$ is reduced to the dependence on $p$.

 The proof of {\bf Proposition 4.2} follows for $p\geq 2$ from Theorems 3.I, 6.II (also 6.I) in Campanato~\cite{camp}. The careful scrutiny of proofs of these results reveals that in all estimates dependence on $p_t$ can be reduced to the dependence on $p$ uniformly.

 Finally, the proof of Theorem 1.4 in \cite{beck} appeals to {\bf Steps 2(a)} and {\bf 2(b)} of \cite[Theorem 1.3]{beck}. Since, $\dim M=2\leq p_t$ for all $t\in[0,1]$, Step 2(b) is crucial. There, the dependence of constants on $t$ comes through \cite[estimate (25)]{beck} and Step 2(a). The estimate relies on Proposition 4.2 and Lemma 4.4 already discussed above to give the uniform estimates in $t$. The same applies to Step 2(a).
\end{proof}

By using Propositions~\ref{prop-unifsmooth} and \ref{prop-unifderiv} we show the following result.

\begin{prop}\label{prop-smoothfamily}
The mapping family $U^t : [0,1] \times M \to N$ is $C^{1,\gamma''}$-regular in $M$ for some $\gamma''>0$ depending on the same set of parameters as the H\"older exponent in Proposition~\ref{prop-unifderiv}.
\end{prop}
\begin{proof}
Let $t\in [0,1]$ and suppose first that $\{t_n\}_{n=1}^\infty$ is a sequence of points in $[0,1]$ such that $t_n \to t$. Since mappings $U^s$, for $s=t$ and $s=t_n$, correspond to $\tilU^s$ via $\tilU^s=\tilu(U^s)$, we may apply Propositions~\ref{prop-unifsmooth} and~\ref{prop-unifderiv} to $U^s$ by composing them with $\tilu$. Then, the resulting estimates change only by a constant depending on the geometry of $\tilu$.

By Proposition \ref{prop-unifsmooth}, the family $\{U^{t_n}\}_{n=1}^\infty$ is equicontinuous on $M$ and uniformly bounded. Therefore, the normal family argument implies that for a subsequence $\{t_{n_k}\}_{k=1}^{\infty}$ the sequence of mappings $\{U^{t_{n_k}}\}_{k=1}^{\infty}$ converges uniformly in $M$ to a map denoted $F$.

Let us now recall that Proposition \ref{prop-unifderiv} gives an uniform bound for the derivatives of the mappings $\{U^{t_{n_k}}\}_{k=1}^{\infty}$ of the form  $\|DU^t\|_{C^{\gamma'}(M)} \leq C$. By choosing $\gamma'' \leq \gamma$ if necessary, the application of the Arzela-Ascoli theorem shows that one may choose a further subsequence of this family whose derivatives converge in the $C^{\gamma''}$-norm. Hence this subsequence, which we still denote by $(U^{t_{n_k}})$, converges in $C^{1,\gamma''}$-norm and the limiting map remains to be $F$. Since the derivatives of $(U^{t_{n_k}})$ converge to the derivatives of $F$ uniformly in $M$, we may pass to the limit in the weak formulation of the $p_t$-harmonic system \eqref{Ft-system} and see that $F$ solves the same weak system as the map $U^t$. Furthermore, since $F$ and $U^t$ have the same boundary values we must have $F = U^t$ by the uniqueness, Proposition~\ref{sect-exist-uni-Prop4}, applied to the boundary data $\phi^t$ and exponent $p_t$. By the above reasoning, every subsequence of $(U^{t_n})$ must have a subsequence converging to $U^t$. This proves that $U^{t_n} \to U^t$ uniformly in $C^{1,\gamma''}(M)$, as $t_n\to t$.
\end{proof}

\begin{cor}\label{prop-unifnormal}(Uniform boundary estimate). For some $c_1 > 0$ which does not depend on $t$ we have that
\[J_{\tilU^t}(z) \geq c_1 \ \ \text{ for all } z \in \partial M.\]
\end{cor}
\begin{proof} By Proposition \ref{prop-smoothfamily}, we know that the family $\{\tilU^t : t \in [0,1]\}$ is $C^{1,\gamma''}$-smooth up to the boundary of $M$. Hence the Jacobians $J_{\tilU^t}$ form a continuous family in $M$. Since $\partial{M}$ is compact and the Jacobian $J_{\tilU^t}$ is positive along the boundary (Proposition \ref{Prop-2}), this implies an uniform lower bound by continuity and compactness.
\end{proof}

We are finally in a position to complete the proof of Theorem~\ref{thm-jacob}.

\begin{proof}[Proof of Theorem~\ref{thm-jacob}]
The smoothness of mappings $\ups$ follows from Theorem~\ref{Thm14-Beck} and Proposition~\ref{PosJacSmoothness}.

In order to show the positivity of Jacobians $J_{\ups}$ in $M$, let us introduce the following set
\[
S = \{t \in [0,1] : J_{\tilU^t} > 0 \text{ everywhere in } M\}.
 \]
We prove that $S$ is both open and closed on the interval $[0,1]$. In conclusion, since $0 \in S$ we must also have that $1 \in S$, proving the assertion that $J_{u^\epsilon} > 0$, cf. the definition of the homotopy $\tilU^t$ above. Therefore, the proof of Theorem \ref{thm-jacob} will be  concluded, provided that we show the following claim.
\smallskip

\noindent \emph{Claim.} There exists a constant $c_2 > 0$ such that if $t \in S$, then in fact $J_{\tilU^t} > c_2$ in $M$.
\smallskip

\noindent \emph{Proof of the claim.} The proof is based on the minimum principle for the expression $T$, see Corollary \ref{MinPrin}. We obtain that
\[\inf_{z \in M} (\epsilon^2 + |D\tilU^t(z)|^2)^{\frac{p_t-2}{2}} J_{\tilU^t}(z) \geq \inf_{z \in \partial M} (\epsilon^2 + |D\tilU^t(z)|^2)^{\frac{p_t-2}{2}} J_{\tilU^t}(z) \geq \epsilon^{p_t - 2} c_1.\]
Proposition \ref{prop-unifderiv} also gives the uniform upper estimate
\[(\epsilon^2 + |D\tilU^t(z)|^2)^{\frac{p_t-2}{2}} \leq C_3 < \infty \ \ \text{ for all } z \in M.\]
Thus we have the uniform estimate $J_{\tilU^t}(z) \geq \epsilon^{p_t - 2} c_1/C_3 > 0$ for all $z \in M$. This proves the claim.
\medskip

\noindent \emph{The set $S$ is open.} Suppose now that $t_0 \in S$. Then $J_{\tilU^{t_0}}(z) \geq c_2$ in $M$. By smoothness,  $J_{\tilU^{t}}(z) \geq c_2/2 > 0$ for $t$ close to $t_0$ and $z \in M$. This proves that $S$ is open.
\medskip

\noindent \emph{The set $S$ is closed.} Suppose that $(t_n) \subset S$, and $t = \lim_{n\to\infty} t_n$. By Proposition \ref{prop-unifderiv}, we infer that
\[J_{\tilU^{t}}(z) = \lim_{n\to\infty} J_{\tilU^{t_{n}}}(z) \geq c_2 > 0,\]
for all $z$ in $M$. Thus $t \in S$ and the set $S$ is closed.
\end{proof}

\section{Convergence results for $\ep$-perturbed $p$-harmonic systems}\label{sect9}

\subsection{The Caccioppoli-type estimates}\label{sect91}

The purpose of this section is to provide various energy estimates for weak solutions of the $\ep$-perturbed $p$-harmonic system. Such estimates are employed in Section~\ref{subsect-conv} to show the $L^p$ convergence of $Du^{\ep}$ to $Du$ and the uniform convergence of $\det Du^{\ep}$
on compacta. We follow the ideas from Lemmas 11.1, 11.2, 11.3 in \cite{iko}, as well as adapting techniques from \cite{fare}.

Since the computations in \cite{fare} are based on the Nash embedding of the target manifold into the Euclidean space, we also prefer to work in this setting here. Hence, let $i : N' \to \rr^k$ denote the isometric Nash embedding of $N'$, and set $\vps := i \circ \ups$. Thus $\vps$ satisfies the system \eqref{v-system-weak} for $\lambda = \lambda^{\ep}$. Here we also identify the space $W^{1,p}_{ex}(M,N')$ with the space of Sobolev maps from $M$ to $\rr^k$ taking values in $i(N')$ almost everywhere, and denote by $W^{1,p}_{ex}(M,N)$ the subspace of maps taking values only in $i(N)$.

\begin{lem}\label{lem91}
  Let $v^{\ep} = i \circ \ups$, where $\ups$ denotes a minimizer of the $\ep$-perturbed $p$-harmonic energy \eqref{ep-energy}, including $\ep=0$, and $i: N' \to \rr^k$ is the Nash embedding. In particular $\vps$ is a weak solution to the system of equations~\eqref{v-system-weak}.
  Let also $v_0\in W^{1,p}_{ex}(M,N)$ be a given Sobolev map with the same trace as $\vps$ on $\partial M$. Then the following Caccioppoli-type estimates hold for $v^{\ep}$ and $v_{0}$:
  \begin{equation}\label{Cac1}
  \int_{M} \left(\ep^2+|\nabla  \vps|^2\right)^{\frac{p-2}{2}} |\nabla\vps|^2\,dV_M \leq c\int_{M} \left(\ep^2+|\nabla v_0|^2\right)^{\frac{p-2}{2}} |\nabla v_0|^2\,dV_M,
 \end{equation}
 provided that the submanifold $N$ in the target satisfies the smallness condition $\diam N \leq \ep_{N',p}$ for $\ep_{N',p}$ sufficiently small. The constant $c$ in \eqref{Cac1} depends only on $p, N',$ and $\ep_{N',p}$.
\end{lem}

\begin{proof}
 Recall from the maximum principles in Section \ref{sect-geometry} that the maps $\vps$ are already known to take values inside $i(N)$. Let $v_0\in W^{1,p}_{ex}(M,N)$ be a given map and let us define a test mapping $\phi:= \vps - v_0$ with zero trace in $W^{1,p}_{ex}(M,N)$. We use $\phi$ and $\lambda^{\ep}:=\left(\ep^2+|\nabla \vps|^2\right)^{\frac{p-2}{2}}$ at \eqref{v-system-weak} and arrive at the following estimate:
 \begin{equation*}
 - \int_{M} \lambda^{\ep} \left(|\nabla \vps|^2 - \nabla \vps \cdot \nabla v_{0} \right) \,dV_M =\int_{M} \lambda^{\ep} A'(\vps)(\nabla \vps,\nabla \vps)\cdot(\vps- v_0)\,dV_M.
 \end{equation*}
 From this we get
 \begin{equation*}
  \int_{M} \lambda^{\ep} |\nabla \vps|^2 \,dV_M \leq \int_{M} \lambda^{\ep} |A'(\vps)(\nabla \vps,\nabla \vps)| |\vps-v_{0}|\,dV_M+
  \int_{M} \lambda^{\ep} |\nabla \vps||\nabla v_{0}|.
 \end{equation*}
 In order to estimate the first term on the right-hand side of this inequality, we recall that the second fundamental form $A'(\vps)$ is defined as the bilinear form, and hence,  $|A'(\vps)(\nabla \vps,\nabla \vps)|\leq C_{N'}|\nabla \vps|^2$. The smallness assumption on the set $N$ guarantees that $\dist(\vps,v_0) \leq \diam\, i(N) = \diam N \leq \ep_{N',p}$, implying the estimate $|\vps - v_{0}| \leq \ep_{N',p}$.
Combined with the inequality \cite[(11.5)]{iko} to treat the second term above, we obtain as a consequence that
 \begin{equation}\label{Cac1-aux2}
  \int_{M} \lambda^{\ep} |\nabla \vps|^2 \,dV_M \leq C_{N'} \ep_{N',p} \int_{M} \lambda^{\ep} |\nabla \vps|^2 dV_M+
  \frac12\int_{M} \lambda^{\ep} |\nabla \vps|^2\,dV_M+2^p\int_{M} \left(\ep^2+|\nabla v_0|^2\right)^{\frac{p-2}{2}} |\nabla v_0|^2\,dV_M.
 \end{equation}
Upon rearranging, we find that \eqref{Cac1-aux2} takes the following form:
 \begin{equation*}
  \int_{M} \lambda^{\ep} |\nabla \vps|^2 \left(\frac12-C_{N'}\ep_{N',p}\right)\,dV_M \leq 2^p\int_{M} \left(\ep^2+|\nabla v_0|^2\right)^{\frac{p-2}{2}} |\nabla v_0|^2\,dV_M.
 \end{equation*}
 From this, assertion \eqref{Cac1} follows immediately with $c:=\frac{2^p}{\frac12-C_{N'}\ep_{N',p}}$, provided $\ep_{N',p}$ is sufficiently small.

 By taking $\phi:= \vps - i \circ u^0$ with zero trace in $W^{1,p}_{ex}(M,N)$, where $u^0$ denotes the $p$-harmonic map $u^{\ep}$ for $\ep=0$, we obtain in particular the assertion of Lemma \ref{lem91} with $v^0 = i \circ u^0$ instead of $v_0$ on the right-hand side of \eqref{Cac1}. Note that here we use the fact that the maps $v^0$ and $\vps$ for $\ep > 0$ have the same trace on $\partial M$.
 \end{proof}

 Since for any map $u_1\in W^{1,p}(M,N)$ the norm equality $|D u_1| = |\nabla (i \circ u_1)|$ holds, we obtain the following corollary as a consequence of Lemma \ref{lem91}.

 \begin{cor}\label{cor91}
  Let $\ups$, $\ep \geq 0$, denote a minimizer of the $\ep$-perturbed $p$-harmonic energy as before, and $u_0\in W^{1,p}(M,N)$ be a given Sobolev map with the same trace as $\ups$ on $\partial M$. Then we have the Caccioppoli-type inequality:
  \begin{equation}\label{CacU-1}
  \int_{M} \left(\ep^2+|D  \ups|^2\right)^{\frac{p-2}{2}} |D\ups|^2\,dV_M \leq c\int_{M} \left(\ep^2+|D u_0|^2\right)^{\frac{p-2}{2}} |D u_0|^2\,dV_M,
 \end{equation}
 The constant $c$ here depends only on $p, N',$ and $\ep_{N',p}$, and we again assume the smallness condition \eqref{cond-small} on $N$.
\end{cor}

\subsection{Convergence of $\nabla \vps$ and the Jacobian $J_{\vps}$}\label{subsect-conv}

 In this section we employ some of the energy estimates presented above to show two convergence lemmas. First, in Lemma~\ref{lem92} we discuss the $L^p$ convergence on $M$ for the differentials $\nabla v$ of weak mappings solving the $\ep$-perturbed $p$-harmonic systems \eqref{v-system-weak} and prove that $\nabla v$ can be $L^p$-approximated globally by $\nabla \vps$. Then we show the $W^{1,2}$-Sobolev regularity for an auxiliary differential expression, see Lemma~\ref{lem93}. By combining results of this section together with the Gehring-type reverse H\"older inequality we arrive at one of the key-results of this note, namely, at Corollary~\ref{cor-det-conv}. It says that on compact subsets of $M$ we have the uniform convergence:
$$
 J_{\vps}\to J_v,\qquad \hbox{for } \ep\to 0.
$$

We use this observation to complete the proof of Theorem~\ref{MainThm} in Section~\ref{Sec10}.

\begin{lem}\label{lem92}
 Under the above notation and definitions, let $v$ (defined as $v = i \circ u$ where $u = u^0$) and $\vps$ be solutions to the boundary value problem with the same trace $v_0\in W^{1,p}_{ex}(M,N)$. Then it holds that
 \begin{equation*}
  \|\nabla \vps - \nabla v\|_{L^p(M)}\to 0,\quad\hbox{ as }\ep\to 0.
 \end{equation*}
\end{lem}
\begin{proof}
 For a given $\ep\geq 0$ let us take the following test function $\phi:=\vps - v$. We may apply $\phi$ in the weak equation \eqref{v-system-weak} by the density of $C_{0}^{\infty}(M, \rr^k)$ in $\Sob(M, \rr^k)$ and by the fact that both $\vps$ and $v$ have the same boundary data $v_0$. Therefore, we use $\phi$ in \eqref{v-system-weak} for both of the maps $v$ and $\vps$ and subtract the two equations from each other (as in the proof of \cite[Lemma 11.2]{iko}), to obtain:
 \begin{equation}\label{lem-conv0}
 - \int_{M} \left(\lambda^{\ep} \nabla\vps \cdot \nabla \phi -  \lambda \nabla v \cdot \nabla \phi \right)\,dV_M =\int_{M} (\lambda^{\ep} A'(\vps)(\nabla \vps,\nabla \vps)-\lambda  A'(v)(\nabla v, \nabla v))\cdot \phi  \,dV_M.
 \end{equation}

 We follow the steps of the proof for Lemma 11.2 in \cite{iko} and obtain the estimate similar to (11.8) in \cite{iko}:
 \begin{align}
  \|\nabla \vps - \nabla v\|^p_{_{L^p(M)}} & \leq\! 2^{\frac{p-1}{2}}\!\left[\int_{M} \left(\ep^2+|\nabla \vps|^2+|\nabla v|^2\right)^{\frac{p}{2}}\right]^{\frac12}\! \left[\int_{M} \left(\ep^2+|\nabla \vps|^2+|\nabla v|^2\right)^{\frac{p-2}{2}}\,\left|\nabla \ups-\nabla v\right|^2\right]^{\frac12}.  \label{lem-conv1}
 \end{align}

Let us estimate the first factor on the right-hand side of \eqref{lem-conv1}.
\begin{equation}
\int_{M} \left(\ep^2+|\nabla \vps|^2+|\nabla v|^2\right)^{\frac{p}{2}}\leq 2^{\frac{p}{2}}
\left(\int_{M} \left(\ep^2+|\nabla \vps|^2\right)^{\frac{p}{2}}+\int_{M} |\nabla v|^p \right).
\label{lem-1st-rhs}
\end{equation}
Since by assumptions $p>2$, then $\frac{p}{p-2}>1$ and by the $\delta$-Young inequality and \eqref{Cac1} applied above we obtain that
\begin{align}
\int_{M}\!\!\left(\ep^2+|\nabla \vps|^2\right)^{\frac{p}{2}}\!dV_M&=\ep^2\!\int_{M} \left(\ep^2+|\nabla \vps|^2\right)^{\frac{p-2}{2}}\,dV_M+\int_{M} \left(\ep^2+|\nabla \vps|^2\right)^{\frac{p-2}{2}} |\nabla \vps|^2\,dV_M \label{lem-conv2} \\
&\leq\! \int_{M}\! \left\{\left(\ep^2+|\nabla \vps|^2\right)^{\frac{p}{2}}\delta^\frac{p}{p-2}+\frac{\ep^p}{\delta^{\frac{p}{2}}}\right\}\!dV_M+ c(p, C_N, \ep_{N',p})\!\int_{M} \left(\ep^2+|\nabla v_0|^2\right)^{\frac{p-2}{2}} |\nabla v_0|^2\,dV_M. \nonumber
\end{align}
We choose $\delta=2^{\frac{2-p}{p}}$ and include the first of the above integrals on the right-hand side into the left-hand side of \eqref{lem-conv2}. As $\ep\leq 1$, we obtain that
\begin{align*}
\int_{M} \left(\ep^2+|\nabla \vps|^2\right)^{\frac{p}{2}}\,dV_M
 \leq 2^{\frac{p}{2}}\ep^p|{\rm vol} M|+ c(p, C_N, \ep_{N',p})\int_{M} \left(\ep^2+|\nabla v_0|^2\right)^{\frac{p}{2}}\,dV_M  \leq C\int_{M} (1+|\nabla v_0|^{2})^{\frac{p}{2}}.
\end{align*}
The constant $C$ depends on $2^{\frac{p}{2}}$, ${\rm vol} M$ and $c(p, C_{N'}, \ep_{N',p})$. In summary, $C$ depends only on $p$ and the geometry of surfaces $M$ and $N'$. This completes the estimate of \eqref{lem-1st-rhs}.

The estimate of the crucial second factor in \eqref{lem-conv1} is more tedious and difficult. By the standard $p$-harmonic estimate, see e.g. (11.7) in \cite{iko}, the definition of the inner product and \eqref{lem-conv0}, we have
\begin{align}
&\int_{M} \left(\ep^2+|\nabla \vps|^2+|\nabla v|^2\right)^{\frac{p-2}{2}}\,\left|\nabla \vps-\nabla v\right|^2\,dV_M \nonumber \\
&\leq c(p) \int_{M} \left \langle \left(\ep^2+|\nabla \vps|^2\right)^{\frac{p-2}{2}}\nabla \vps-\left(\ep^2+|\nabla v|^2\right)^{\frac{p-2}{2}}\nabla v, \nabla \vps-\nabla v\right \rangle \,dV_M\label{lemm92-ineq1} \\
& = c(p) \left(- \int_{M} (\lambda^{\ep} A'(\vps)(\nabla \vps,\nabla \vps)-\lambda  A'(v)(\nabla v,\nabla v))\cdot (\vps-v) \,dV_M  \nonumber \right.\\
& \left.\phantom{AA}+ \int_{M} \left \langle |\nabla v|^{p-2}\nabla v-\left(\ep^2+|\nabla v|^2\right)^{\frac{p-2}{2}}\nabla v, \nabla \vps-\nabla v\right \rangle \,dV_M \right) \nonumber \\
& \leq c(p) \left(- \int_{M} \left(\left(\ep^2+|\nabla \vps|^2\right)^{\frac{p-2}{2}} A'(\vps)(\nabla \vps,\nabla \vps)-|\nabla v|^{p-2} A'(v)(\nabla v,\nabla v)\right)\cdot (\vps-v) \,dV_M \right. \label{lem92-aux1}  \\
&\left.\phantom{AA}+ \int_{M} \left||\nabla v|^{p-2}-\left(\ep^2+|\nabla v|^2\right)^{\frac{p-2}{2}}\right||\nabla v||\nabla \vps-\nabla v| \,dV_M \right). \label{lem92-aux2}
\end{align}

{\bf Estimates for integral \eqref{lem92-aux1}:}

According to Lemma 2.2 in \cite{fare} the following estimate holds for any point $y,z\in \rr^k$ and any vectors $Y,Z\in \rr^k$ with a constant $C=C(k)$:
\begin{equation}\label{app-A-est}
|A'(y)(Y,Y)-A'(z)(Z,Z)|\leq C(|Y|^2+|Z|^2)|y-z|+C(|Y|+|Z|)|Y-Z|.
\end{equation}

Let us apply observation \eqref{app-A-est} for
\[
y=\vps,\quad z=v, \quad Y=\left(\ep^2+|\nabla \vps|^2\right)^{\frac{p-2}{4}}\nabla \vps, \quad Z=|\nabla v|^{\frac{p-2}{2}}\nabla v.
\]
This, together with the bilinearity of the second fundamental form $A'$ allow us to estimate from the above the integral in \eqref{lem92-aux1} by the following expression:
\begin{align}
&\int_{M} \left|A'(\vps)\left(\left(\ep^2+|\nabla \vps|^2\right)^{\frac{p-2}{4}}\nabla \vps,\left(\ep^2+|\nabla \vps|^2\right)^{\frac{p-2}{4}}\nabla \vps\right)- A'(v)\left(|\nabla v|^{\frac{p-2}{2}}\nabla v,|\nabla v|^{\frac{p-2}{2}}\nabla v\right)\right|\,|\vps-v|\,dV_M \nonumber \\
&\leq C\int_{M}\left(\left(\ep^2+|\nabla \vps|^2\right)^{\frac{p-2}{2}}|\nabla \vps|^2+ |\nabla v|^p \right)|\vps-v|^2\,dV_M \label{lem92-int3}\\
&\phantom{AA}+C\int_{M}\left(\left(\ep^2+|\nabla \vps|^2\right)^{\frac{p-2}{4}}|\nabla \vps|+|\nabla v|^{\frac{p}{2}}\right)
\left|\left(\ep^2+|\nabla \vps|^2\right)^{\frac{p-2}{4}}\nabla \vps-|\nabla v|^{\frac{p-2}{2}}\nabla v\right||\vps-v|\,dV_M.
\label{lem92-int4}
\end{align}

In order to estimate the above two integrals we now need to apply Proposition \ref{app-prop} from the appendix. We apply the proposition twice by choosing $\epsilon$ in the proposition to be $\epsilon>0$ and zero, respectively. Moreover, we choose as our test function $\eta = \vps - v$ and thus obtain the estimate
\begin{equation*}
  \int_{M} \left((\ep^2+|\nabla \vps|^2)^{\frac{p-2}{2}}|\nabla \vps|^2 + |\nabla v|^p\right)|\vps - v|^2\,dV_M \leq 16 r^2 \int_{M} \left((\ep^2+|\nabla \vps|^2)^{\frac{p-2}{2}} + |\nabla v|^{\frac{p-2}{2}}\right)|\nabla \vps - \nabla v|^2\,dV_M.
 \end{equation*}
We apply this directly to \eqref{lem92-int3}, while for \eqref{lem92-int4} we use the Cauchy-Schwartz inequality and the elementary inequalities $(x+y)^2 \leq 2(x^2 + y^2)$ and $2xy \leq x^2 + y^2$ to obtain
\begin{align}
 & \int_{M}\left(\left(\ep^2+|\nabla \vps|^2\right)^{\frac{p-2}{4}}|\nabla \vps| + |\nabla v|^{\frac{p}{2}}\right)
\left|\left(\ep^2+|\nabla \vps|^2\right)^{\frac{p-2}{4}}\nabla \vps-|\nabla v|^{\frac{p-2}{2}}\nabla v\right||\vps-v|\,dV_M
\nonumber\\&\leq\! 4\sqrt{2}r \Bigg(\!\int_{M}\bigg|\left(\ep^2+|\nabla \vps|^2\right)^{\!\frac{p-2}{4}}\nabla \vps-|\nabla v|^{\frac{p-2}{2}}\nabla v\bigg|^2\!dV_M\!\Bigg)^{\!\frac12}\!\!\left(\int_{M} \left((\ep^2+|\nabla \vps|^2)^{\!\frac{p-2}{2}}+ |\nabla v|^{\frac{p-2}{2}}\right)|\nabla \vps - \nabla v|^2\!dV_M\!\right)^{\!\frac12}\nonumber\\&\leq\! 2\sqrt{2} r\! \int_{M}\left|\left(\ep^2+|\nabla \vps|^2\right)^{\!\frac{p-2}{4}}\nabla \vps-|\nabla v|^{\frac{p-2}{2}}\nabla v\right|^2\!dV_M + 2\sqrt{2} r\! \int_{M}\! \left((\ep^2+|\nabla \vps|^2)^{\frac{p-2}{2}}+ |\nabla v|^{\frac{p-2}{2}}\right)|\nabla \vps - \nabla v|^2\,dV_M.
\label{lem92someEstim}
 \end{align}
We now estimate the first term in the above expression \eqref{lem92someEstim}. By the elementary inequality
\[\left|(\ep^2 + |x|^2)^{\frac{p-2}{4}}x - (\ep^2 + |y|^2)^{\frac{p-2}{4}}y\right| \leq C_p (\ep^2 + |x|^2 + |y|^2)^{\frac{p-2}{4}} |x-y|, \qquad \text{ for } x,y \in \rr^k,\]
which holds for a constant $C_p$ independent on $\epsilon \in [0,1)$, we find that
\begin{align*}
\int_{M}\left|\left(\ep^2+|\nabla \vps|^2\right)^{\frac{p-2}{4}}\nabla \vps-|\nabla v|^{\frac{p-2}{2}}\nabla v\right|^2\,dV_M
&\leq C \int_{M} \left(\ep^2+|\nabla \vps|^2+ |\nabla v|^2\right)^{\frac{p-2}{2}}|\nabla \vps - \nabla v|^2\,dV_M
\\& \quad + C \int_{M} \left|\left(\ep^2+|\nabla v|^2\right)^{\frac{p-2}{4}}\nabla v-|\nabla v|^{\frac{p-2}{2}}\nabla v\right|^2\,dV_M.
\end{align*}
We now combine the above estimates for terms \eqref{lem92-aux1} and \eqref{lem92-aux2}. For small enough $r$ we hence have the following estimate for the integral \eqref{lem92-aux1}
\begin{align}
&- \int_{M} \left(\left(\ep^2+|\nabla \vps|^2\right)^{\frac{p-2}{2}} A'(\vps)(\nabla \vps,\nabla \vps)-|\nabla v|^{p-2} A'(v)(\nabla v,\nabla v)\right)\cdot (\vps-v) \,dV_M \nonumber \\& \leq \frac{1}{2} \int_{M} \left(\ep^2+|\nabla \vps|^2+|\nabla v|^2\right)^{\frac{p-2}{2}}\,\left|\nabla \vps-\nabla v\right|^2\,dV_M + \frac{1}{2} \int_{M} \left|\left(\ep^2+|\nabla v|^2\right)^{\frac{p-2}{4}}\nabla v-|\nabla v|^{\frac{p-2}{2}}\nabla v\right|^2\,dV_M.
\label{lem92-48estim}
\end{align}
{\bf Estimates for integral \eqref{lem92-aux2}:}
The only estimate we need here is what we get from an application of H\"older's inequality
\begin{align}
&\int_{M} \left||\nabla v|^{p-2}-\left(\ep^2+|\nabla v|^2\right)^{\frac{p-2}{2}}\right||\nabla v||\nabla \vps-\nabla v| \,dV_M
\nonumber \\&\leq \| \nabla \vps - \nabla v \|_{L^p(M)} \left(\int_M \left||\nabla v|^{p-2}-\left(\ep^2+|\nabla v|^2\right)^{\frac{p-2}{2}}\right|^{\frac{p}{p-1}}|\nabla v|^{\frac{p}{p-1}}\, dV_M\right)^{\frac{p-1}{p}}. \label{lem92-49estim}
\end{align}
Set $I_{\ep}:= \int_{M} \left(\ep^2+|\nabla \vps|^2+|\nabla v|^2\right)^{\frac{p-2}{2}}\,\left|\nabla \vps-\nabla v\right|^2\,dV_M$. By \eqref{lem-conv1} we know that $\| \nabla \vps - \nabla v \|_{L^p(M)}^p \leq C I_{\ep}$, where $C$ is independent of $\ep$ by the discussion following \eqref{lem-1st-rhs}. Hence, it will be enough to show that $I_{\ep} \to 0$ as $\ep \to 0$. Combining estimates \eqref{lem92-48estim} and \eqref{lem92-49estim} gives together with the estimates starting at \eqref{lemm92-ineq1} that
\begin{align}I_{\ep} &\leq \frac12 I_{\ep} + I_{\ep}^{\frac{1}{p}}  \left(\int_M \left||\nabla v|^{p-2}-\left(\ep^2+|\nabla v|^2\right)^{\frac{p-2}{2}}\right|^{\frac{p}{p-1}}|\nabla v|^{\frac{p}{p-1}}\, dV_M\right)^{\frac{p-1}{p}}\nonumber\\&\quad + \frac{1}{2} \int_{M} \left|\left(\ep^2+|\nabla v|^2\right)^{\frac{p-2}{4}}\nabla v-|\nabla v|^{\frac{p-2}{2}}\nabla v\right|^2\,dV_M. \label{Iep_estimate}\end{align}
By the Lebesgue Dominated Convergence theorem we have that
\begin{align*}
\int_M \left||\nabla v|^{p-2}-\left(\ep^2+|\nabla v|^2\right)^{\frac{p-2}{2}}\right|^{\frac{p}{p-1}}|\nabla v|^{\frac{p}{p-1}}\, dV_M &\to \ 0 \quad \text{ and }\\
\int_{M} \left|\left(\ep^2+|\nabla v|^2\right)^{\frac{p-2}{4}}\nabla v-|\nabla v|^{\frac{p-2}{2}}\nabla v\right|^2\,dV_M &\to\  0
\end{align*}
as $\ep \to 0$. Therefore, \eqref{Iep_estimate} implies that $I_{\ep}\to 0$, as $\ep \to 0$. This proves the claim of the lemma.
\end{proof}

As a result of Lemma \ref{lem92} we now know that the differentials $\nabla \vps$ converge to $\nabla v$ in $L^p$-norm. As a consequence, we also know that the differentials $D \ups$ must converge to $Du$ in $L^p$ norm. Our next aim will be to strengthen this to local uniform convergence. For this we will not use the Nash embedding interpretation for the derivatives. Instead, we recall that the mappings $u^\ep$ arise from solutions of the following system of equations (cf. \eqref{system} and the presentation following it):
\begin{equation*}
 \delta \left(\left(\ep^2+|\dups|^2\right)^{\frac{p-2}{2}}Du^{\ep}\right)=0,
\end{equation*}
where $|\dups|^2=|\utet|^2+|\utetb|^2$, cf. \eqref{eq-Du-norm} and \eqref{def-lam}. We use here this notation in order to appeal to Hamburger~\cite{ham}. Then, we have the following result.

\begin{lem}\label{lem93}
 Let $\ep\geq 0$. Denote by
 \[
  V_{\ep}:=\left(\ep^2+|\dups|^2\right)^{\frac{p-2}{4}}\dups.
 \]
 Then $V_{\ep}\in W^{1,2}(B_R,\Lambda_1)$ for any (geodesic) ball $B_R\Subset M$ and
 \begin{equation}\label{lem93-ineq1}
  \int_{B_{\frac{R}{2}}}|DV_{\ep}|^2\,dV_M\leq \frac{c}{R^2}\int_{B_R}|V_{\ep}-(V_{\ep})_{B_R}|^2\,dV_M+c\int_{B_R} (\ep^2+|\dups|^2)^{\frac{p}{2}}\,dV_M.
  \end{equation}
  Moreover, for any compact set $K\subset \Om$ and for all $\ep\geq 0$ it holds that
  \begin{equation}\label{lem93-uni-est}
   \|DV_{\ep}\|_{L^2(K)}\leq C(p, K, M, \|Du_0\|_{L^p(M)}),
  \end{equation}
  where $u_0$ is a given Sobolev map such that $u^{\ep}-u_0\in \Sobzero(M, \C)$.
\end{lem}
\begin{proof}
 In order to show the first part of the theorem we follow the discussion in the proof of Theorem 6.2 in Hamburger~\cite{ham}. The proof of \eqref{lem93-ineq1} therein is reduced to the corresponding part of the proof of Theorem 6.1 which in turn  reduces to the proof of parts 1 and 2 in \cite[Theorem 4.1]{ham}. In particular, we may apply the Caccioppoli inequality on pg. 29 in \cite{ham} in our case with the following notation, cf. (2.8) and (2.9) in \cite{ham}:
  \begin{align*}
  H(\omega):=\left(\ep^2+|\utet|^2+|\utetb|^2\right)^{\frac{p}{2}}, \qquad
  V_{\ep}(\omega_0):=(V_{\ep})_{x_0, R}=(V_{\ep})_{B_R}=\vint_{B_R}V_{\ep}dV_M.
  \end{align*}
 In consequence one obtains estimate \eqref{lem93-ineq1}:
  \begin{equation*}
  \int_{B_{R/2}}|DV_{\ep}|^2\,dV_M \leq \frac{c}{R^2}\int_{B_R}|V_{\ep}-(V_{\ep})_{B_R}|^2\,dV_M+c\int_{B_R} (\ep^2+|\dups|^2)^{\frac{p}{2}}\,dV_M.
 \end{equation*}
The constant $c$ does not depend on $\ep$ and does not explode for any value of $p>2$. Moreover, let us observe that since our definition of norm~\eqref{eq-Du-norm} appears in definitions of $V_{\ep}$ and $H$ above, it influences formula (4.10) on pg. 28 in \cite{ham} and computations following it on pg. 28-29. Indeed, the factor $\rho(\ups)/\sigma$ can be included into coefficients $G^{IJ}$ and $G^{KL}$, cf. \cite[(4.10)]{ham} as a multiplication factor. However, since both metrics $\sigma$ and $\rho$ are assumed to be smooth and bounded, there presence in computations at \cite[(4.11)-(4.14)]{ham} manifests in changing only the constants in Hamburger's estimates.

As a result we have that
 \begin{align}
  \int_{B_{R/2}}|DV_{\ep}|^2\,dV_M &\leq \frac{c}{R^2}\int_{B_R}|V_{\ep}-(V_{\ep})_{B_R}|^2\,dV_M+c\int_{B_R} (\ep^2+|\dups|^2)^{\frac{p}{2}}\,dV_M \nonumber \\
  & \leq \frac{4c}{R^2}\int_{B_R}|V_{\ep}|^2\,dV_M+c\int_{B_R} (\ep^2+|\dups|^2)^{\frac{p}{2}}\,dV_M \nonumber \\
  & \leq \frac{4c}{R^2}\int_{B_R} \left(\ep^2+|\dups|^2\right)^{\frac{p-2}{2}}|Du^{\ep}|^2\,dV_M+ c\int_{B_R} (\ep^2+|\dups|^2)^{\frac{p}{2}}\,dV_M. \label{DVep-est0}
\end{align}

In order to complete the above estimate, let us notice that by Caccioppoli-type inequality \eqref{CacU-1} we have for all $\ep>0$ that
\[
  \int_{M} |\dups|^p \,dV_M \leq c\int_{M} \left(\ep^2+|Du_0|^2\right)^{\frac{p-2}{2}} |Du_0|^2\,dV_M.
\]
By combining this estimate with \eqref{CacU-1} and applying these at \eqref{DVep-est0}, we obtain
\begin{align}
 &\int_{B_{R/2}}|DV_{\ep}|^2\,dV_M \nonumber \\
 & \leq \frac{c}{R^2} \int_{M}\left(\ep^2+|Du_0|^2\right)^{\frac{p-2}{2}} |Du_0|^2\,dV_M+c2^p\ep^p{\rm Vol}_M|B_R|+
  2^pc\int_{M} \left(\ep^2+|Du_0|^2\right)^{\frac{p-2}{2}} |Du_0|^2\,dV_M \nonumber \\
 &\leq c\left(1+\frac{1}{R^2}\right) \int_{M}\left(\ep^2+|Du_0|^2\right)^{\frac{p-2}{2}} |Du_0|^2\,dV_M +c\ep^p{\rm Vol}_M|B_R|.   \label{DVep-unif}
\end{align}
where $u_0$ corresponds to the fixed boundary data. Thus, we may bound $\|DV_{\ep}\|_{L^p(B_R)}$ uniformly by a constant depending only on $p$, $R$, norm $\|Du_0\|_{L^p(M, \R^{2\times 2})}$ and on the geometry of $M$.

For an arbitrary compact set $K\subset \Om$ we cover it by balls $B_R$ with $B_R\Subset \Om$ and $R>\frac12$ and obtain assertion \eqref{lem93-uni-est}.
\end{proof}

We are now in a position to state and prove the following crucial result of this section.
\begin{cor}\label{cor-det-conv}
 Under the above notation it holds that $Du^\ep$ converges uniformly to $Du$ on compacta for $\ep\to 0$.
 Furthermore, it holds that
 \begin{equation*}
 J_{V_{\ep}}\to J_V,\quad \hbox{for } \ep \to 0
 \end{equation*}
  uniformly on compact subsets $K\subset M$.
\end{cor}

\begin{proof}
 By estimate \eqref{DVep-unif} and the discussion following it we have the uniform in $\ep$ estimates on the Sobolev norms $\|V_\ep\|=\|V_\ep\|_{L^{p}(K)}+\|DV_\ep\|_{L^{p}(K)}$ for any compact $K\subset M$. Hence, by the Sobolev embedding we also have uniform in $\ep$ H\"older estimates for $V_{\ep}$ on $K$. This implies the equicontinuity of the family $\{V_{\ep}\}_{\ep\in(0,1)}$. We combine this observation with a consequence of the convergence result in Lemma~\ref{lem92}:
 \[
  V_{\ep}\longrightarrow |Du|^{p-2} Du\qquad\hbox{ for }\ep\to 0\quad \hbox{ a.e. in } M
 \]
 and infer from the Arzela--Ascoli theorem the uniform convergence of $V_{\ep}$ on any compact $K\subset M$.
\end{proof}

\section{The proof of Theorem~\ref{MainThm}}\label{Sec10}

 \noindent We are in a position to collect the results of the work and to complete the proof of the Rad\'o--Knseser--Chocquet theorem.

\begin{proof} We may assume without loss of generality that $\phi_0$ is positively oriented.

By Proposition \ref{prop67-sec6}, we know that $J_u > 0$ near the boundary $\bdy M$. Thus we can reduce the problem to a subdomain of $M$ with smooth boundary and diffeomorphic boundary data.

 By Propositions~\ref{existence}, \ref{sect-exist-uni-Prop4} and Theorem~\ref{app-thm-uniq} we obtain the existence and the uniqueness of mappings $\ups$ solving the $\ep$-perturbed $p$-harmonic Dirichlet problem and thus can apply  the homotopy argument in Theorem~\ref{thm-jacob}. By Proposition~\ref{Prop-2} we know that each of the Jacobians $J_{\ups}$ of $\ups$ is positive on the boundary and thus obtain by the homotopy argument that all $J_{\ups}$ are positive everywhere.

The computations in the previous section, in particular Corollary~\ref{cor-det-conv}, show that the Jacobians $J_{\ups}$ converge uniformly on compact subsets to $J_u$. Since the Jacobians $J_{\ups}$ are positive, Corollary~\ref{MinPrin} applies and we obtain a minimum principle for the respective expressions $T$ whenever $\ep > 0$. By the uniform convergence, the minimum principle must also hold when $\ep = 0$.

Since the Jacobian of $u$ does not vanish close to the boundary, we may finally conclude that $J_u$ is strictly positive  everywhere.

Therefore, we have shown that $u: M\to N'$ is a local $C^{1,\alpha}$-diffeomorphism with homeomorphic boundary data $\phi_0$. Hence, $u$ is also a homeomorphism from $M$ to $N$. Moreover, $u$ is actually $C^{1, \alpha}$-diffeomorphism from $M$ to $N$, by the inverse function theorem.
\end{proof}

\appendix

\section{Uniqueness for solutions of $\ep$-perturbed $p$-harmonic systems}

\noindent
The first part of the appendix is devoted to proving Proposition \ref{sect-exist-uni-Prop4} (Theorem~\ref{app-thm-uniq}), the uniqueness result for the solution of the Dirichlet problem for the $\ep$-perturbed $p$-harmonic mappings between Riemannian surfaces under the curvature and the smallness assumptions (A1-3), see the introduction  and also the system of equations \eqref{u-system-weak} for $\lambda:=\lambda^{\ep}$. According to our the best knowledge such a result does not appear in the literature explicitly, except for the case $\ep=0$.

First we show an auxiliary result, namely a Caccioppoli type estimate. For this recall that, as in Section~\ref{sect91}, we denote by $v^{\ep} = i \circ \ups$, where $\ups$ stands for a minimizer of the $\ep$-perturbed $p$-harmonic energy \eqref{ep-energy}, including $\ep=0$, and $i: N' \to \rr^k$ is the Nash embedding. By the discussion in Section~\ref{sect31}, $\vps$ is a weak solution to the system of equations~\eqref{v-system-weak}.

\begin{prop}[cf. Proposition 2.1 in \cite{fare} for $\ep=0$]\label{app-prop}
 Let $\vps \in \Sob_{ex}(M,N')$ be a solution to the $\ep$-perturbed system of equations for some $\ep>0$ such that $\vps(M)\subset i(N) \subset B(x,r)$ for some $x\in i(N')$ and $0<r<c_{N'}$, where $B(x, r)$ denotes an Euclidean ball in $\R^k$ and $c_{N'}$ depends only on $N'$. Then it holds
 \begin{equation*}
  \int_{M} (\ep^2+|\nabla \vps|^2)^{\frac{p-2}{2}}|\nabla \vps|^2|\eta|^2\,dV_M \leq 16r^2 \int_{M} (\ep^2+|\nabla \vps|^2)^{\frac{p-2}{2}}|\nabla \eta|^2\,dV_M\quad \hbox{ for all } \eta\in \Sobzero(M, \R^k).
 \end{equation*}
\end{prop}
\begin{proof}
 Following the steps of the proof for Proposition 2.1 in \cite{fare} we fix a point $x\in i(N')$ and define a test mapping $\phi=\eta^2(\vps-x)$, where $\eta \in \Sobzero(M, \R^k)$. We use this mapping in \eqref{v-system-weak} with $\lambda^{\ep}=(\ep^2+|\nabla \vps|^2)^{\frac{p-2}{2}}$ and obtain
 \begin{align*}
  \int_{M} (\ep^2+|\nabla \vps|^2)^{\frac{p-2}{2}}|\nabla \vps|^2 \eta^2\,dV_M &\leq
 2\int_{M} (\ep^2+|\nabla \vps|^2)^{\frac{p-2}{2}}|\nabla \vps||\nabla \eta||\eta||\vps-x|\,dV_M  \nonumber \\
 &\phantom{AA}+\int_{M} (\ep^2+|\nabla \vps|^2)^{\frac{p-2}{2}}A'(\vps)(\nabla \vps, \nabla \vps)|\vps-x||\eta|^2\,dV_M \nonumber \\
 & \leq 2r\! \left(\int_{M} (\!\ep^2+|\nabla \vps|^2)^{\frac{p-2}{2}}|\nabla \vps|^2|\eta|^2\,dV_M\!\right)^{\!\frac12}\!\! \left(\!\int_{M} (\ep^2+|\nabla \vps|^2)^{\frac{p-2}{2}}|\nabla \eta|^2\,dV_M\!\right)^{\!\frac12}\nonumber \\
 & \phantom{AA}+C_{N'}r\int_{M} (\ep^2+|\nabla \vps|^2)^{\frac{p-2}{2}}|\nabla \vps|^2|\eta|^2\,dV_M,
 \end{align*}
where above we appeal to the smallness assumption on the image of $\ups$ and the H\"older inequality. The bound $C_{N'}$ comes from  the second fundamental form $A'$ and it depends only on the geometry of $N'$ as noted in Section \ref{sect31}.

Let $r$ be such that $C_{N'}r<\frac12$, then upon dividing the both sides by the appropriate factor integral we get
\begin{equation*}
 (1-C_{N'}r)^2\int_{M} (\ep^2+|\nabla \vps|^2)^{\frac{p-2}{2}}|\nabla \vps|^2|\eta|^2\,dV_M
 \leq 4r^2 \int_{M} (\ep^2+|\nabla \vps|^2)^{\frac{p-2}{2}}|\nabla\eta|^2\,dV_M.
\end{equation*}
 From this the assertion follows immediately.
\end{proof}

In the proof of the next uniqueness result we will need the following well-known inequalities, see e.g.  Mingione~\cite{ming}. Let $p\geq 2$. Then for any $\ep>0$ and and vectors $X,Y\in \R^k$ it holds:
\begin{align}
 &\left ( (\ep^2+|X|^2)^{\frac{p-2}{2}}X-(\ep^2+|Y|^2)^{\frac{p-2}{2}}Y \right)\cdot(X-Y) \geq C(p)(\ep^2+|X|^2+|Y|^2)^{\frac{p-2}{2}}|X-Y|^2 \label{phs-est-aux1} \\
 &\left|(\ep^2+|X|^2)^{q}X-(\ep^2+|Y|^2)^{q}Y\right|\leq C(q) \left((\ep^2+|X|^2)^{q}+(\ep^2+|Y|^2)^{q}\right) |X-Y|,\quad q\geq 0. \label{phs-est-aux2}
 \end{align}

For the proof of~\eqref{phs-est-aux2} we define $F:\R^k\to \R^k$  as  follows: if $X \in \R^k$, then set $F(X):=(\ep^2+|X|^2)^{q}X$. We apply the mean-value theorem and find that for any given $X,Y\in \R^k$ it holds that
$$
\aligned
|F(X) - F(Y)| &\leq \sup_{t\in[0,1]}\|DF(tX+(1-t)Y)\|\,|X - Y|\leq \sup_{t\in[0,1]} \left(C_1(q)\left(\ep^2+\left|t|X|+(1-t)|Y|\right|^2\right)^{q}\right)|X - Y| \\
&\leq \left(C_1(q)\left(\ep^2+|X|^2+|Y|^2\right)^{q}\right)|X - Y| \leq C_2(q)\left((\ep^2+|X|^2)^{q}+ (\ep^2+|Y|^2)^{q}\right)|X - Y|,
\endaligned
$$
thus resulting in \eqref{phs-est-aux2}.

We are in a position to state an $\ep$-version of Theorem 1.1 in \cite{fare}.  In the proof below we follow the steps for Theorem 1.1 in \cite{fare} proven for the case $\ep=0$ only. However, since the proof of uniqueness for $\ep$-perturbed system of equations on Riemannian manifolds is, to our best knowledge, not available in the literature, we present it here for the readers convenience and completeness of the work.

\begin{theorem}[cf. Theorem 1.1 in \cite{fare} for $\ep=0$]\label{app-thm-uniq}
Suppose that $v^{\ep}_1, v^{\ep}_2$ are solutions to the $\ep$-perturbed $p$-harmonic system of equations \eqref{v-system-weak} both satisfying the smallness condition for a sufficiently small $r_0=r_0(N', p)>0$ with respect to a geodesic ball $B(x,r_0)\subset N'$, that is, $v^{\ep}_1(M) \subset i(B(x, r_0))$ and $v^{\ep}_2(M) \subset i(B(x, r_0))$. If  $v^{\ep}_1=v^{\ep}_2$ on $\partial M$, then $v^{\ep}_1=v^{\ep}_2$ in $M$.
\end{theorem}

\begin{proof}
 Let $v^{\ep}_1$ and $v^{\ep}_2$ be two solutions of the perturbed system \eqref{v-system-weak}
 as in the assumptions of the theorem. Then, $\phi:=v^{\ep}_1-v^{\ep}_2$ is a bounded test mapping in $\Sobzero(M, \R^k)$. We apply $\phi$ in \eqref{v-system-weak} for $v^{\ep}_1$ and $v^{\ep}_2$ and subtract both equations from each other. We obtain
 \begin{align}\label{app-thm-1}
  -&\int_{M} \left ( (\ep^2+|\nabla v^{\ep}_1|^2)^{\frac{p-2}{2}}\nabla v^{\ep}_1 - (\ep^2+|\nabla v^{\ep}_2|^2)^{\frac{p-2}{2}})\nabla v^{\ep}_2\right)\cdot (\nabla v^{\ep}_1-\nabla v^{\ep}_2) \,dV_M \nonumber \\
  &=\int_{M} \left( (\ep^2+|\nabla v^{\ep}_1|^2)^{\frac{p-2}{2}} A'(v^{\ep}_1)(\nabla v^{\ep}_1,\nabla v^{\ep}_1)-(\ep^2+|\nabla v^{\ep}_2|^2)^{\frac{p-2}{2}}A'(v^{\ep}_2)(\nabla v^{\ep}_2,\nabla v^{\ep}_2)\right)\cdot( v^{\ep}_1-v^{\ep}_2)\,dV_M.
 \end{align}
 Upon applying estimate \eqref{phs-est-aux1} to the left-hand side above one gets:
 \begin{align}
& C(p)\int_{M} (\ep^2+|\nabla v^{\ep}_1|^2+|\nabla v^{\ep}_2|^2)^\frac{p-2}{2}|\nabla v^{\ep}_1-\nabla v^{\ep}_2|^2\,dV_M \nonumber \\
& \phantom{AA}\leq\int_{M}  \left ( (\ep^2+|\nabla v^{\ep}_1|^2)^{\frac{p-2}{2}}\nabla v^{\ep}_1 - (\ep^2+|\nabla v^{\ep}_2|^2)^{\frac{p-2}{2}}\nabla v^{\ep}_2\right)\cdot( \nabla v^{\ep}_1-\nabla v^{\ep}_2) \,dV_M. \label{app-thm-2}
 \end{align}
  Next, we recall the following estimate from Lemma 2.2 in \cite{fare} holding for any point $y,z\in i(N')$ and any vectors $Y\in T_yi(N'), Z\in T_z i(N')$ with a constant $C=C(i, N')$, cf. discussion following~\eqref{app-A-est}:
\begin{equation*}
|A'(y)(Y,Y)-A'(z)(Z,Z)|\leq C(|Y|^2+|Z|^2)|y-z|+C(|Y|+|Z|)|Y-Z|.
\end{equation*}
We apply observation  \eqref{app-A-est} for
\[
y=u^{\ep},\quad z=v^{\ep}_2, \quad Y=\left(\ep^2+|\nabla v^{\ep}_1|^2\right)^{\frac{p-2}{4}}\nabla v^{\ep}_1, \quad Z=\left(\ep^2+|\nabla v^{\ep}_2|^2\right)^{\frac{p-2}{4}}\nabla v^{\ep}_2.
\]
In fact, we abuse here slightly the notation since, technically, we apply  \eqref{app-A-est}  to each of the $k$ gradients of the component functions of map $v^{\ep}_1$ ($v^{\ep}_2$, respectively) and then use the Cauchy-Schwarz inequality, cf. similar discussion on pg. 267 in \cite{fare}. As a consequence we arrive at a counterpart of formula (2.15) in \cite{fare}:
\begin{align}\label{app-thm-3}
&\left|(\ep^2+|\nabla v^{\ep}_1|^2)^{\frac{p-2}{2}} A'(v^{\ep}_1)(\nabla v^{\ep}_1,\nabla v^{\ep}_1)-(\ep^2+|\nabla v^{\ep}_2|^2)^{\frac{p-2}{2}}A'(v^{\ep}_2)(\nabla v^{\ep}_2,\nabla v^{\ep}_2)\right| \nonumber \\
&\leq C(p)\left(\left(\ep^2+|\nabla v^{\ep}_1|^2\right)^{\frac{p-2}{2}}|\nabla v^{\ep}_1|^2+\left(\ep^2+|\nabla v^{\ep}_2|^2\right)^{\frac{p-2}{2}}|\nabla v^{\ep}_2|^2\right)|v^{\ep}_1-v^{\ep}_2| \nonumber \\
&+C(p)\left(\left(\ep^2+|\nabla v^{\ep}_1|^2\right)^{\frac{p-2}{4}}|\nabla v^{\ep}_1|+\left(\ep^2+|\nabla v^{\ep}_2|^2\right)^{\frac{p-2}{4}}|\nabla v^{\ep}_2|\right)\left|\left(\ep^2+|\nabla v^{\ep}_1|^2\right)^{\frac{p-2}{4}}\nabla v^{\ep}_1-\left(\ep^2+|\nabla v^{\ep}_2|^2\right)^{\frac{p-2}{4}}\nabla v^{\ep}_2\right|.
\end{align}
We apply estimate \eqref{phs-est-aux2} and obtain:
\begin{align*}
\left|(\ep^2+|\nabla v^{\ep}_1|^2)^{\frac{p-2}{4}}\nabla v^{\ep}_1-(\ep^2+|\nabla v^{\ep}_2|^2)^{\frac{p-2}{4}}\nabla v^{\ep}_2\right|&\leq C(p) \left((\ep^2+|\nabla v^{\ep}_1|^2)^{\frac{p-2}{4}}+(\ep^2+|\nabla v^{\ep}_2|^2)^{\frac{p-2}{4}}\right) |\nabla v^{\ep}_1-\nabla v^{\ep}_2|,
\end{align*}
which applied at \eqref{app-thm-3} gives us the following estimate:
\begin{align}\label{app-thm-4}
&\left|(\ep^2+|\nabla v^{\ep}_1|^2)^{\frac{p-2}{2}} A'(v^{\ep}_1)(\nabla v^{\ep}_1,\nabla v^{\ep}_1)-(\ep^2+|\nabla v^{\ep}_2|^2)^{\frac{p-2}{2}}A'(v^{\ep}_2)(\nabla v^{\ep}_2,\nabla v^{\ep}_2)\right| \nonumber \\
&\leq C(p)\left(\left(\ep^2+|\nabla v^{\ep}_1|^2\right)^{\frac{p-2}{2}}|\nabla v^{\ep}_1|^2+\left(\ep^2+|\nabla v^{\ep}_2|^2\right)^{\frac{p-2}{2}}|\nabla v^{\ep}_2|^2\right)|v^{\ep}_1-v^{\ep}_2| \nonumber \\
&+2C(p) \left(\left(\ep^2+|\nabla v^{\ep}_1|^2\right)^{\frac{p-2}{4}}|\nabla v^{\ep}_1|+\left(\ep^2+|\nabla v^{\ep}_2|^2\right)^{\frac{p-2}{4}}|\nabla v^{\ep}_2|\right)\left((\ep^2+|\nabla v^{\ep}_1|^2)^{\frac{p-2}{4}}+(\ep^2+|\nabla v^{\ep}_2|^2)^{\frac{p-2}{4}}\right) |\nabla v^{\ep}_1-\nabla v^{\ep}_2|.
\end{align}

Notice that if $r_0$ is small enough, then a geodesic ball $B(x, r_0)\subset N'$\, is, under the Nash embedding $i$, contained in a ball $B(i(x),r)\subset \R^k$, for $0<r<C_{N'}$ satisfying assumptions of Proposition~\ref{app-prop}. Thus, upon collecting inequalities \eqref{app-thm-2}, \eqref{app-thm-4}, using them at \eqref{app-thm-1}  and applying the Cauchy-Schwarz inequality together with three times Proposition \ref{app-prop} for $\eta=v^{\ep}_1-v^{\ep}_2\in \Sobzero(M, \R^k)$ we get
\begin{align}
&\int_{M}  (\ep^2+|\nabla v^{\ep}_1|^2+|\nabla v^{\ep}_2|^2)^\frac{p-2}{2}|\nabla v^{\ep}_1-\nabla v^{\ep}_2|^2\,dV_M \nonumber \\
&\leq C(p)\int_{M} \left(\left(\ep^2+|\nabla v^{\ep}_1|^2\right)^{\frac{p-2}{2}}|\nabla v^{\ep}_1|^2+\left(\ep^2+|\nabla v^{\ep}_2|^2\right)^{\frac{p-2}{2}}|\nabla v^{\ep}_2|^2\right)|v^{\ep}_1-v^{\ep}_2|^2\,dV_M \nonumber \\
&\phantom{AA}+2C(p)\int_{M} \left(\left(\ep^2+|\nabla v^{\ep}_1|^2\right)^{\frac{p-2}{4}}|\nabla v^{\ep}_1|+\left(\ep^2+|\nabla v^{\ep}_2|^2\right)^{\frac{p-2}{4}}|\nabla v^{\ep}_2|\right)|v^{\ep}_1-v^{\ep}_2| \nonumber \\
&\phantom{AAAA}\times\left((\ep^2+|\nabla v^{\ep}_1|^2)^{\frac{p-2}{4}}+(\ep^2+|\nabla v^{\ep}_2|^2)^{\frac{p-2}{4}}\right) |\nabla v^{\ep}_1-\nabla v^{\ep}_2| \,dV_M \nonumber \\
&\leq 16r^2C(p)\int_{M} \left(\ep^2+|\nabla v^{\ep}_1|^2+|\nabla v^{\ep}_2|^2\right)^{\frac{p-2}{2}}|\nabla v^{\ep}_1-\nabla v^{\ep}_2|^2\,dV_M \nonumber \\
&\phantom{AA}+8\sqrt{2}rC(p)\left(\int_{M} \left(\left(\ep^2+|\nabla v^{\ep}_1|^2\right)^{\frac{p-2}{2}}+\left(\ep^2+|\nabla v^{\ep}_2|^2\right)^{\frac{p-2}{2}}\right)|\nabla v^{\ep}_1-\nabla v^{\ep}_2|^2 \,dV_M\right)^{\frac12} \nonumber \\
&\phantom{AAAA}\times\left(\int_{M} \left(\ep^2+|\nabla v^{\ep}_1|^2+|\nabla v^{\ep}_2|^2\right)^{\frac{p-2}{2}}|\nabla v^{\ep}_1-\nabla v^{\ep}_2|^2 \,dV_M\right)^{\frac12}. \label{thmA1-factor}
\end{align}
We divide the both sides of the resulting inequality by the factor~\eqref{thmA1-factor} and notice that both $\left(\ep^2+|\nabla v^{\ep}_i|^2\right)^{\frac{p-2}{2}}$, for $i=1,2$, are trivially bounded by $ (\ep^2+|\nabla v^{\ep}_1|^2+|\nabla v^{\ep}_2|^2)^\frac{p-2}{2}$. Hence, we obtain:
\begin{align*}
 &\left(\int_{M}  (\ep^2+|\nabla v^{\ep}_1|^2+|\nabla v^{\ep}_2|^2)^\frac{p-2}{2}|\nabla v^{\ep}_1-\nabla v^{\ep}_2|^2\,dV_M\right)^{\frac12}\\
& \leq C(p)(16r^2+8\sqrt{2}\cdot\sqrt{2}r)\left(\int_{M} \left(\ep^2+|\nabla v^{\ep}_1|^2+|\nabla v^{\ep}_2|^2\right)^{\frac{p-2}{2}}|\nabla v^{\ep}_1-\nabla v^{\ep}_2|^2\,dV_M\right)^{\frac12}. 
\end{align*}
From this, upon squaring the both sides of the last inequality we get the following estimate:
\begin{align*}
\int_{M}(\ep^2+|\nabla v^{\ep}_1|^2+|\nabla v^{\ep}_2|^2)^{\frac{p-2}{2}} |\nabla v^{\ep}_1-\nabla v^{\ep}_2|^2\leq C^2(p)(16r^2+16r)^2\int_{M} \left(\ep^2+|\nabla v^{\ep}_1|^2+|\nabla v^{\ep}_2|^2\right)^{\frac{p-2}{2}}|\nabla v^{\ep}_1-\nabla v^{\ep}_2|^2\,dV_M.
\end{align*}
Therefore, if $r$ is small enough, then $\nabla v^{\ep}_1=\nabla v^{\ep}_2$ a.e. in $M$ which implies that $v^{\ep}_1=v^{\ep}_2$ as these mappings agree on $\partial M$.
\end{proof}

\section{Construction of a local contraction in the proof of Proposition~\ref{weakmaxprin}}

\noindent
In this section we suppose that $(M,g)$ is a Riemannian manifold diffeomorphic to the unit sphere. We aim to prove the following.

\begin{theorem}\label{contraction}
Let $B_r := B(p,r)$ be a geodesic ball on the manifold $M$. If $r = r_M > 0$ is a small enough radius, then there exists a Lipschitz map $\Psi:M \to M$ which is the identity map on $\bar{B_r}$ and a local contraction from $M \setminus \bar{B_r}$ to $B_r$. In terms of the derivative, $|D\Psi| = 1$ on $\bar{B_r}$ and $|D\Psi| < 1$ outside $\bar{B_r}$.
\end{theorem}
\noindent In the proof we will need the following auxiliary results.
\begin{lem}\label{triangleLemma} Let $T$ denote a geodesic triangle contained in a ball $B(p,r)$ on $M$. Denote its angles by $\theta_1,\theta_2$ and $\theta_3$. If $r$ is small enough (only depending on $M$), then the sum of any two of these angles is less than $\pi$.
\end{lem}
\begin{proof}
Without loss of generality, we aim to show that
\begin{equation}\label{angleClaim}
\theta_2 + \theta_3 < \pi.
\end{equation}
Using the Gauss--Bonnet theorem, we find the formula
\begin{equation}\label{angleEstimate}
\pi - \theta_2 - \theta_3 = \theta_1 - \int_{T} K(x) dV_M(x),
\end{equation}
where $K(x)$ denotes the pointwise Gauss curvature of $M$. Let $K_0$ be an upper bound for $K(x)$. The triangle $T$ is contained in the ball $B(p,r)$, hence it has two sides with length less than $r$ and angle $\theta_1$ between them. The area of a flat triangle with these sides and angle is $\frac12 r^2 \sin \theta_1$. Since the exponential map is a local diffeomorphism and the preimage of $T$ under $\exp_p$ is contained in such a flat triangle, we get the estimate
\[\left|\int_{T} K(x) dV_M(x)\right| \leq K_0 \int_T dV_M(x) \leq C r^2 \sin \theta_1 \leq C r^2 \theta_1.\]
Choose now $r$ small enough so that $C r^2 < 1$. Then by \eqref{angleEstimate} we get equation \eqref{angleClaim} as wanted.
\end{proof}
\begin{lem}\label{oneillformula}
Let $\gamma : (-\ep,\ep) \times [a,b] \to M$ be a smooth family of curves $\gamma_s(t) = \gamma(s,t)$. Then
\begin{equation}\label{oneillequation}\frac{d}{ds} \biggr\rvert_{s=0} \ \frac{1}{2}\int_a^b |\dot{\gamma_s}|^2 dt = g(V,\dot{\gamma_0})\biggr\rvert_a^b - \int_a^b g(V,\ddot{\gamma}),
\end{equation}
where $V = \partial_s \gamma \rvert_{s = 0}$ is the variation field.
\end{lem}
\begin{proof} This formula follows from O'Neill \cite{one}, Chapter 10, Proposition 39, page 289.
\end{proof}

\begin{lem}\label{tangentlemma}
Let $r > 0$ be a small enough radius and $B_r := B(p,r)$ be a geodesic ball on the manifold $M$. Let $\gamma$ denote a geodesic passing through a point $p_0 \in \partial B_r$ and passing in the same direction as the boundary $\partial B_r$ at $p_0$. By the latter condition we mean that $\dot{\gamma}(p_0)$ is equal, up to multiplication by $\pm 1$, to the tangent vector of $\partial B_r$ at $p_0$, see the picture below. Then the intersection of $B_r$ and $\gamma$ contains only the point $p_0$.
\end{lem}
\begin{proof}
Let $\epsilon > 0$ be a small positive number. We suppose that $\gamma(0) = p_0$, and aim to first prove that the geodesic segment $\gamma((-\ep,\ep))$ does not contain other points of $B_r$ except $p_0$. To this end, let $\phi_s$ denote the geodesic passing through $p$ (the center of $B_r$) and the point $\gamma(s)$. This way we obtain a family of geodesics $\phi_s : (-\ep,\ep) \times [0,1] \to M$. By Lemma \ref{oneillformula} applied at time $s_0\in (-\ep,\ep)$ in place of $0$, we obtain the formula
\[\frac{d}{ds} \biggr\rvert_{s=s_0} \ \frac{1}{2}\int_0^1 |\dot{\phi_s}|^2 dt = g(\dot{\gamma}(s_0),\dot{\phi_{s_0}}(1)).\]
Since $|\dot{\phi_s}|$ is constant and equal to the geodesic distance between points $p$ and $\gamma(s)$ due to $\phi_s$ being parametrized on $[0,1]$, we obtain that
\begin{equation}\label{distanceCalculation}\frac{d}{ds} \biggr\rvert_{s=s_0} \ \frac12 \dist(p,\gamma(s))^2 = g(\dot{\gamma}(s_0),\dot{\phi_{s_0}}(1)).\end{equation}
We wish to show that the minimum of $\dist(p,\gamma(s_0))^2$ is attained at $s_0 = 0$, as this will prove that $\dist(p,\gamma(s_0)) > r$ for every $s_0 \in (-\ep,\ep) \setminus \{0\}$. If we let $s_0 = 0$, then the right hand side of the above equation is zero since the Gauss lemma states that the boundary of $B_r$ at $p_0$ is orthogonal to the geodesic passing through $p$ and $p_0$. Furthermore, Lemma \ref{triangleLemma} shows that the angle between $\gamma$ and $\phi_{s_0}$ is less than $\pi/2$ for $s_0 > 0$ and greater than $\pi/2$ for $s_0 < 0$. Hence by inspecting the sign of the derivative we see from \eqref{distanceCalculation} that the minimum for $\dist(p,\gamma(s_0))^2$ is truly attained at $s_0 = 0$.\\
\begin{center}\includegraphics[scale=0.35]{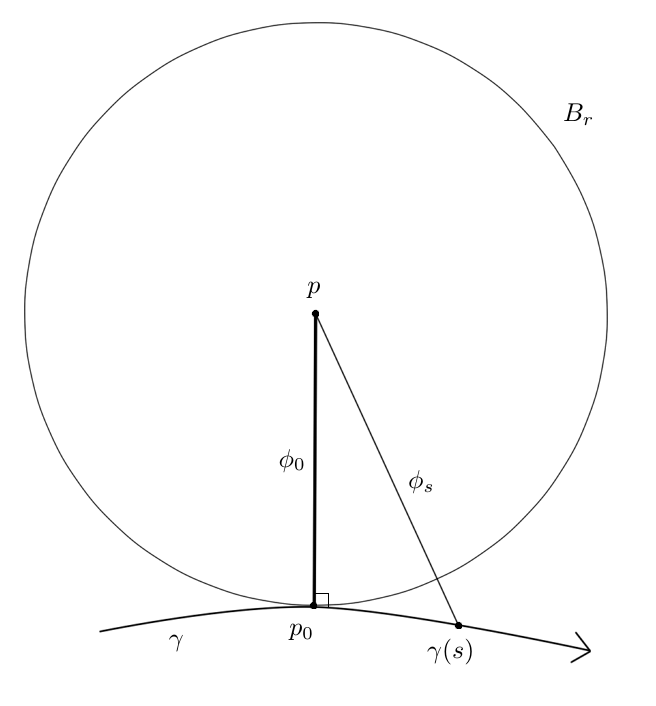}\end{center}
Now by \cite{Sakai}, Theorem 5.3, p. 169 in Chapter IV, the ball $B_r$ is geodesically convex for small enough $r$. Hence if a point $\gamma(s)$ lies in $\bar{B_r}$, then the whole geodesic segment $\gamma([0,s])$ lies in $\bar{B_r}$. If $s \neq 0$ this gives a contradiction to the fact that the segment $\gamma((-\ep,\ep))$ only intersects $\partial B_r$ at $p_0$.
\end{proof}
\begin{lem}\label{geodesiclemma}
Let $a, b: [0,1]\to M$ denote two unit speed geodesics on $M$, such that $a(0) = b(0) = p \in M$. Then the geodesic distance $\dist(a(t),b(t))$ is a strictly increasing function of $t$ for all small enough $t$. More explicitly, the distance remains strictly increasing as long as $a(t), b(t) \subset B(p,r)$ for a fixed small enough radius $r$ only depending on $M$.
\end{lem}
\begin{proof}
 Let $\gamma_t : [0,1] \to M$ denote the geodesic between $a(t)$ and $b(t)$ for any $t\in [0,1]$. Then $|\dot{\gamma_t}|$ is equal to the length of this geodesic since $\gamma_t$ has constant speed. Hence, it holds that
\begin{equation}\label{variationGeodesic}
\dist(a(t),b(t))^2 = \int_0^1 |\dot{\gamma_t}(s)|^2 ds.
\end{equation}
We wish to differentiate this expression at $t = t_0$ and prove that the derivative is positive when $a(t)$ and $b(t)$ are in a small enough ball $B(p,r)$. To calculate the derivative of the expression \eqref{variationGeodesic} with respect to $t$, we apply Lemma \ref{oneillformula} and see that all but the last term on the right hand side of \eqref{oneillequation} vanish due to geodesics having constant speed. Recalling that $g$ denotes the Riemannian metric on $M$, we obtain the following formula for any small enough $t\geq 0$
\[
\left. \frac{d}{d t}\dist(a(t),b(t))^2 \right\rvert_{t = t_0} = g(\dot{a}(t_0),\dot{\gamma}_{t_0}(0)) - g(\dot{b}(t_0),\dot{\gamma}_{t_0}(1)).
\]
Since $a$ and $b$ have unit speed, the expression on the right hand side is nonnegative if the angle between the geodesics $a$ and $\gamma_{t_0}$ is greater or equal to the angle between $b$ and $\gamma_{t_0}$. For this purpose, let us denote by $q_1 = a(t_0)$ and $q_2 = b(t_0)$ the respective intersection points between these geodesics. Consider now the geodesic triangle $T := \Delta p q_1 q_2$ whose angles we denote by $\theta_1, \theta_2$ and $\theta_3$. Our claim is hence equivalent with the inequality
\begin{equation*}
\pi - \theta_2 > \theta_1,
\end{equation*}
which follows directly from Lemma \ref{triangleLemma}.
\end{proof}
We are now in a position to prove the main result of this appendix.
\begin{proof}[Proof of Theorem \ref{contraction}]
Let $B_{r_0}:= B(p,r_0)$ denote a small geodesic ball centered at $p$. Let first $\psi:M\setminus B_{r_0}\to \bar{B_{r_0}}$ denote any Lipschitz map which is the identity map on the boundary. Let $r < r_0$ be a small radius. We now let $\tau : B_{r_0} \to B_r$ be a contraction along geodesics emanating from $p$. In other words, if $\tau_0(v) = \frac{r}{r_0} v$ denotes a contraction map from $T_pM$ onto itself then
\[
\tau(q) = \exp_p\left(\tau_0\left(\exp_p^{-1}(q)\right)\right), \qquad q \in B_{r_0}.
\]
From this representation and the smoothness of the exponential map, its inverse and $\tau_0$ we infer that the Lipschitz constant of $\tau$ may be chosen as small as we wish by choosing $r > 0$ small. Let us now define $\Psi$ as follows
\begin{enumerate}
\item{On $\bar{B_r}$, let $\Psi$ be the identity map.}
\item{On $M \setminus B_{r_0}$, let $\Psi = \tau \circ \psi$. Hence $\Psi$ is a local contraction here for $r$ small enough.}
\item{On $B_{r_0} \setminus \bar{B_r}$, define $\Psi(q)$ as follows. Let $\gamma$ denote the unique geodesic passing through $p$ and $q$. Define $\Psi(q)$ as the intersection point of $\gamma$ and $\partial B_r$.}
\end{enumerate}
The only remaining thing to show is that $\Psi$ is a local contraction on $V := B_{r_0} \setminus \bar{B_r}$. For this, let $q_1$ and $q_2$ be two points in $V$. Let $p_1 = \Psi(q_1)$ and $p_2 = \Psi(q_2)$ denote their images on $\partial B_r$. Our aim is to show that $\dist(p_1,p_2) < \dist(q_1,q_2)$.

Suppose without loss of generality that $\dist(q_1,p) \leq \dist(q_2,p)$. Let $q_3$ denote a point on the geodesic segment between $q_2$ and $p$ such that $\dist(q_3,p) = \dist(q_1,p)$. We claim that $\dist(q_1,q_3) \leq \dist(q_1,q_2)$.

We claim that the angle $\angle q_1 q_3 q_2 > \pi/2$ (by angle between points, we mean the angle between respective geodesics). By the Gauss lemma, the angle between the geodesic circle $\partial B(p,\dist(q_1,p))$ and the geodesic from $p$ to $q_3$ is $\pi/2$. The geodesic from $q_1$ to $q_3$ lies inside the ball $B(p,\dist(q_1,p))$ by convexity, which means that the angle $\angle q_1 q_3 q_2$ must be strictly larger than $\pi/2$.

We now note that the circle with center $q_1$ passing through the point $q_3$ does not intersect with the geodesic segment between $q_3$ and $q_2$. This is due to the fact that $\angle q_1 q_3 q_2 > \pi/2$ and another application of the Gauss lemma and Lemma \ref{tangentlemma}. Hence this circle intersects the segment between $q_1$ and $q_2$ at some point $q$. Now $\dist(q_1,q_3) = \dist(q_1,q) \leq \dist(q_1,q_2)$ as we wanted to prove.

Let now $a(t)$ and $b(t)$ denote unit speed geodesics passing through $p$ and the points $q_1$ and $q_3$ respectively. Hence there are $t_1,t_2$ with $t_1 \leq t_2$ such that $a(t_1) = p_1, b(t_1) = p_2$ and $a(t_2) = q_1, b(t_2) = q_3$.
By Lemma \ref{geodesiclemma}, we find that if $r_0$ is small enough then $\dist(p_1,p_2) < \dist(q_1,q_3)$. Since $\dist(q_1,q_3) \leq \dist(q_1,q_2)$, we have shown that $\Psi$ is a contraction as wanted.
\end{proof}

\end{document}